%% file: main.tex
\documentclass[11pt,a4paper,reqno]{amsart}%article}%
\input{preamble}
\DeclareMathAlphabet{\mathcal}{OMS}{cmsy}{m}{n}
\title[The $G$-Noncommutative minimal model program]{The $G$-Noncommutative minimal model program}

\author{Dongjian Wu}
\address{Department of Mathematics, Graduate School of Science, The University of Osaka, Toy
onaka Osaka, Japan}
\email{wu.dongjian.7cx@osaka-u.ac.jp}
\date{}

\author{Nantao Zhang}
\address{Department of Mathematical Sciences, Tsinghua University, 100084 Beijing, China}
\email{znt21@mails.tsinghua.edu.cn}
\date{}

\keywords{Bridgeland stability conditions, equivariant derived categories, semiorthogonal decompositions, quantum cohomology}
\begin{document}
\setcounter{tocdepth}{1}

%=========================================================

\begin{abstract}
In this paper, we study the $G$-equivariant noncommutative minimal model program ($G$-NMMP), as an equivariant generalization of the framework introduced in \cite{halpernleistner2024noncommutativeminimalmodelprogram}. The aim of this program is to construct quasi-convergent paths in the spaces of Bridgeland stability conditions on derived categories of $G$-equivariant coherent sheaves. For finite groups, we employ induction techniques to construct such paths from the non-equivariant setting. In the setting of algebraic group actions, we introduce the notion of $\bT$-stability conditions to reformulate the proposal, and then we construct quasi-convergent paths for equivariant projective spaces from small quantum cohomology.
\end{abstract}

\maketitle
\tableofcontents
% \setcounter{section}{1}

%\addtocontents{toc}{\setcounter{tocdepth}{1}}

\setlength\parindent{0pt}
\setlength{\parskip}{5pt}

%=========================================================
\section{Introduction}
\subsection{Motivation}
The classical minimal model program (MMP) seeks to simplify an algebraic variety via a sequence of birational transformations, culminating in either a minimal model (with nef canonical class) or a Mori fiber space. In the presence of a group action,
the $G$-equivariant minimal model program ($G$-MMP) carries out this process while preserving the group throughout. In the classical MMP, the two-dimensional case is well-understood: each step corresponds to the contraction of a $(-1)$ curve. The three-dimensional case introduces new phenomena, such as flips \cite{Kollár_Mori_1998}. The two-dimensional $G$-MMP for finite groups is also well developed and has been applied successfully in many instances \cite{Blanc2006LinearisationOF,Dolgachev2009,Tsygankov_2011}. We refer to the comprehensive survey \cite{Prokhorov_2021} for a detailed discussion of $G$-MMP, including applications in higher dimensional birational geometry. 

From a categorical perspective, the classical MMP of a smooth projective variety $X$ is interwined with the structure of the bounded derived category of coherent sheaves $D^b(X)$. Seminal work of Bondal and Orlov \cite{Bondal1995SemiorthogonalDF} established that birational transformations occurring in the MMP-such as divisorial contractions and simple flips-induce semiorthogonal decompositions of $D^b(X)$. When $X$ carries an action of a reductive algebraic group $G$, one considers the $G$-equivariant derived category $D^b_G(X)$, whose objects are $G$-equivariant coherent sheaves. Analogous to the non-equivariant case, $G$-equivariant birational maps can similarly induce semiorthogonal decompositions of $D^b_G(X)$. This provides a categorical framework for studying the $G$-MMP through derived categories.
\begin{sloppypar}
The classical correspondence between birational geometry and semiorthogonal decompositions further motivates the noncommutative minimal model program (NMMP), introduced by Halpern-Leistner \cite{halpernleistner2024noncommutativeminimalmodelprogram}. The NMMP operates directly on triangulated categories using Bridgeland stability conditions \cite{Bridgeland2006SpacesOS}, constructing semiorthogonal decompositions via quasi-convergent paths in the stability manifolds. A key motivation for the NMMP comes from quantum cohomology. For a Fano variety $X$, the quantum product $\star_{\tau}$ deforms the classical cup product via Gromov-Witten invariants, and the associated quantum differential equation governs the flat sections of the quantum connection \cite{IRITANI20091016}. Dubrovin's conjecture and its refinement, the Gamma Conjecture II \cite{10.1215/00127094-3476593}, predict a deep correspondence between the analytic properties of solutions to this equation and exceptional collections in $D^b(X)$. This suggests that stability conditions, whose central charges often resemble period integrals, should be constructible from quantum differential equations. 
\end{sloppypar}

The equivariant counterpart-equivariant Gromov-Witten theory and equivariant quantum differential equations-has been well studied in the literature (see \cite{1994hep.th....5035K,Givental1996EquivariantG,Cox1999MirrorSA,liu2017equivariant,Anderson_Fulton_2023}). It is therefore natural to study relations among the following themes:
\begin{enumerate}
\item semiorthgonal decompositions of $D^b_G(X)$;
\item the equivariant quantum cohomology of $X$ with $G$-action;
\item quasi-convergent paths in the space of stability conditions $\Stab(D^b_{G}(X))$.
\end{enumerate}
 The goal of this note is to develop the $G$-equivariant noncommutative minimal model program ($G$-NMMP), which serves as an equivariant generalization of the NMMP that incorporates group actions systematically.

\subsection{Main proposal}
In \cite{halpernleistner2024noncommutativeminimalmodelprogram}, the author formulated a proposal relating semiorthgonal decompositions, stability conditions and (truncated) quantum differential equations within the framework of the NMMP. In the equivariant setting, it is natural to seek an analogous proposal that incorporates a group action. A key new ingredient in our work is the notion of $\bT$-stability conditions (see \Cref{sec:T-stab}), which generalizes the $q$-stability conditions introduced in \cite{Ikeda_Qiu_2023}. Recall that a $\bT$\emph{-category} $\cD_{\bT}$ consists of a triangulated category $\cD$ with $m$ commuting auto-equivalences
\[
T_i:\cD_{\bT}\to\cD_{\bT}, \quad i=1,\dots,m.
\] 
\begin{definition}[\Cref{def:T-stab}]
A \emph{pre-$\bT$-stability condition} $(\sigma,\s)$ on a $\bT$-category $\cD_{\bT}$ consists of a (Bridgeland) pre-stability condition $\sigma=(Z,\cP)$ on $\cD_{\bT}$ and a tuple $\s=(s_1,\dots,s_m)\in\mathbb C^m$ satisfying 
\begin{equation}
T_i(\sigma)=s_i\cdot\sigma, \text{ for } i=1,\dots,m.
\end{equation}
\end{definition}

Under the \Cref{as:free of K}, the Grothendieck group $K(\cD_{\bT})$ can be endowed with a norm $\|\cdot\|$ which allows us to formulate the $\bT$-support property for pre-$\bT$-stability conditions (see \Cref{def:support property for T-stab}). A pre-$\bT$-stability condition $(\sigma,\s)$ on $\cD_{\bT}$ is called $\bT$-stability condition if it satisfies the $\bT$-support property. We denote by $\bT\Stab_{\s}\cD_{\bT}$ the space of all $\bT$-stability conditions with respect to a fixed tuple $\s\in\mathbb C^m$. \Cref{pro:defomation of T-stab} establishes the deformation property of this space. 

For a smooth projective variety $X$ with an action of a connected reductive algebraic group $G$, the equivariant category $D^b_{G}(X)$ naturally inherits the structure of a $\bT$-category with respect to a maximal torus of $G$ (\Cref{ex:T-sheaf}). Hence, we obtain the associated space of $\bT$-stability conditions, denoted by $\bT\Stab_{\s}(X)\coloneq\bT\Stab_{\s}(D^b_{G}(X))$. Our main proposal is stated as follows:

\begin{proposal}[\Cref{proposal:main}]
  \label{intro:proposal}
Let $\pi:X\to Y$ be a $G$-contraction of a smooth projective variety $X$ of dimension $m$ and $\Phi^G_t\in\End(H^{\ast}_G(X;\mathbb C))$ be a fundamental solution of the  $G$-truncated quantum differential equation \eqref{eq: qde}. Then, for any generic tuple $\s=(s_1,\dots,s_m)\in\mathbb C^m$, there exists a quasi-convergent path $\sigma_t^G=(Z_t^G,\cP_t^G)$ for $t\in[t_0,\infty)$ in $\bT\Stab_{\s}(X)$ whose central charges take the form 
\[
Z_{t}^G(\cE)=\ev_\s\int^{eq}_X\Phi^G_t(v^G(\cE)),
\]
for $\cE\in K^G(X)$, where $\ev_{\s}$ is the evaluation map with respect to $\s$ described in \Cref{subsec:main proposal}.
\end{proposal}

\begin{remark}
The genericity of $\s$ ensures that the evaluation $\ev_\s$ yields well-defined complex numbers. When $G$ is a finite group, the connectedness assumption is unnecessary. Moreover, since the torus maximal $\bT$ is trivial, we have $\bT\Stab_{\s}(X)=\Stab(D^b_G(X))$. In this setting, one may also consider  stability conditions supported on the lattice provided by the Chen-Ruan cohomology ring, together with quantum differential equations arising from orbifold Gromov-Witten theory \cite{ChenRuancohomology, chen2002orbifold}.
\end{remark}

In \cite[Remark 13]{halpernleistner2024noncommutativeminimalmodelprogram}, it is suggested that the quasi-convergent paths for \Cref{intro:proposal} should start from a geometric stability condition-one under which all skyscraper sheaves of points are stable and of the same phase. In the equivariant setting, we introduce the notion of $G$\emph{-point sheaves} and use them to define an analogue of geometric stability conditions for $D^b_G(X)$ (\Cref{def:geometric stab for stack}).

\subsection{Main results}
Building on technique for inducing stability conditions developed in \cite{MacriSukhenduPaolo07,2023arXiv231002917Q,dell2024fusionequivariantstabilityconditionsmorita}, our first result shows how to lift quasi-convergent paths from non-equivariant setting to the equivariant one, thereby producing solutions to \Cref{intro:proposal} for finite group actions.

\begin{theorem}[\Cref{thm: inducing solutions}]
  \label{intro:induce}
Let $X$ be a smooth projective variety with an action of a finite group $G$. Let $\sigma_{t}=(Z_t,\cP_t)$ for $t\in[t_0,\infty)$ be a quasi-convergent path in $\Stab(X)^G$ that solves the non-equivariant \Cref{intro:proposal} via a fundamental solution $\Phi_t$ of \eqref{eq: qde}. Then it induces a quasi-convergent path $\sigma^G_t$ in $\Stab(D^b_G(X))$ that solves \Cref{intro:proposal} with a fundamental solution $\Phi^G_t$ of \eqref{eq: qde} and preserves the spanning condition. Moreover, if $\sigma_{t_0}$ is geometric, then $\sigma_{t_0}^G$ is also geometric.
\end{theorem}

Here, the spanning condition ensures that the asymptotic growth rates of central charges from the quantum differential equation correspond to genuine limit semistable objects (see \Cref{def:spanning condition for finite groups}). As a corollary, we recover the result of \cite{krug2024endomorphismalgebrasequivariantexceptional} on inducing semiorthgonal decompositions by quasi-convergent paths. Recall that in \cite{halpernleistner2024noncommutativeminimalmodelprogram}, Halpern-Leisner verified \Cref{intro:proposal} for the projective line; later work \cite{zuliani2024semiorthogonaldecompositionsprojectivespaces} treated projective spaces, and \cite{karube2024noncommutativemmpblowupsurfaces} studied the NMMP for blowup surfaces. More recently, \cite{KRZ2026} extended the construction of quasi‑convergent lifts to further Fano examples such as Grassmannians, quadrics, and smooth cubic threefolds and fourfolds. By applying \Cref{intro:induce} to these known non-equivariant solutions, we obtain explicit families of quasi-convergent paths that solve the equivariant \Cref{intro:proposal}. Concrete examples are discussed in \Cref{subsec:application of inducing}.

In \Cref{sec:G-NMMP for projective spaces}, we examine the $G$-NMMP for projective spaces equipped with a torus $\bT$-action. Let \(\Omega_{\bT}\) be the complement in \(\bC^{m}\) of the hyperplanes
  \[z_{i} - z_{j} = k, \quad i,j = 1, \cdots, m, i \neq j, k \in \bZ\]
and define
\[
H_{\bT}^{\Omega_{\bT}}(\bP^{m - 1}) = H_{\bT}^{*}(\bP^{m - 1}) \otimes_{H_{\bT}({\rm pt})} \cO_{\Omega_{\bT}},
\]
where $\cO_{\Omega_{\bT}}$ denotes the ring of holomorphic functions on $\Omega_{\bT}$. Our second result is the following: 
\begin{theorem}[\Cref{thm:tncmmp}]
  For any \(\s \in \Omega_{\bT}\), there exists \(R > 0\), such that for every admissible \(\theta\), there exists a quasi-convergent path \(\sigma_{t} = (Z_{t}, \cA_{t})\in\bT\Stab_{\s}(\bP^{m-1})\) for
  \[t = r e^{- 2 \i\pi \theta}, \quad  r \in (R, + \infty)\]
  with \(Z_{t}(\alpha) = \ev_{\s} (\int_{\bP^{m - 1}}^{T} \Phi_{t}({\rm Ch}_{\bT}(\alpha)))\), where \(\Phi_{t}\) is a fundamental solution of equivariant quantum differential equation linearly extended to \({\rm End}(H_{\bT}^{\Omega_{\bT}}(\bP^{m - 1}))\), valid in the sector
  \[\frac{n}{m} - 1 < \phi < \frac{n}{m}, \quad n\in\mathbb Z, \phi\in\mathbb R.\]
\end{theorem}

When $m=3$ (i.e. for $\mathbb P^2$), we show in \Cref{thm:m=3} that the quasi-convergent paths can be chosen to start from a geometric point. The same argument may also apply to any Fano variety \(X\) that satisfies the Gamma conjecture II \cite{10.1215/00127094-3476593}. We expect equivariant parameters do not affect the asymptotic behavior as \(t \to \infty\), and the fundamental solution of the quantum differential equation should again produces a quasi-convergent path corresponding to a semiorthgonal decomposition \[D^{b}_{\bT}(X) = \langle \cE_{1} \otimes \Rep(\bT), \cdots, \cE_{n} \otimes \Rep(\bT) \rangle\] as in the projective-space case.

Finally, in \Cref{sec:relat-with-birat} we explore connections with classcial birational geometry, based on \cite{halpernleistner2024noncommutativeminimalmodelprogram}. For finite groups $G$, we show that under the assumptions of the \Cref{intro:proposal} and \Cref{conjecture:1}, the asscoiated D-equivalence conjecture holds (\Cref{cor:D-equivalence conj}) and one-way imiplication of Dubrovin's conjecture is satisfied (\Cref{pro:Dubrovin's conj}).

\subsection{Contents}
\begin{sloppypar}
In \Cref{sec:pre}, we review basic concepts on stability conditions on triangulated categories, derived categories of equivariant sheaves, and semiorthogonal decompositions. In \Cref{sec:G-NMMP}, we introduce the framework of $G$-NMMP, which extends the classical NMMP by incorporating actions of a reductive algebraic group $G$ on projective varieties and their derived categories. In \Cref{sec:induce paths for finite groups}, we develop techniques for inducing quasi-convergent paths in $\Stab(D^b_G(X))$ in the case of finite groups as required in \Cref{proposal:main} from the NMMP. In \Cref{sec: paths in T-stab}, we introduce the notion of $\bT$-stability conditions on $\bT$-categories and establish the deformation property of spaces of $\bT$-stability conditions. Then, we construct quasi-convergent paths in spaces of $\bT$-stability conditions to fulfill the requirements of \Cref{proposal:main} for equivariant projective spaces from small quantum cohomology. In \Cref{sec:relat-with-birat}, we explain how \Cref{conjecture:1} and \Cref{proposal:main} extend some non-equivariant birational results to the \(G\)-equivariant setting for finite groups.
\end{sloppypar}

\subsection*{Acknowledgement}
D. W. is grateful to Atsushi Takahashi for his support throughout his research. N. Z. would like to thank Will Donovan for support during the work, and Tomohiro Karube and Yukinobu Toda for hospitality and fruitful discussions during author's visit in Kavli, IPMU. We also thank Yiran Cheng, Fabian Haiden, Koshiro Murai and Tianle Mao for helpful discussions. D. W. is supported by JSPS KAKENHI KIBAN(S) 21H04994. N. Z. is supported by Yau Mathematical Science Center, Tsinghua University.

%=========================================================
\section{Preliminaries}
\label{sec:pre}
In this section, we recall basic concepts of stability conditions on triangulated categories, derived categories of equivariant sheaves, and semiorthogonal decompositions, based on \cite{Bondal1995SemiorthogonalDF,Bondal2002DerivedCO,MR2373143,elagin2015equivarianttriangulatedcategories,Cotti2019EquivariantQD,Beckmann2020OnED}.

\subsection{Bridgeland stability conditions}
\label{sec:pre on stab}
Let $\cD$ be a triangulated category and $K(\cD)$ the Grothendieck group of $\cD$, which may be of infinite rank. A \emph{(Bridgeland) pre-stability condition} $\sigma=(Z, \mathcal P)$ on a triangulated category $\mathcal D$ is characterized by a group homomorphism $Z: K(\mathcal D)\to \mathbb C$, termed the \emph{central charge}, and a collection of full additive subcategories $\mathcal P(\phi)\subset \mathcal D$ for each $\phi\in \mathbb R$, called the \emph{slicing}. This pair is required to satisfy the following axioms:
\begin{enumerate}
\item[(a)] If $0\ne E\in \mathcal P(\phi)$, then $Z(E) \in \mathbb R_{>0}\cdot{\rm exp}(\textbf i\pi\phi)$,
\item[(b)] For all $\phi\in \mathbb R$, $\mathcal P(\phi +1) = \mathcal P(\phi)[1]$,
\item[(c)] If $\phi_1>\phi_2$ and $A_i\in \mathcal P(\phi_i)\,(i=1,2)$, then ${\rm Hom}_{\mathcal D}(A_1, A_2) = 0$,
\item[(d)] For every nonzero object $E\in \mathcal D$, there exists a finite sequence of real numbers
\[
\phi_1>\phi_2>\dots>\phi_m
\]
and a collection of triangles called \emph{Harder-Narasimhan filtration}
\begin{center}
\begin{tikzcd}
0=E_0 \arrow[rr] &                        & E_1 \arrow[ld] \arrow[r] & \dots \arrow[r] & E_{m-1} \arrow[rr] &                        & E_m=E \arrow[ld] \\
                 & A_1 \arrow[lu, dashed] &                          &                 &                    & A_m \arrow[lu, dotted] &                 
\end{tikzcd}
\end{center}
with $A_i\in\mathcal P(\phi_i)$ for all $1\leq i\leq m$.
\end{enumerate}

Let $E$ be a nonzero object in $\mathcal D$ that admits a Harder-Narasimhan filtration as described in axiom (d). We associate two numbers to $E$: 
\[
\phi^+_{\sigma}(E):=\phi_1,\quad \phi^-_{\sigma}(E):=\phi_m,
\]
 where $\phi_1$ and $\phi_m$ are the first and last phases appearing in the filtration. An object $E\in \mathcal P(\phi)$ for some $\phi\in \mathbb R$ is called \emph{semistable}, and in such a case, $\phi = \phi^{\pm}_{\sigma}(E)$. Moreover, if $E$ is a simple object in $\mathcal P(\phi)$, it is said to be stable. For an interval $I$ in $\mathbb R$, we define $\mathcal P(I)$
 \[
 \mathcal P(I) = \{E\in \mathcal D\ |\ \phi^{\pm}_{\sigma}(E)\in I\}\cup\{0\}.
 \]
Consequently, for any $\phi\in \mathbb R$, both $\mathcal P[\phi, \infty)$ and $\mathcal P(\phi, \infty)$ are t-structures on $\mathcal D$ with hearts $\cP((\phi,\phi+1])$ and $\phi((\phi,\phi+1])$, respectively. According to \cite{MR2373143}, each category $\cP(\phi)\subset\cD$ is abelian, and for any interval $I\subset\mathbb R$ of length less than 1, the subcategory $\cP(I)\subset\cD$ is quasi-abelian. A stability condition $\sigma$ is called \emph{locally-finite} if there exists $\eta>0$ such that for every $t>\mathbb R$, the quasi-abelian category $\cP((t-\eta,t+\eta))\subset\cD$ is of finite length. In what follows, we restrict our attention to pre-stability conditions that fulfill the local-finiteness property. A pre-stability condition that satisfies the local-finiteness property is referred to as a \emph{ stability condition}. We denote by $\Stab(\cD)$ the set of stability conditions on $\cD$. The main result of \cite{MR2373143} asserts that $\Stab(\cD)$ carries a natural complex structure:
\begin{theorem}[{\cite[Theorem 1.2]{MR2373143}}]
\label{deformation}
The space of stability conditions $\mathrm{Stab}(\mathcal D)$ has the structure of a complex manifold, and the map
\[
\mathrm{Stab}(\mathcal D)\to\mathrm{Hom}_{\mathbb Z}(K(\cD),\mathbb C)
\]
that sends a stability condition to its central charge is a local isomorphism.
\end{theorem}

Now suppose $K(\cD)$ is free of finite rank, i.e. $K(\cD)\cong\mathbb Z^{\oplus n}$ for some $n$. In this case, one can formulate the \emph{support property} (cf. \cite{kontsevich2008stabilitystructuresmotivicdonaldsonthomas}) to pre-stability conditions. Fix a norm $\|\cdot\|$ on $K(\cD)\otimes\mathbb R$. 
A pre-stability condition $\sigma=(Z,\mathcal P)$ is said to satisfy the support property if there exists a constant $C_{\sigma}>0$ such that for every $\sigma$-semistable objects $0\ne E\in\cD$,
\begin{equation}
\label{support_property}
\Vert E\Vert\le C_{\sigma}\vert Z(E)\vert.
\end{equation}
Throughout this paper, the primary context is the bounded derived category of coherent sheaves on  a smooth projective variety. In such cases, the Grothendieck group is taken to be the numerical Grothendieck group, and the norm is defined on $K(\cD)\otimes\mathbb R$, where $K(\cD)$ now denotes the numerical Grothendieck group. By \cite{Bridgelandk3surface}, the support property implies the local-finiteness condition. Moreover, the space of pre-stability conditions with support property also admits a complex structure as described in \Cref{deformation}. In light of this, we refer to such pre-stability conditions simply as stability conditions and continue to denote their space by $\Stab(\cD)$.

Let $\mathrm{Aut}(\mathcal D)$ be the group of autoequivalences of $\mathcal D$, let $\mathrm{GL}^+_2(\mathbb R)$ be the group of elements in $\mathrm{GL}_2(\mathbb R)$ with positive determinant and let $\widetilde{\mathrm{GL}}_2^+(\mathbb R)$ be the universal cover of $\mathrm{GL}^+_2(\mathbb R)$.  As described in \cite{MR2373143}, there are natural actions of $\mathrm{Aut}(\mathcal D)$ and a right action of the group $\widetilde{\mathrm{GL}}_2^+(\mathbb R)$ on $\mathrm{Stab}(\mathcal D)$.
Explicitely, for $T\in\mathrm{Aut}(\mathcal D)$ and a stability condition $\sigma=(Z,\cP)$, we set 
\[
T\sigma=(Z',\cP'),\quad Z'=ZT^{-1},\quad \cP'_{\phi}=T\cP_{\phi}.
\]

For any element $g=(T,f)\in\widetilde{\mathrm{GL}}_2^+(\mathbb R)$, we define
\[
\sigma[g]=(Z[g],\cP[g]), \quad Z[g]=T^{-1}Z,\quad \cP[g]_{\phi}=\cP_{f(\phi)}. 
\]
In particular, for $a+\i b\in\mathbb C\subset\widetilde{\mathrm{GL}}_2^+(\mathbb R)$, we have
\[
Z[a+\i b]=e^{-\i\pi a+\pi b}Z, \quad \cP[a+\i b]_{\phi}=\cP_{\phi+a}.
\]

\subsection{G-equivarinat triangulated categories}
Let $\cD$ be a pre-additive category, linear over a field $\k$, and let $G$ be an algebraic group. We first treat the case when $G$ is finite. In this case, we assume that $(\char(\k),|G|)=1$ unless $\k$ is algebraically closed. A \emph{(right) action} of $G$ on $\cD$ consists a family of autoequivalences $(\phi_g:\cD\to\cD)_{g\in G}$ together with natural isomorphisms 
\[
(\varepsilon_{g,h}:\phi_g\phi_h\to\phi_{hg})_{g,h\in G}
\]
satisfying the compatibility condition that for all $f,g,h\in G$, the following diagram commutes:
\[
\begin{tikzcd}
\phi_{f}\phi_g\phi_h \arrow[r, "{\varepsilon_{g,h}}"] \arrow[d, "{\varepsilon_{f,g}}"'] & \phi_f\phi_{hg} \arrow[d, "{\varepsilon_{f,gh}}"] \\
\phi_{gf}\phi_h \arrow[r, "{\varepsilon_{gf,h}}"]                                       & \phi_{hgf}                                       
\end{tikzcd}
\]
Given such an action, a $G$-\emph{equivariant object} in $\cD$ is a pair $(F,(\theta_g)_{g\in G})$ where $F\in\cD$ and each $\theta_g\colon F\to\phi_g(F)$ is an isomorphism, called a  $G$-\emph{linearisation}, such that all for all $g,h\in G$, the diagram
\[
\begin{tikzcd}
F \arrow[r, "\theta_g"] \arrow[d, "\theta_{hg}"'] & \phi_g(F) \arrow[d, "\phi_g(\theta_h)"]              \\
\phi_{hg}(F)                                      & {\phi_g(\phi_h(F))} \arrow[l, "{\varepsilon_{g,h}}"']
\end{tikzcd}
\]
commutates. A \emph{morphism of } $G$-\emph{equivariant objects} from $(F_1,(\theta^1_g))$ to $(F_2,(\theta_g^2))$ is a morphism $f:F_1\to F_2$ in $\cD$ for which the diagrams 
\[
\begin{tikzcd}
F_1 \arrow[r, "\theta_g^1"] \arrow[d, "f"'] & \phi_g(F_1) \arrow[d, "\theta_g(f)"] \\
F_2 \arrow[r, "\theta_g^2"]                 & \phi_g(F_2)                         
\end{tikzcd}
\]
commute for all $g\in G$. The category of $G$-equivariant objects in $\cD$ is denoted by $\cD_G$. For a subgroup $H\subset G$, the \emph{restriction functor}
\[
\mathrm{Res}_H^G\colon\cD_G\to\cD_H
\]
is defined by $(F,(\theta_g))\mapsto(F,(\theta_g\vert_{H}))$. Conversely, the \emph{induction functor} 
\[
\mathrm{Ind}^G_H:\cD_H\to\cD_G\]
sends $(F,(\theta_h))$ to 
\[\left(\bigoplus_{[g]\in G/H}\phi_gF,(\theta_g^G)\right),\] 
where for each coset $[g]\in G/H$, the restriction of $\theta_g^G$ to the summand $\phi_{g_i}F$ is given by the composition
\[
\phi_{g_i}F\xrightarrow[]{\phi_{g_i}\theta_h}\theta_{g_i}\theta_h
F\xrightarrow[]{\varepsilon_{g_i,h}}\phi_{g_ih}F\xrightarrow[]{\varepsilon_{g,g_j}^{-1}}\phi_{g}\phi_{g_j}F,
\]
with $g_j$ and $h$ determined uniquely by $gg_j=g_ih$.

\begin{example}
Let $G$ be a finite group acting on a scheme $X$, and for each $g\in G$, set $\phi_g:=g^{\ast}:\mathrm{Coh}(X)\to\mathrm{Coh}(X)$. The canonical ismorphisms:
\[
\phi_{g}\phi_h=g^{\ast}h^{\ast}\xrightarrow[]{\sim}(hg)^{\ast}=\phi_{hg}.
\]
provides $\varepsilon_{g,h}$ defining an action of $G$ on $\mathrm{Coh}(X)$. In this setting, $G$-equivariant objects are precisely $G$-equivariant coherent sheaves (G-sheaves). We write 
\[
\mathrm{Coh}_G(X):=(\mathrm{Coh}(X))_G,\quad D^b_G(X):=D^b(\mathrm{Coh}_G(X)).
\]
If $\k=\overline{\k}$ and $G$ acts freely on a smooth projective variety $X$ over $\k$, then the quotient map $\pi:X\to X/G$ induces equivalences  
\[
\mathrm{Coh}(X/G)\cong\mathrm{Coh}_G(X),\quad D^b_G(X)\cong D^b(X/G),
\]
where the functor sends $F\in\mathrm{Coh}(X/G)$ to $(\pi^{\ast}F,(\theta_g))$ with 
\[
\theta_g:\pi^{\ast}F=\bigoplus_{h\in G}h^{\ast}F\xrightarrow[]{\sim}\bigoplus_{h\in G}g^{\ast}(h^{\ast}F).
\]
\end{example}

Now let $G$ be an algebraic group acting on a vareity $X$. Write $\mu: G \times G \to G$ for multiplication and $a: G \times X \to X$ for the action of $G$ on $X$. We denote by $p_i$ the projection onto the $i$-th factor of products such as $G \times G$, $G \times X$, or $G \times G \times X$. Additionally, $p_{12}$ and $p_{23}$ denote projections onto the first two or last two factors of $G \times G \times X$ or $G \times G \times G$, respectively. A $G$-sheaf on $X$ is a sheaf $F$ on $X$ together with an isomorphism $\theta:p_2^{\ast}F\to a^{\ast}F$ of sheaves on $G\times X$ such that the following diagram on $G\times G\times X$ commutes:
\[
\begin{tikzcd}
P_3^{\ast}F \arrow[r, "p_{23}^{\ast}\theta"] \arrow[d, "(\mu\times\mathrm{Id})^{\ast}\theta"'] & (ap_{23})^{\ast}(F) \arrow[d, "(\mathrm{Id}\times a)^{\ast}{\theta}"]              \\
(a(\mu\times\mathrm{Id}))^{\ast}F                                     & (a(\mathrm{Id}\times a))^{\ast}F \arrow[l, Rightarrow, no head, "{\varepsilon}"']
\end{tikzcd}
\]
 where $\varepsilon$ is the canonical isomorphism of functors. We denote the abelian category of $G$-sheaves by $\mathrm{Coh}_G(X)$. \emph{A morphism of $G$-equivariant objects within} $\mathrm{Coh}_G(X)$ from $(F_1,\theta_1)$ to $(F_2,\theta_2)$ is a morphism $f:F_1\to F_2$ such that the diagram
 \[
\begin{tikzcd}
p_2^{\ast}F_1 \arrow[r, "\theta_1"] \arrow[d, "p_2^{\ast}f"'] & a^{\ast}F_1\arrow[d, "a^{\ast}f"] \\
p_2^{\ast}F_2 \arrow[r, "\theta_2"]                 & a^{\ast}F_2                       
\end{tikzcd}
\]
commutes. The bounded derived category of $G$-equivariant sheaves is then defined as $D^b_G(X)\coloneq D^b(\mathrm{Coh}_G(X))$.

Let $\pi:X\to\mathrm{Spec}(\mathbb C)$ be the structural morphism. It induces functors:
\[
R\pi_{\ast}:D^b(X)\to D^b(\mathbb C),\quad R\pi^G_{\ast}:D^b_G(X)\to D^b(\Rep(G)),
\]
where $\Rep(G)$ is the category of finite- dimensional complex representations. For two objects $\cE,\cF\in D^b_G(X)$, we set
\[
\Hom^{\bullet}_G(\cE,\cF):=R\pi_{\ast}^G(E^{\ast}\otimes F)\in D^b(\Rep(G)),
\]
where $E^{\ast}:=R{\sheafhom}(E,\cO_X)$ is the ordinary dual sheaf of $E$. For any algebraic subgroup $H\subset G$, the restrction functor $\mathrm{Res}^G_H:D^b_G(X)\to D^b_H(X)$ is defined analogously to the finite-group case. When the index $[G:H]$ is finite, the induciton functor $Ind^G_H$ is defined as the exact right adjoint of $\mathrm{Res}^G_H$.

\subsection{Semiorthogonal decompositions and $G$-exceptional collections}
Recall that a \emph{semiorthogonal decomposition} of a triangulated category $\cD$ is a collection $\mathcal{D}_1,\dots,\mathcal{D}_n$ of full triangulated subcategories such that 
\begin{enumerate}[$(1)$]
\item  $\Hom_{\mathcal{D}}(\mathcal{D}_i,\mathcal{D}_j)=0$ for all $1\le j<i\le n$;
\item $\mathcal{D}=\mathrm{thick}\langle\mathcal{D}_1,\dots,\mathcal{D}_n\rangle$, i.e. $\mathcal{D}$ is the smallest triangulated subcategory containing $\mathcal{D}_1,\dots,\mathcal{D}_n$.
\end{enumerate}
A semiorthogonal decomposition is called \emph{polarizable} if each factor admits a stability condition. 
%===================================================================
We write $\mathcal{D}=\langle\mathcal{D}_1,\dots,\mathcal{D}_n\rangle$ for such a decomposition. A full triangulated subcategory $\mathcal{D}_1\subset\mathcal{D}$ is called \emph{left admissible} (resp. \emph{right admissible}) if the inclusion functor $i\colon\mathcal{D}_1\to\mathcal{D}$ admits a left adjoint $i^{\ast}$ (resp. a right adjoint $i^{!}$). A left admissible subcategory $\mathcal{D}_1\subset\mathcal{D}$ induces a semiorthogonal decomposition $\mathcal{D}=\langle\mathcal{D}_1, ^{\perp}\mathcal{D}_1\rangle$, where
\[
^{\perp}{\mathcal{D}_1}:=\{E\in\mathcal{D}\mid\Hom(E,\mathcal{D}_1[n])=0 \text{ for all } n\in\mathbb Z\}
\]
is the \emph{left orthogonal} to $\mathcal{D}_1$. Dually, 
A right admissible subcategory $\mathcal{D}_1\subset\mathcal{D}$ yields  $\mathcal{D}=\langle\mathcal{D}_1^{\perp},\mathcal{D}_1\rangle$, with \emph{right orthogonal}
\[
\mathcal{D}_1^{\perp}:=\{E\in\mathcal{D}\mid\Hom(\mathcal{D}_1[n],E)=0 \text{ for all } n\in\mathbb Z\}.
\]
A subcategory is \emph{admissible} if it is both left and right admissible. 

Now let $X$ be a smooth projective variety with an action of an algebraic group $G$. Denote by $\Rep(G)$ the category of finite-dimensional complex representations of $G$.

\begin{definition}[{\cite[Definition 2.1]{Cotti2019EquivariantQD}}]
    An object $\cE\in D^b_G(X)$ is called $G$-\emph{exceptional} if 
    \[
    \mathrm{Hom}^{\bullet}_G(\cE,\cE)\cong\mathbb C_{G},
    \]
    where $\mathbb C_G$ denotes the object of $D^b(\Rep(G))$ given by the trivial one-dimensional representation of $G$, concentrated in degree zero.  An ordered collection $(\cE_1,\dots,\cE_n)$ is a $G$-\emph{exceptional collection} if
    \begin{enumerate}
        \item each $\cE_i$ is $G$-exceptional;
        \item $\mathrm{Hom}^{\bullet}_G(\cE_j,\cE_i)=0$ for all $j>i$.
    \end{enumerate}
\end{definition}
For $\cE\in D^b_G(X)$ and $V^{\bullet}\in D^b(\Rep(G))$, the tensor product $E\otimes V^{\bullet}$ is defined as the object 
\[
\bigoplus_iE[-i]\otimes V^i,
\]
which extends the usual tensor product between  $\mathrm{Coh}_G(X)$ and $\Rep(G)$. A $G$-exceptional collection $(\cE_1,\dots,\cE_n)$ is called $G$-full if it induces a semiorthogonal decomposition 
\[
D^b_G(X)=\langle \cE_1\otimes D^b(\Rep(G)),\dots,\cE_n\otimes D^b(\Rep(G))\rangle.
\]

\section{$G$-noncommutative minimal model program}
\label{sec:G-NMMP}

In this section, we introduce the $G$-equivariant noncommutative minimal model program ($G$-NMMP). This framework extends the noncommutative minimal model program (NMMP) by incorporating actions of a reductive algebraic group $G$ on projective varieties and their derived categories, following \cite{halpernleistner2024noncommutativeminimalmodelprogram}. Throughout this section, we assume $G$ is either a finite group or a connected reductive algebraic group.

\subsection{Quasi-convergent paths}

Let $\cD$ be a triangulated category and consider a continuous path $\sigma_{\bullet}:[t_0,\infty)\to\Stab(\cD)$. For each $t\in[t_0,\infty)$, we write $\sigma_t:=(Z_t,\cP_t)$. For an object $E\in\cD$, we denote $\phi^+_t(E)=\phi_{\sigma_t}^+(E)$ and $\phi_t^-(E)=\phi_{\sigma_t}^-(E)$. Let $A_1^t,\dots,A_{k_t}$ be the Harder-Narasimhan factors of $E$ with respect to $\sigma_t$. The \emph{mass} of $E$ at $t$ is 
\[
m_t(E):=\sum_{i=1}^{k_t}\vert Z_t(A_i^t)\vert,
\]
and we define the \emph{average phase} of $E$ at $t$ by  
\[
\phi_t(E):=\frac{1}{m_t(E)}\sum_{i=1}^{k_t}\phi^+_t(A^t_i)\cdot |Z_t(A_i^t)|.
\] 

Recall that for a fixed $t\in[t_0,\infty)$, an object $0\ne E\in\cD$ is $\sigma_t$-semistable if and only if $\phi^+_t(E)=\phi^-_t(E)$. The following notion of semistability along the whole path is therefore natural:

\begin{definition}
A non zero object $E\in\cD$ is called \emph{limit semistable } (with respect to $\sigma_{\bullet}$) if:
\[
\lim\limits_{t\to\infty}(\phi^+_t(E)-\phi^-_t(E))=0.
\]
\end{definition}

\begin{definition}[{\cite{halpernleistner2024stabilityconditionssemiorthogonaldecompositions}}]
A path $\sigma_{\bullet}$ in the space of stability conditions is called \emph{quasi-convergent} if:
\begin{enumerate}
    \item For any $E\in\cD$ there exists a sequence of morphisms 
    \[0=E_0\to E_1\to E_2\to\cdots\to E_{n-1}\to E_n=E\]
    whose cones $G_i=\mathrm{cone}(E_{i-1}\to E_i)$ are limit semistable and satisfy 
    \[
    \mathop{\lim\inf}\limits_{t\to\infty}(\phi_t(G_i)-\phi_t(G_{i+1}))>0.
    \]
    This sequence is called a \emph{limit semistable filtration} of $E$.
    \item For a limit semistable object $E\in D$, we define its \emph{average logarithm} by 
    \[
    l_t(E):=\ln(|Z_t(E)|)+i\pi\phi_t(E).
    \]
    For two limit semistable objects $E,F$, we set $l_t(E|F):=l_t(E)-l_t(F)$ and require that the limit
    \[
    \lim\limits_{t\to\infty}\frac{l_t(E|F)}{1+|l_t(E|F)|}
    \]
  exists.
\end{enumerate}
\end{definition}

\begin{definition}[{\cite[Definition 2.16]{halpernleistner2024stabilityconditionssemiorthogonaldecompositions}}]
Given two limit semistable objects $E,F$, we write 
\begin{enumerate}
\item $E\sim^{\inf}F$ if $
\mathop{\lim\inf}\limits_{t\to\infty}(\phi_t(E)-\phi_t(F))<+\infty$;
\item $E<^{\inf}F$ if $\mathop{\lim\inf}\limits_{t\to\infty}(\phi_t(E)-\phi_t(F))=+\infty$.
\end{enumerate}  
\end{definition}

By {\cite[Lemma 2.17]{halpernleistner2024stabilityconditionssemiorthogonaldecompositions}}, $\sim^{\inf}$ is an equivalence relation on the set $\cP$ of limit semistable objects. The set of equivalence classes, denoted by $\cP_{\bullet}/\sim^{\inf}$, is finite (see \cite[Example 2.42 and Lemma 2.35]{halpernleistner2024stabilityconditionssemiorthogonaldecompositions}). Moreover, $<^{\inf}$ induces a total order on $\cP/\sim^{\inf}$ by {\cite[Lemma 2.17]{halpernleistner2024stabilityconditionssemiorthogonaldecompositions}}. For $E\in\cP_{\bullet}$, let $\cD^{E}$ be the full subcategory of objects whose limit semistable Harder-Narasimhan factors are $\sim^{\inf}$ equivariant to $E$. The collection $\{\cD^{E}\mid E\in\cP_{\bullet}/\sim^{\inf}\}$ is finite and totally ordered by $<^{\inf}$.

Consider germs of real $C^0$ functions at infinity, i.e. elements of \[C^0_{\infty}:=\lim\limits_{\longrightarrow}C^0((a,\infty),\mathbb R).\] 
Write  $f=g$ if $\lim\limits_{t\to\infty}f(t)-g(t)=0$ and $f<g$ if $\lim\limits_{t\to\infty}f(t)-g(t)<0$. Given $f\in C^0_{\infty}$, denote by $\cP_{\bullet}(f)\subset\cD$ the full subcategory containing $0$ together with all limit semistable objects $E\in\cD$ for which $\phi_t^{\pm}(E)=f$. Further, let $\cP_{\bullet}(>f)$ (resp. $\cP_{\bullet}(\le f)$) be the extension-closed subcategory of $\cD$ generated by $\{\cP_{\bullet}(g)\colon g>f\}$ (resp. $\{\cP_{\bullet}(g)\colon g\le f\}$).

\begin{definition}[{\cite[Definition 3.4]{halpernleistner2024stabilityconditionssemiorthogonaldecompositions}}]
Let $\overrightarrow{\sigma}=(\sigma_1,\dots,\sigma_n)\in\prod^n_{i=1}\Stab(\cD_i)$ and write $\sigma_i=(Z_i,\cP_i)$ for each $i$. Set $\cA_i=\cP_i((0,1])$. We say that a stability condition $\sigma=(Z,\cP)\in\Stab(\cD)$ is \emph{glued} from $\overrightarrow{\sigma}$ if 
\begin{enumerate}
    \item $\Hom^{\le}_{\cD}(\cA_i,\cA_j)=0$;
    \item $\cA_{\sigma}:=\cP((0,1])=\langle\cA_1,\dots,\cA_n\rangle$;
    \item $Z\vert_{\cD_i}=Z_i$.
\end{enumerate}
\end{definition}

\begin{theorem}[{\cite[Proposition 2.20 and Theorem 3.15]{halpernleistner2024stabilityconditionssemiorthogonaldecompositions}}]
\label{SODsfromQuasipath}
For any $E\in\cP_{\bullet}$, the subcategories $\cD^{E}$ are thick triangulated subcategories that give a semiorthogonal decomposition
\[
\cD=\langle \cD^{E}\mid E\in\cP_{\bullet}/\sim^{\inf}\rangle,
\]
where the order of components is induced $<^{\inf}$. Furthermore, if $\cD$ is smooth, proper and idempotent complete, then every polarizable semiorthogonal decomposition of $\cD$ can be recovered by a quasi-convergent path based by the above construction
\end{theorem}

\begin{proof}[Sketch of the proof]
Let $E,F\in\cP_{\bullet}$ with $E<^{\inf}F$. By definition, for any fixed $k\in\mathbb Z$, we have $\phi^-_t(F)>\phi^+_t(E)+k$ for sufficiently large $t$. Hence $\Hom(F,E(k))=0$ which implies that 
$\Hom(\cD^{F},\cD^{E})=0$. By coarsening the limit Harder-Narasimhan filtration, every object of $\cD$ lies in the triangulated hull of the subcategories $\{\cD^{E}\}$. Since $\cP_{\bullet}/\sim^{\inf}$ is totally ordered by $<^{\inf}$, the collection $\langle \cD^{E}\mid E\in\cP_{\bullet}/\sim^{\inf}\rangle$ forms a semiorthogonal decomposition.

Now let $\cD=\langle\cD_1,\dots,\cD_n\rangle$ be a polarizable semiorthogonal decomposition and choose stability conditions $\sigma_i\in\Stab(\cD_i)$. Pick functions $z_i:[t_0,\infty)\to\mathbb C$ for $i=1,\dots,n$ such that for all $i<j$, 
\[
\lim\limits_{t\to\infty}\frac{z_j(t)-z_i(t)}{1+|z_j(t)-z_i(t)|}=e^{\i\theta_{ij}},
\]
for $\theta_{ij}\in(0,\pi)$. By \cite[Proposition 3.2]{halpernleistner2024stabilityconditionssemiorthogonaldecompositions}, the family 
$\overrightarrow{\sigma_t}:=(z_i(t)\cdot\sigma_i)_{i=1}^n$ can be glued into a quasi-convergent path $\sigma_{\bullet}\in\Stab(\cD)$ for $t\gg0$. Based on the construction, $\sigma_{\bullet}$ recovers the given semiorthogonal decomposition. Note that the choice of such a path quasi-convergent path is unique up to the $\mathbb C$-action on each $\sigma_i$; see \cite[Theorem 2.30]{halpernleistner2024stabilityconditionssemiorthogonaldecompositions}.
\end{proof}

\subsection{Equivariant quantum cohomology}
We give a brief review of equivariant quantum cohomology for projective varieties and we recommend
\cite{Anderson_Fulton_2023} for details.
Let $X$ be a smooth projective variety of dimension $m$ equipped with an action of a complex algebraic reductive group $G$. The \emph{equivariant cohomology} of $X$ is defined as the ordinary cohomology of the Borel construction:
\[
H^{\ast}_G(X)=H^{\ast}(X\times_{G}EG),
\]
where $EG$ is a contractible space with a free $G$-action and the group $G$ acts diagonally on $X\times EG$. Denote by $BG:=EG/G$ the \emph{classifying space} of $G$. The projection
\[
p_X:X\times_GEG\to BG
\]
is a fibration with fiber $X$. In particular, if $X$ is a point, we have
\[
 H^{\ast}_G(\mathrm{pt})=H^{\ast}(BG).
\]
In general, $H^{\ast}_G(X)$ is an algebra over $H^{\ast}_G(\mathrm{pt})$, which coincides with the invariants of the Weyl group acting on the characteristic classes for a maximal torus $\bT$:
\[
H^{\ast}_G(\mathrm{pt})=H^{\ast}_{\bT}(\mathrm{pt})^W.
\]
For example, if $G=\bT=(\mathbb C^{\ast})^n$, then
\[
H^{\ast}_G(\mathrm{pt})=\mathbb C[z_1,\dots,z_n].
\]
For $G=\mathrm{GL}_n(\mathbb C)$, the Weyl group is the symmetric group $S_n$, and 
\[
H^{\ast}_{\mathrm{GL}_n(\mathbb C)}(\mathrm{pt})=\mathbb C[z_1,\dots,z_n]^{S_n}.
\]
If $G$ is finite, $H^{\ast}_G(\mathrm{pt})=\mathbb C$. By the equivariant formality result of \cite{GKM97}, $H^{\ast}_\bT(X)$ is a free module of rank $\dim H^{\ast}(X)$ over $H^{\ast}_\bT(\mathrm{pt})$.

Let $\beta\in H_2(X;\mathbb Z)$. An $n$-pointed, genus $g$, degree $\beta$ \emph{prestable map} to $X$ is a morphism $f:(C,x_1,\dots,x_n)\to X$, where
\begin{enumerate}
    \item $(C,x_1,\dots,x_n)$ is a connected algebraic curve $C$ of arithmetic genus $g$ with $n$ ordered marked points $x_1,\dots,x_n\in C$, where $C$ has at most nodal singularities, and $(x_1,\dots,x_n)$ are distinct smooth points;
    \item $f_{\ast}[C]=\beta$.
\end{enumerate}
Two such prestable maps
\[
f:(C,x_1,\dots,x_n)\to X,\quad\text{ and } f':(C',x_1',\dots,x_n')\to X
\]
are isomorphic it there exists an isomorphism $\phi:(C,x_1,\dots,x_n) \to(C',x_1',\dots,x_n')$ of $n$-pointed prestable curves such that $f=f'\circ\phi$. A prestable map $f$ is called \emph{stable} if its automorphism group is finite. 

Denote by $\overline{\cM}_{g,n}(X,\beta)$ the Deligne-Mumford stack of $n$-pointed, genus $g$, degree $\beta$  stable maps. According to \cite{Behrend1995StacksOS}, $\overline{\cM}_{g,n}(X,\beta)$ is proper when $X$ is projective. We assume that either $\beta\in\mathrm{Eff}(X)$ or $2g+n>2$ so that $\overline{\cM}_{g,n}(X,\beta)$ is non-empty. Given $n$ cohomology classes
\[
\gamma_1,\dots,\gamma_n\in H^{\ast}_{G}(X;\mathbb C),
\]
and integers $d_1,\dots,d_n\in\mathbb Z_{\ge0}$,  the \emph{genus} $g$, degree $\beta$, $G$-\emph{equivariant descendant Gromov-Witten invariants} of $X$ are defined as
\[
\langle\tau_{d_1}(\gamma_1),\dots,\tau_{d_n}(\gamma_n)\rangle^{X,G}_{g,n,\beta}:=\left(\int^{eq}_{[\overline{\cM}_{g,n}(X,\beta)]^{\mathrm{vir}}_{G}}\prod^n_{j=1}\psi_j^{d_j}\mathrm{ev}^{\ast}_j(\gamma_j)\right)\in H^{\ast}_{G}(\mathrm{pt},\mathbb C),
\]
where
\begin{enumerate}
    \item $[\overline{\cM}_{g,n}(X,\beta)]^{\mathrm{vir}}_{G}$ is the \emph{equivariant virtual fundamental class}, living in $A^{G}_{D_{\mathrm{vir}}}(\overline{\cM}_{g,n}(X,\beta))$. The virtual dimension is
    \[
    D_{\mathrm{vir}}:=\int_{\beta}c_1(TX)+(m-3)(1-g)+n;
    \]
    \item $\mathrm{ev}_j:\overline{\cM}_{g,n}(X,\beta)\to X$ is the evaluation map at the $j$-th marked point, which is $G$-equivariant;
    \item $\psi_j\in A_1^{G}(\overline{\cM}_{g,n}(X,\beta))$ denotes any equivariant lift of the first Chern class of the universal cotangent line bundles $\cL_j$ on $\overline{\cM}_{g,n}(X,\beta)$.
\end{enumerate}

If $d_1,\dots,d_n$ are all zero, the $G$-equivariant descendant Gromov-Witten invariants above are called \emph{primary} equivariant Gromov-Witten invariants. Let \(\{\phi_{i}\}_{i = 0}^{s}\) be a homogeneous basis of \(H^{*}_{G}(X)\) over \(H^{\ast}_G(\mathrm{pt})\) with $\phi_0=1$, and let $\{\tau^i\}^s_{i=0}$ be the dual coordinate system on $H^{\ast}_G(X)$. A general point \(\tau \in H_{G}^{*}(X)\) can be written as \(\tau = \sum_{i = 0}^{s} \tau^{i} \phi_{i}\) where \(\tau^{i} \in H^{\ast}_G(\mathrm{pt})\). The degree of $\tau^i$ is set to be $\deg\tau^i:=2-\deg\phi_i$.

Equivariant quantum cohomology is defined over the ring \[\mathbb C[\![Q,\tau]\!]=\mathbb C[\![Q]\!][\![\tau^0,\dots,\tau^s]\!].\]
Note that odd variables anti-commute with each other: for $|i|:=\deg\,\phi_i\,(\mod\,2)$, 
\[
\tau^i\tau^j=(-1)^{|i||j|}\tau^j\tau^i,\quad \tau^i\phi_j=(-1)^{|i||j|}\phi_j\tau^i.
\]
 The \emph{big equivariant quantum product} at $\tau\in H^{*}_{G}(X, \mathbb{C})$, denoted by $\star_{\gamma}$, is defined by
\[
\langle \phi_{\alpha}\star_{\tau} \phi_{\beta}, \phi_{\gamma} \rangle^{X,G} = \sum_{d \in \mathrm{NE_{\mathbb{N}}(X)}, n \geq 0} \langle \phi_{i}, \phi_{j}, \phi_{k}, \tau, \cdots, \tau \rangle_{0, n + 3, d}^{X,G} \frac{Q^{d}}{n!}
\]
The quantum product \(\star_{\tau}\) makes \(H_{G}^{*}(X)\llbracket Q, \tau \rrbracket\) a \emph{Frobenius algebra}: it is commutative, associative algebra, unital, and compatible with the \emph{equivariant Poincar\'{e} pairing}
\[
\eta^G(a,b):=\int^{eq}_{X}a\cdot b.
\]
i.e.
\[
\eta^G(a\star_{\tau} b, c)=\eta^G(a,b\star_{\tau} c), \quad a,b,c\in B.
\]

The algebra $(H^{\ast}_G(X)[\![Q,\tau]\!],\star_{\tau})$ is called the \emph{equivariant quantum cohomology}. The product is homogeneous and has $\phi_0$ as the unit.
We assume that there exists an open subset \(B \subset H_{G}^{*}(X)\), such that for every \(\tau \in B\), the sum defining \(\star_{\tau}\) converges. If \(\tau \in H^{2}_{G}(X)\) and we set \(Q=1\), then the  divisor equation simplies the expression to
\begin{equation*}
    \label{eq:quantum product}
\langle \phi_{\alpha} \star_{\tau} \phi_{\beta}, \phi_{\gamma} \rangle = \sum_{d \in \mathrm{NE}(X)_{\mathbb{N}}} \langle \phi_{\alpha}, \phi_{\beta}, \phi_{\gamma} \rangle_{0, 3, d}^{X,G}\exp(\tau \cdot d)
\end{equation*}

\subsection{Main proposal}
\label{subsec:main proposal}
Let $G$ be a reductive algebraic group acting on smooth projective varieties $X$ and $Y$. Recall that a $G$-\emph{contraction} is a projective surjective $G$-equivariant morphism $\pi:X\to Y$ with connected fibers. Building on the recent breakthrough result of \cite{Li2026}, which establishes the existence of stability conditions on any smooth projective varieties, we shall henceforth assume throughout this paper that 
 $D^b_G(X)$ also admits stability conditions. Moreover, all semiorthogonal decompositions we consider are required to be \emph{polarizable}, meaning that each component admits a stability condition. For brevity, we write 
\[
\Stab(X_G)\coloneq \Stab(D^b_G(X)),\quad K^G(X)\coloneq K^G_{\num}(D^b_G(X)).
\] 

\begin{definition}
\label{def:base change}
Let $\mathcal K$ be the field of meromorphic functions on the affine scheme $\mathrm{Spec}(H^{\ast}_G(\mathrm{pt}))$. The $G$\emph{-equivariant cohomology of $X$ with coefficients in} $\mathcal K$ is defined as
\[
H^{\mathcal K}_G(X;\bC)\coloneq H^{\ast}_G(X;\bC)\otimes_{H^{\ast}_G(\mathrm{pt};\bC)}\mathcal K.
\]
\end{definition}
The \emph{twisted chern character map} is then given as follows:
\[
v^G\coloneq (2\pi \i)^{\deg/2}\ch_G:K_{\num}^G(X)\to H^{\mathcal K}_G(X;\mathbb C).
\]
We require that the central charges of stability conditions in $\Stab(X_G)$ factor through $v^G$. 

For completeness, we provide a proof of the equivariant version of the Bondal-Orlov deomposition result \cite{Bondal1995SemiorthogonalDF}.

\begin{proposition}
  Let $G$ be an algebraic group acting on smooth projective varieties $X$ and $X'$, and let $\pi:X\to X'$ be a $G$-equivariant morphism such that $R\pi_{\ast}(\cO_X)=\cO_{X'}$. Then $\pi^{\ast}$ is fully faithful and there is a semiorthogonal decomposition of $D^b_G(X)$:
\[
D^b_G(X)=\langle\mathrm{ker}(R\pi_{\ast}),\pi^{\ast}(D^b_G(X'))\rangle.
\]
\end{proposition}
\begin{proof}
For \(\cE, \cF \in D^{b}_{G}(X')\), we have
\[
\begin{aligned}
  {\rm Hom}_{G}(\pi^{*}\cE, \pi^{*}\cF) &= {\rm Hom}_{G}(\cE, R\pi_{*}L\pi^{*}\cF) \\
  &= {\rm Hom}_{G}(\cE, \cF)
\end{aligned}
\]
where the second equality uses the projection formula and the hypothesis \(R\pi_{*} \cO_{X} = \cO_{X'}\). Therefore, $\pi^{\ast}$ is fully faithful. For any $\cE\in\ker(\pi_{\ast})$ and $\cF\in D^b_G(X')$, we have 
\[
\Hom_G(\cE,\pi^{\ast}\cF)\cong\Hom_G(R\pi_{\ast}\cE,\cF)=0.
\]
Hence, $\ker(\pi_{R\ast})$ is left-orthogonal to $\pi^{\ast}(D^b_G(X'))$. For any $\cE\in D^b_G(X)$, the counit of the adjunction $(\pi^{\ast},R\pi_{\ast})$ yields the triangle in $D^b_G(X)$:
\[
\pi^{\ast}R\pi_{\ast}\cE\xrightarrow{\eta} B\to\cB\to C(\eta)\to \pi^{\ast}R\pi_{\ast}\cE[1].
\]
Applying $R\pi_{\ast}$ and using $R\pi_{\ast}\pi^{\ast}\cong\mathrm{id}$, we get $R\pi_{\ast}C(\eta)\cong0$, i.e. $C(\eta)\in\ker(\pi_{\ast})$. Therefore, this gives the claimed semiorthogonal decomposition.
\end{proof}

We now state the following conjecture:
\begin{sloppypar}
\begin{conjecture}
    \label{conjecture:1}
Let $\pi:X\to Y$ be a $G$-contraction of a smooth projective variety $X$. 
\begin{itemize}
    \item[(A)] One can associate to $\pi$ a canonical family of quasi-convergent paths $\{\sigma_t^{\pi,\psi}\}$ in $\Stab(X_G)/\mathbb C$. Generic values of the parameter $\psi$ yield a $G$-linear semiorthogonal decomposition of $D^b_G(X)$, and different generic values give mutation-equivalent decompositions.
    \item[(B)] For a generic $\psi$, the semiorthogonal factors of $D^b_G(X)$ are closed under tensor product with complexes of the form $\pi^{\ast}(E)$ for $E\in\mathrm{Perf}_G(Y)$.
    \item[(C)] Given a further $G$-contraction $Y\to Y'$, the semiorthogonal decomposition of $D^b_G(X)$ associated with $X\to Y'$ refines the decomposition associated with $X\to Y$.
    \item[(D)] If $\pi:X\to X'$ is a $G$-equivariant morphism of smooth varieties with $R\pi_{\ast}(\cO_X)=\cO_{X'}$, then for suitable parameters, the semiorthogonal decomposition of $D^b_G(X)$ asscoiated to $X\to Y$ refines the decomposition obtained by combining the semiorthogonal decomposition of 
    \[\pi^{\ast}(D^b_G(X'))\cong D^b_G(X')\]
    associated to $X'\to Y$ with the decomposition 
    \[
    D^b_G(X)=\langle\mathrm{ker}(\pi_{\ast}),\pi^{\ast}(D^b_G(X'))\rangle.
    \]
\end{itemize}
\end{conjecture}
\end{sloppypar}

A recent work \cite{ESS2025} provides strong support for this conjecture in the two-dimensional case.

We now formulate a precise proposal for the canonical quasi-convergent paths in $\Stab(X_G)$ predicted by \Cref{conjecture:1}.
Let \(\zeta = \zeta(t) \in H_{G}^{*}(X; \mathbb{C})\). Consider the equivariant quantum differential equation

\begin{equation}
    \label{eq: qde for Fano}
 t\frac{d \zeta}{d t} + c_{1}^G(X) \star_{\ln(t)c^G_{1}(X)} \zeta=0.
\end{equation}

Here, $c_1^G(X)$ denotes the $G$-equivariant first Chern class. In general, the sum in \eqref{eq: qde for Fano} is only formal in $t$. To handle this, we recall the truncation technique introduced in \cite{halpernleistner2024noncommutativeminimalmodelprogram}. Denote by $\mathrm{NE}(X/Y)_{\mathbb N}$ the set of numerical equivalence classes of effective integral 1-cycles spanned by curves that are contracted by $\pi$. For a 1-cycle $d\in\mathrm{NE}(X/Y)_{\mathbb N}$ with $c_1(X)\cdot d>0$, define the operator $T_d^G$ by
\[
\eta^G(T^G_d(\alpha),\beta)=\langle\alpha,\beta\rangle_{0,2,d}^{X,G}.
\]
 Let $\psi=\omega+iB\in H^2(X;\mathbb C)$ with $\omega$ a relatively ample class. The \emph{$G$-truncated quantum endomorphism} is defined by 
\begin{equation}
\label{eq:truncated endomorphism}
E^G_{\psi}(t):=c_1^G(X)\cup(-)+\sum_{\substack{d\in\mathrm{NE}(X/Y)_{\mathbb N}\\ (c_1(X)-\omega)\cdot d>0}}(c_1(X)\cdot d)t^{c_1(X)\cdot d}e^{-\psi\cdot d}T^G_d.
\end{equation}
According to \cite[Lemma 6]{halpernleistner2024noncommutativeminimalmodelprogram}, the sum in \eqref{eq:truncated endomorphism} is finite.  The corresponding \emph{$G$-truncated quantum differential equation} is 
\begin{equation}
    \label{eq: qde}
t\frac{d\zeta(t)}{dt}+\frac{1}{z}E_{\psi}^G(t)\zeta(t)=0, \quad \zeta(t)\in H^{\ast}_G(X;\mathbb C),
\end{equation}
where $z\in\mathbb C$ is a parameter. 
Recall the grading operator $\mu$ on $H^{\ast}_G(X)$: for $\phi_i\in H^i_G(X)$ and $\tau\in H^{k}(BG)$,
\[
\mu(\tau\phi_i)\coloneq\frac{\deg(\tau\phi_i)-m}{2}=\frac{i+k-m}{2}\tau\phi_i.
\] 
For $\bT=(\mathbb C^{\ast})^n$, this becomes 
\[
\mu(t_i^k\phi_d)=\left(\frac{\deg\phi_i-m}{2}+|k|\right)t_i^k\phi_i.
\]
\begin{lemma}
  $tE^G_{\psi}(1)t^{\mu}=t^{\mu}E_{\psi}^G(t)$.
\end{lemma} 
\begin{proof}
According to \cite[Lemma 6]{halpernleistner2024noncommutativeminimalmodelprogram}, the operator $T^G_d$ is homogeneous of degree $\Delta_d\coloneq2(1-c_1(X)\cdot d)$. Indeed, for homogeneous $\alpha,\beta\in H^{\ast}_G(X)$,
\[
\deg(\alpha)+\deg(\beta)-2(m-1+c_1(X)\cdot d)=\deg\langle\alpha,\beta\rangle_{0,2,d}^{X,G},
\]
and
\[
\deg(T^G_d\alpha)+\deg(\beta)-2m=\deg(\eta^G(T^G_d(\alpha),\beta)).
\]
Therefore, 
\[
\deg(T_d^G\alpha)-\deg(\alpha)=2-2c_1(X)\cdot d.
\]
Consequently,
\[
t^{\mu}(T^G_d\alpha)=t^{\frac{(\deg(\alpha)+\Delta_d)-m}{2}}\cdot T_d^G\alpha=t^{1-c_1(X)\cdot d}\cdot T_d^G(t^{\mu}\alpha).
 \]
In addition,
 \[
 t^{\mu}(c^G_1(X)\cup\alpha)=t^{\frac{(\deg\alpha+2)-m}{2}}(c_1^G(X)\cup\alpha)=t\cdot(c_1^G(X)\cup((t^{\mu}\alpha))).
 \]
Thus, 
\[
\begin{aligned}
t^{\mu}E_{\psi}^G(t)(\alpha)&=t\cdot(c_1^G(X)\cup((t^{\mu}\alpha)))+\sum_{\substack{d\in\mathrm{NE}(X/Y)_{\mathbb N}\\ (c_1(X)-\omega)\cdot d>0}}(c_1(X)\cdot d)te^{-\psi\cdot d}T^G_d(t^{\mu}\alpha)\\
&=tE^G_{\psi}(1)t^{\mu}.
\end{aligned}
\]
\end{proof} 

Under the change of variables $\tilde{\zeta}(t)=t^{\mu}\zeta(t)$, we obtain the equivalent system:
\begin{equation*}
\label{eq:truncated qde2}
\frac{d\tilde{\zeta}(t)}{dt}=A(t)\tilde{\zeta}(t),
\end{equation*}
where
\[
A(t)\coloneq\frac{1}{t}\mu^G-\frac{1}{z}E_{\psi}^G(1).
\]
Note that $A(t)$ is meromorphic with poles at $0$ and $\infty$. By the Hukuhara-Turrittin theorem in \cite{wasow2018asymptotic}, for $|t|>t_0$ in a sector $S\subset\mathbb C$ containing $\mathbb R_{>0}$, there exists a fundamental solution of the form
\[
\Phi^G_t=Y^G(t)e^{tD^G+B^G(t)},
\]
where 
\begin{enumerate}
\item $Y^G(t)$ is invertible and holomorphic, with a uniform asymptotic expansion 
\[
Y^G(t)\sim Y_0^G+Y_1^Gt^{-1/p}+Y_2^Gt^{-2/p}+\cdots,
\] 
for some $p\in\mathbb Z_{>0}$, and the columns of $Y_0^G$ form a basis of generalized eigenvectors of $\frac{-1}{z}E_{\psi}^G(1)$;

\item $D^G$ is the diagonal matrix of eigenvalues of $\frac{-1}{z}E^G_{\psi}(1)$ corresponding to $Y_0^G$;

\item $B^G(t)=\sum_{k=1}^{p-1}D^G_kt^{k/p}+C^G\ln t$, with constant diagonal matrices $D^G_k$ and $C^G$ commuting with $D^G$ and $\|B^G(t)\|=O(t^{(p-1)/p})$.
\end{enumerate}
Moreover, 
\begin{itemize}
\item if $E^G_{\psi}(1)$ is semisimple, then $D^G_{k}=0$ for $k=1,\dots,p-1$;
\item if the eigenvalues of $E^G_{\psi}(1)$ are distinct, then in an orthonormal eigenbasis one may take $B(t)=0$.
\end{itemize}

Let $\Phi^G_t\in\End(H^{\ast}_{G}(X;\mathbb C))$ be the fundamental solution of \eqref{eq: qde}. It extends canonically to
\[
\Phi^G_{t}\in\End_{\mathcal K}(H^{\mathcal K}_{G}(X;\mathbb C))
\]
by setting
\[
\Phi^G_{t}(\alpha\otimes f)\coloneq \Phi^{G}_t(\alpha)\otimes f
\]
for $\alpha\in H^{\ast}_{G}(X;\mathbb C)$ and $f\in\mathcal K$. Consider the \emph{evaluation map}:
\begin{equation*}
\begin{aligned}
\ev_{\s}:H^{\ast}_{\bT}(\mathrm{pt})&\to\mathbb C\\
z_i^n&\mapsto {\i \pi n s_i},
\end{aligned}
\end{equation*}
which naturally extends to $H^{\ast}_\bT(X)$ via its module structure over $H^{\ast}_{\bT}(\mathrm{pt})$. Let $\bT=(\mathbb C^{\ast})^m$ be a maximal torus of $G$ with Weyl group $W$, and let $\s=(s_1,\dots,s_m)\in\mathbb C^m$. Via the restriction map in equivariant cohomology, we have an isomorphism
 \[
 H_G^{\ast}(X)\cong H^{\ast}_{\bT}(X)^W.
 \]
Consequently, we may restrict $\ev_\s$ to $H^{\ast}_G(X)$. For a meromorphic function $f\in\mathcal K$ that is homolomorphic at $({\i\pi s_1},\dots,{\i\pi s_m})$, set
\[
\ev_\s(\alpha\otimes f)\coloneq \ev_\s(\alpha)\cdot f(\s).
\]
where $f(\s)$ denotes evaluation of $f$ at the point $({\i\pi s_1},\dots,{\i\pi s_m})$. We say that $\alpha\otimes f\in H^{\mathcal K}_G(X;\mathbb C)$ is holomorphic at $\s$ if $f$ is holomorphic at $({\i\pi s_1},\dots,{\i\pi s_m})$.

\begin{remark}
The evaluation map $\ev_{\s}$ is chosen to be compatible with the one used to define $\bT$-stability conditions in \Cref{sec:T-stab}. For a finite group $G$, the maximal torus is trivial. Hence, the evaluation map becomes trivial.
\end{remark}

The main proposal can now be stated as follows:
\begin{proposal}
\label{proposal:main}
Let $\pi:X\to Y$ be a $G$-contraction of a smooth projective variety $X$ of dimension $m$. Let $\Phi^G_t\in\End(H^{\ast}_G(X;\mathbb C))$ be a fundamental solution of the  $G$-truncated quantum differential equation \eqref{eq: qde}. Then, for any generic tuple $\s=(s_1,\dots,s_m)\in\mathbb C^m$, there exists a quasi-convergent path $\sigma_t^G=(Z_t^G,\cP_t^G)$ for $t\in[t_0,\infty)$ in $\Stab(X_G)/\mathbb C$ whose central charges are given by
\[
Z_{t}^G(\cE)=\ev_\s\int^{eq}_X\Phi^G_t(v^G(\cE)),
\]
for $\cE\in K^G(X)$.
\end{proposal}
Here, the genericity of $\s$ ensures that the evaluation $\ev_\s$ yields finite complex numbers. For instance, for $\bP^{m-1}$ endowed with the diagonal action of the torus $\bT=(\mathbb C^{\ast})^m$ (see \Cref{sec:G-NMMP for projective spaces}), the tuple $\s$ can be taken from  the set $\Omega_{\bT}$, which is defined as the complement in \(\bC^{m}\) of the hyperplanes given by
\[
z_{i} - z_{j} = k, \quad i,j = 1, \cdots, m, i \neq j, k \in \bZ.
\]

\begin{remark}
If $X$ is a Fano variety, the sum in \eqref{eq: qde for Fano} is finite. In this case, one can formulate the proposal of \Cref{conjecture:1} directly using the equivariant quantum differential equation \eqref{eq: qde for Fano}. If $G$ is finite, one can also consider stability conditions supported on the lattice provided by the Chen-Ruan cohomology ring, together with quantum differential equations arising from orbifold Gromov-Witten theory \cite{ChenRuancohomology, chen2002orbifold}. For a connected reductive algebraic group $G$, quasi-convergent paths required in \Cref{proposal:main} are acturally contained in the subspace $\bT\Stab_{\s}(X)$ of $\Stab(X_G)$ defined in \Cref{sec:T-stab}, where $\bT$ denotes a maximal torus of $G$.
\end{remark}

\begin{remark}
Decategorification \Cref{proposal:main} suggests a direct sum decomposition of the $G$-equivariant topological K-theory:
\[
K^{\mathrm{top}}_G(X)_{\mathbb Q}\cong H_{1,\psi}^G\oplus\cdots\oplus H^G_{n,\psi},
\]
which should arise from the asymptotic behavior of solutions to \eqref{eq: qde}. Specifically, the subspaces $H^G_{i,\psi}$ should correspond to the generalized eigenspaces of the operator $\frac{-1}{z}E_{\psi}^G(1)$ acting on $H^{\ast}_{G}(X)$, with the decomposition being upper-triangular with respect to the $G$-equivariant Euler pairing. 
\end{remark}

\begin{remark}
The whole $G$-equivariant framework developed here is a natural generalization of the noncommutative minimal model program. When $G=1$ (the trivial group), it reduces exactly to the classical non-equivariant setting described in \cite{halpernleistner2024noncommutativeminimalmodelprogram}.
\end{remark}

In \cite[Remark 13]{halpernleistner2024noncommutativeminimalmodelprogram}, it is suggested that the quasi-convergent paths of \Cref{proposal:main} should start from a \emph{geometric} stability condition, i.e. one under which all skyscraper sheaves of points are stable and share the same phase. We now extend this idea to the equivariant setting. Let $x\in X$ and denote by $G_x\subseteq G$ the stabilizer of $x$. The skyscraper sheaf $\cO_x$, equipped with the trivial $G_x$-action, defines an object $\cO_x^{G_x}\in D^b_{G_x}(X)$. Assuming the index $[G:G_x]$ is finite, the $G$\emph{-point sheaf} associated to $x$ is given by
\[
P_x\coloneq\mathrm{Ind}^G_{G_x}(\cO_x^{G_x})\in D^b_G(X).
\]

\begin{definition}
\label{def:geometric stab for stack}
Let $X$ be a smooth projective variety with an action of an  algebraic group $G$. A stability condition $\sigma\in\Stab(X_G)$ is called \emph{geometric} if, for every $x\in X$ with finite $[G:G_x]$, the corresponding $G$-point sheaf $P_x$ is $\sigma$-stable and all such $P_x$ have the same phase.
\end{definition}

If $G$ is finite, $[G,G_x]$ is automatically finite for any $x$. For the diagonal action of $\bT=(\mathbb C^{\ast})^m$ on $\mathbb P^{m-1}$, the $\bT$-point sheaf is defined only for the $m$ torus-fixed point.

\section{Inducing quasi-convergent paths}
\label{sec:induce paths for finite groups}
\begin{sloppypar}
In this section, we treat the $G$-NMMP for finite groups. Using techniques for inducing stability conditions from \cite{MacriSukhenduPaolo07,2023arXiv231002917Q,dell2024fusionequivariantstabilityconditionsmorita}, we construct quasi-convergent paths in $\Stab(X_G)/\mathbb C$ that fulfill \Cref{proposal:main} by lifting known non-equivariant solutions. We then apply this method to projective spaces and certain blow-up surfaces with finite group actions, relying the explicit non-equivariant paths obtained in \cite{zuliani2024semiorthogonaldecompositionsprojectivespaces,karube2024noncommutativemmpblowupsurfaces}.
\end{sloppypar}

\subsection{Inducing theorem}
\begin{definition}[{\cite{MacriSukhenduPaolo07,dell2024fusionequivariantstabilityconditionsmorita}}]
Let $\cD$ be a triangulated category with an action of a finite group $G$. A stability condition $(Z,\cP)\in\Stab(\cD)$ is called $G$-\emph{invariant} if:
\begin{itemize}
    \item $\phi_g\cP(\phi)=\cP(\phi)$ for all $g\in G$ and for all $\phi\in\mathbb R$;
    \item $Z\in\Hom_{\mathbb Z}(K(\cD),\mathbb C)^G\subset\Hom_{\mathbb Z}(K(\cD),\mathbb C)$,
\end{itemize}
where $\mathrm{Hom}_{\mathbb Z}(K(\cD),\mathbb Z)^G$ denotes the $\mathbb C$-linear subspace of $G$-invariant homomorphisms $Z$. We denote by $\Stab(\cD)^G$ the set of all $G$-invariant stability conditions.
\end{definition}

Based on a recent result of \cite{Li2026}, we show the non-existence of $G$-invariant stability conditions on any smooth projective variety:
\begin{proposition}
  \label{prop:g-stab}
  Let \(X\) be a smooth projective variety with an action of a finite group $G$. Then $\Stab(X)^G$ is non-empty.
\end{proposition}

\begin{proof}
We first consider the case $X=\mathbb P^n$. By \cite[Theorem 6.3]{Li2026}, there exists a stability condition $\sigma=(Z,\cP)$ on $\bP^n$ whose central charge $Z$ is invariant under the action of $G\subseteq\Aut(\bP^n)=\PGL(n+1)$. 
Since $\PGL(n+1)$ preserves the numerical Grothendieck group of $\bP^n$, $Z$ is in fact invariant under the entire $\PGL(n+1)$. Moreover, because  $\PGL(n+1)$ is connected and its action on $\Stab(\bP^n)$ is continuous, $\sigma$ is $\PGL(n+1)$-invariant, and hence $G$-invariant. Thus, $\sigma\in\Stab(\bP^n)^G$. Now let $X$ be any smooth projective variety with a $G$-action. Choose a very ample line bundle $L$ on X and set \[\widetilde{L}\coloneq \otimes_{g\in G}g^{\ast}L.\]
Then $\widetilde{L}$ is a $G$-invariant very ample line bundle, and we may identify $\bP(H^0(X,\widetilde{L}))\cong\bP^n$ for some $n$. This gives a $G$-equivariant embedding $\iota:X\hookrightarrow\bP^n$. By the first part of the proof, there exists a $G$-invariant stability condition $\sigma\in\Stab(\bP^n)^G$. Applying the pullback construction from \cite[Theorem 6.5]{Li2026}, we obtain a stability condition 
\[
\iota^{\#}\sigma=(Z\circ\iota_{\ast},\iota^{\#}\cP)\in\Stab(X),
\]
where
\[
\iota^{\#}\cP(\phi)\coloneq\{E\in D^b(X)\mid \iota_{\ast}E\in\cP(\phi)\}.
\]
Since $\sigma$ is $G$-invariant and the embedding $\iota$ is $G$-equivariant, the induced stability condition $\iota^{\#}\sigma$ is also $G$-invariant. Hence, $\iota^{\#}\sigma\in\Stab(X)^G$, which completes the proof. 
\end{proof}

\begin{definition}
A quasi-convergent path $\sigma_{t}$ is called $G$-invariant if $\sigma_t\in\Stab(\cD)^G$ for any $t$.
\end{definition}

Given a quasi-convergent path $\sigma^G_t$ in $\Stab(X_G)$ that satisfies \Cref{proposal:main} via a fundamental solution $\Phi^G_t$ of \eqref{eq: qde}, we define the \emph{spanning condition} as follows:

\begin{definition}
\label{def:spanning condition for finite groups}
Let $r\in\mathbb R$ be the real part of an eigenvalue of $\frac{-1}{z}E^G_{\psi}(1)$. The \emph{asymptotic growth subspace} is:
\[
F^r_G\coloneq\left\{\alpha\in H^{\ast}_G(X;\mathbb C)\colon\mathop{\lim\sup}\limits_{t\to\infty}\frac{\ln\|\Phi^G_t(\alpha)\|}{t}\le r\right\}.
\]
We say the pair $(\sigma^G_t,\Phi^G_t)$ satisfies the \emph{spanning condition}  if $F^r_G$ is spanned by the classes of limit semistable objects $\cE\in D^b_G(X)$ (with respect to $\sigma^G_t$) for which 
\[
\mathop{\lim\inf}\limits_{t\to\infty}\frac{|Z^G_t(\cE)|}{\|\Phi^G_t(\cE)\|}>0.
\]
\end{definition}

The main result of this section is the following lifting theorem:
\begin{theorem}
\label{thm: inducing solutions}
Let $X$ be a smooth projective variety with an action of a finite group $G$. Let $\sigma_{t}=(Z_t,\cP_t)$ for $t\in[t_0,\infty)$ be a $G$-invariant quasi-convergent path in $\Stab(X)^G$ that solves the non-equivariant \Cref{proposal:main} via a fundamental solution $\Phi_t$ of \eqref{eq: qde}. Then it induces a quasi-convergent path $\sigma^G_t$ in $\Stab(X_G)$ that solves \Cref{proposal:main} via a fundamental solution $\Phi^G_t$ of \eqref{eq: qde} and preserves the spanning condition. Moreover, $\sigma_{t_0}^G$ is geometric if $\sigma_{t_0}$ is geometric.
\end{theorem}

\begin{theorem}[{\cite{2023arXiv231002917Q,dell2024fusionequivariantstabilityconditionsmorita}}]
\label{thm:induce stab for finite groups}
Let $\cD$ be a triangulated category with an action of a finite group $G$.
Then the subset $\Stab(\cD)^G$ of $G$-invariant stability conditions on $\cD$ is a closed, complex submanifold of $\Stab(\cD)$. Moreover, there is a closed embedding 
\begin{equation*}
\begin{aligned}
\mathrm{Res}^{-1}:\Stab(\cD)^G&\hookrightarrow\Stab(\cD_G),\\
 \sigma=(Z,\cP)&\mapsto (Z\circ\mathrm{Res},\mathrm{Res}^{-1}\cP)\\
\end{aligned}
\end{equation*}
where $\mathrm{Res}\coloneq\mathrm{Res}^1_G$ and $\mathrm{Res}^{-1}\cP(\phi):=\{\cE\in\cD_G\mid\mathrm{Res}(\cE)\in\cP(\phi)\}$.
\end{theorem}

\begin{remark}
In \cite{dell2024fusionequivariantstabilityconditionsmorita}, the image of  $\mathrm{Res}^{-1}$ is characterised precisely: it consists of those stability conditions on $\cD_G$ that are $\Rep(G)$\emph{-equivariant}, when  $\Rep(G)$ is regarded as a fusion category. By \Cref{prop:g-stab}, $\Stab(X_G)$ is nonempty for any smooth projective variety $X$ with an action of a finite group $G$.
\end{remark}

\begin{lemma}
\label{Inducing_for_finite_G}
Let $\cD$ be a triangulated category with an action of a finite group $G$. Any $G$-invariant quasi-convergent path $\sigma_{\bullet}\in\mathrm{Stab}(D^b(X))$ induces a quasi-convergent path $\sigma^G_{\bullet}\in\Stab(\cD)$.
\end{lemma}
\begin{proof}
By \Cref{thm:induce stab for finite groups}, the map $\mathrm{Res}^{-1}$ sends $\sigma_{\bullet}$ to a path  is the required quasi-convergent path $\mathrm{Res}^{-1}(\sigma_{\bullet})$ in $\Stab(\cD_G)$. We claim that this induced path is quasi-convergent.

For any $\cE\in\cD_G$, consider the limit semistable filtration of $\mathrm{Res}(\cE)\in\cD$ with respect to $\sigma_{\bullet}$:
\begin{equation*}
\begin{tikzcd}
0=E_0 \arrow[rr] &                        & E_1 \arrow[ld] \arrow[r] & \dots \arrow[r] & E_{m-1} \arrow[rr] &                        & E_m=\mathrm{Res}(\cE) \arrow[ld] \\
                 & G_1 \arrow[lu, dashed] &                          &                 &                    & G_m \arrow[lu, dashed] &                 
\end{tikzcd}
\end{equation*} 
By \cite[Lemma 2.7]{halpernleistner2024stabilityconditionssemiorthogonaldecompositions}, the pair $(\cP_{t}(>\phi^+_t(G_i)),\cP_{t}(\le\phi^+_t(G_i)))$ is a t-structure on $\cD$. Hence, $E_{i-1}\in\cP_{t}(>\phi^+_t(G_i))$. Following \cite[Theorem B]{dell2024fusionequivariantstabilityconditionsmorita}, we have
\[
(\mathrm{Res}\circ\mathrm{Ind}(\cP(\phi)))\subset\cP(\phi),
\]
for any $\phi\in\mathbb R$. Consequently, by \cite[Theorem 2.1.2]{Polishchuk06}, the pair \[(\mathrm{Res}^{-1}\cP_{\bullet}(>\phi^+_t(G_i)),\mathrm{Res}^{-1}\cP_{\bullet}(\le\phi^+_t(G_i)))\] is a $t$-structure on $\cD_G$. Therefore there exists a triangle 
\[
\cE_{i-1}\to\cE_i\to\cG_i
\]
in $\cD_G$ with $\cE_{i-1}\in\mathrm{Res}^{-1}\cP_{\bullet}(>\phi_t^+(G_i))$ and $\cG_i\in\mathrm{Res}^{-1}\cP_{\bullet}(\le\phi^+_t(G_i))$. Since $\mathrm{Res}$ is exact, we obtain a triangle 
\[
\mathrm{Res}(\cE_{i-1})\to\mathrm{Res}(\cE_i)\to\mathrm{Res}(\cG_i)
\]
in $\cD_G$ with 
\[
\mathrm{Res}(\cE_{i-1})\cong E_{i-1}, \quad \mathrm{Res}(\cG_{i})\cong G_i.
\]
 In particular, $\cG_i\in\mathrm{Res}^{-1}\cP_{\bullet}(\phi^+_t(G_i))$, which implies $\cG_i$ is limit semistable with respect to $\mathrm{Res}^{-1}(\sigma_{\bullet})$. Hence, we obtain a limit semistable filtration of $\cE$ in $\cD_G$:
\begin{equation*}
\begin{tikzcd}
0=\cE_0 \arrow[rr] &                        &\cE_1 \arrow[ld] \arrow[r] & \dots \arrow[r] &\cE_{m-1} \arrow[rr] &                        &\cE \arrow[ld] \\
                 &\cG_1 \arrow[lu, dashed] &                          &                 &                    & \cG_m \arrow[lu, dashed] &                 
\end{tikzcd}
\end{equation*} 
Given two limit semistable objects $\cE,\cF\in\cD_G$ with respect to $\mathrm{Res}^{-1}(\sigma_{\bullet})$, set $E=\mathrm{Res}(\cE)$, $F=\mathcal{F}$. Then, 
\begin{equation*}
\begin{aligned}
\lim\limits_{t\to\infty}\frac{l_t(\cE|\cF)}{1+|l_t(\cE|\cF)|}&=\lim\limits_{t\to\infty}\frac{\ln(Z_t\circ\mathrm{Res}(\cE))-\ln(Z_t\circ\mathrm{Res}(\cE))+i\pi(\phi_t(\cE)-\phi_t(\cF))}{1+|\ln(Z_t\circ\mathrm{Res}(\cE))-\ln(Z_t\circ\mathrm{Res}(\cE))+i\pi(\phi_t(\cE)-\phi_t(\cF))|}   \\
&=\lim\limits_{t\to\infty}\frac{\ln(Z_t(E))-\ln(Z_t(E))+i\pi(\phi_t(E)-\phi_t(F))}{1+|\ln(Z_t(E))-\ln(Z_t(E))+i\pi(\phi_t(E)-\phi_t(F))|}\\
&=\lim\limits_{t\to\infty}\frac{l_t(E|F)}{1+|l_t(E|F)|}\\
\end{aligned}    
\end{equation*}
Since $E=\mathrm{Res}(\cE)$ and $F=\mathrm{Res}(\cF)$ are limit semistable with respect to $\sigma_{\bullet}$, the last limit exists. Hence, $\mathrm{Res}^{-1}(\sigma_{\bullet})$ satisfies the second condition of quasi-convergence as well, completing $\Stab(\cD_G)$.
\end{proof}

\begin{lemma}
\label{lem: geometric preserving for finite G}
Let $X$ be a smooth projective variety with an action of a finite group $G$, and let $\sigma=(Z,\cP)\in\Stab(X)^G$ be a $G$-invariant stability condition.
\begin{enumerate}
\item If $\cO_x$ is $\sigma$-stable of phase $\phi_0$, then the $G$-point sheaf $P_x$ is $\mathrm{Res}^{-1}(\sigma)$-stable with the same phase $\phi_0$;
\item If $P_x$ is $\mathrm{Res}^{-1}(\sigma)$-stable of phase $\phi_0$ and every subobject $A$ of $\cO_x$ in $\cP(\phi_0)$ satisfies $g^{\ast}A\cong A$ for all $g\in G_x$, then $\cO_x$ is $\sigma$-stable. 
\end{enumerate}
In particular, $\mathrm{Res}^{-1}$ preserves geometricity.
\end{lemma}

\begin{proof}
Assume $\cO_x$ is $\sigma$-stable of the phase $\phi_0$. Since $\sigma$ is $G$-invariant, each $g^{\ast}$ is also stable of phase $\phi_0$ for all $g\in G$. Hence,
\[
\mathrm{Res}(P_x)=\bigoplus_{g\in G/G_x}g^{\ast}\cO_x\in\cP(\phi_0),
\]
which implies $P_x\in\mathrm{Res}^{-1}\cP(\phi_0)$. Suppose $P_x$ is not stable. Then there is a short exact sequence in $\mathrm{Res}^{-1}\cP(\phi_0)$:
\[
0\to\cE\to P_x\to\cF\to0,
\]
where $\cE$ is neither $0$ nor $P_x$. Applying $\mathrm{Res}$ gives a short exact sequence 
\[
0\to\mathrm{Res}(\cE)\to\bigoplus_{g\in G/G_x}g^{\ast}\cO_{x}\to\mathrm{Res}(\cF)\to0
\]
in $\mathrm{Res}^{-1}\cP(\phi_0)$. Since $\mathrm{Res}(\cE)$ is a $G$-invariant subobject of $\oplus_{g\in G/G_x}g^{\ast}\cO_x$, which is a direct sum of distinct skyscraper sheaves, $\mathrm{Res}(\cE)$ must be either $0$ or $\oplus_{g\in G/G_x}g^{\ast}\cO_x$, which is a contradiction. Thus, $P_x$ is stable.

Conversely, suppose $P_x$ is $\mathrm{Res}^{-1}(\sigma)$-stable of phase $\phi_0$. Then $\mathrm{Res}(P_x)=\oplus_{g\in G/G_x}g^{\ast}\cO_x\in\cP(\phi_0)$, and since $\cP(\phi_0)$ is closed under direct summands, $\cO_x\in\cP(\phi_0)$.  If $\cO_x$ is not stable, there exists a short exact sequence in $\cP(\phi_0)$:
\[
0\to A\to \cO_x\to B\to 0
\]
with $A\ne0$ and $A\ne\cO_x$. By assumption, $g^{\ast}A\cong A$ for all $g\in G_x$. Hence, we may endow $A$ with the trivial $G_x$-action, making $A^{G_x}$ a $G_x$-equivariant subobject of $\cO_x^{G_x}$ in $D^b_{G_x}(X)$. Applying $\mathrm{Ind}^G_{G_x}(A^{G_x})$ yields a nonzero proper subobject 
\[
\mathrm{Ind}^G_{G_x}(A^{G_x})\hookrightarrow\mathrm{Ind}^G_{G_x}(\cO^{G_x}_x)=P_x
\]
in $\mathrm{Res}^{-1}\cP(\phi_0)$. This contradicts the stability of $P_x$. Hence, $\cO_x$ must be stable. 

The final statement follows directly from the two parts.
\end{proof}

\begin{lemma}
\label{lem:finite G-chomology}
Let $X$ be a smooth projective variety with an action of a finite group $G$. Then,
\begin{enumerate}
\item $H^{\ast}_G(X;\mathbb C)= H^{\ast}(X;\mathbb C)^G$;
\item The restriction of the Poincar\'{e} pairing $\eta$ to $H^{\ast}(X;\mathbb C)^G$ is nondegenerate.
\end{enumerate}
\end{lemma}
\begin{proof}
 Consider the fibration
\[
X\to X\times_G EG\to BG.
\]
Its Leray-Serre spectral sequence satisfies:
\[
E_2^{p,q}=H^p(BG;H^q(X))\Rightarrow H_G^{p+q}(X).
\]
Since $G$ is finite, we have
\[
E_2^{p,q}=
\begin{cases}
H^q(X;\mathbb C)^G&\text{ if } p=0;\\
0&\text{ if } p>0.
\end{cases}
\]
Thus the spectral sequence collapses at the $E_2$ page, which implies that 
\[
H^{\ast}_G(X)= H^\ast(X;\mathbb C)^G.
\]
Consider the averaging map 
\[
\mathrm{Av}\colon H^{\ast}(X)\to H^{\ast}(X)^G,\quad \mathrm{Av}(\alpha)=\frac{1}{|G|}\sum_{g\in G}g^{\ast}\alpha.
\] 
If $\alpha\in H^{\ast}(X;\mathbb C)^G$ satifies $\eta(\alpha,\beta)_X=0$ for every $\beta\in H^{\ast}(X;\mathbb C)^G$. Then for any $\gamma\in H^{\ast}(X;\mathbb C)$,
\[
\eta(\alpha,\gamma)_X=\eta(\mathrm{Av}(\alpha),\gamma)_X=\eta(\alpha,\mathrm{Av}(\gamma))_X=0.
\]
Since $\eta$ is nondegenerate on $H^{\ast}(X;\mathbb C)$, $\alpha=0$. Thus, $\eta$ remains nondegenerate on $H^{\ast}(X;\mathbb C)^G$.
\end{proof}

\begin{proof}[Proof of \Cref{thm: inducing solutions}]
Let $\alpha,\beta\in H^{\ast}_G(X;\mathbb C)$. By \Cref{lem:finite G-chomology}, 
\[
\eta^G(T^G_d(\alpha),\beta)=\eta(T^G_d(\alpha),\beta),
\]
and 
\[
\eta^G(T^G_d(\alpha),\beta)=\langle\alpha,\beta\rangle^{X,G}_{0,2,d}=\langle\alpha,\beta\rangle^{X}_{0,2,d}
=\eta(T_d(\alpha),\beta).
\]
By \Cref{lem:finite G-chomology}, the nondegeneracy of $\eta$ on $H^{\ast}(X;\mathbb C)^G$ forces $T_d^G=T_d\vert_{H^{\ast}(X;\mathbb C)^G}$. Consequently, for $\alpha\in H^{\ast}_G(X;\mathbb C)$,
\[
\begin{aligned}
E^G_{\psi}(t)(\alpha)&=c_1^G(X)\cup\alpha+\sum_{\substack{d\in\mathrm{NE}(X/Y)_{\mathbb N}\\ (c_1(X)-\omega)\cdot d>0}}(c_1(X)\cdot d)t^{c_1(X)\cdot d}e^{-\psi\cdot d}T^G_d(\alpha)\\
&=c_1(X)\cup\alpha+\sum_{\substack{d\in\mathrm{NE}(X/Y)_{\mathbb N}\\ (c_1(X)-\omega)\cdot d>0}}(c_1(X)\cdot d)t^{c_1(X)\cdot d}e^{-\psi\cdot d}T_d(\alpha)\\
&=E_{\psi}(t)(\alpha).
\end{aligned}
\]
Therefore, $\Phi^G_t\coloneq\Phi_t\vert_{H^{\ast}_G(X;\mathbb C)}$ is a fundamental solution of the $G$-truncated quantum differential equation \eqref{eq:truncated qde2}, and clearly 
\[
\int_X^{eq}\Phi^G_t(\alpha)
=\int_X\Phi_t(\alpha)
=Z(\alpha).
\]
 Let $\sigma^G_t\coloneq\mathrm{Res}^{-1}\sigma_{t}$. By \Cref{Inducing_for_finite_G}, $\sigma^G_t$ forms a quasi-convergent path in $\Stab(X_G)$ which solves \Cref{proposal:main} via $\Phi^G_t$. 

Now let $r\in\mathbb R$ be the real part of an eigenvalue of $\frac{-1}{z}E_{\psi}^G(1)$. observe that
\[
F^r_G=F^r_1\cap H^{\ast}(X;\mathbb C)^G.
\]
Take $\alpha\in F^r_G$. By the spanning condition for $H^{\ast}(X;\mathbb C)$, there exist constants $c_1,\dots,c_k\in\mathbb C$ and limit semistable objects $E_1,\dots,E_k\in D^b(X)$ with respect to $\sigma_t$ such that 
\[
\mathop{\lim\inf}\limits_{t\to\infty}\frac{|Z_t(E_i)|}{\|\Phi_t(v(E_i))\|} \quad\text{ for all } i,
\]
and 
\[
\alpha=\sum_{i=1}^kc_iv(E_i).
\]
Let $\cE_i\coloneq\mathrm{Ind}^G(E_i)=(\oplus_{g\in G}\phi_g^{\ast}E_i,\{\theta_h\}_{h\in G})$. By \Cref{Inducing_for_finite_G}, each $\cE_i$ is limit semistable with respect to $\sigma^G_t$. Since $\alpha\in H^{\ast}(X;\mathbb C)^G$ is $G$-invariant, 
\[
\alpha=\sum_{i=1}^kc_i\left(\frac{1}{|G|}\sum_{g\in G}\phi_g^{\ast}v(E_i)\right)=\sum_{i=1}^kc_iv^G(\cE_i).
\]
Moreover,
\[
\mathop{\lim\inf}\limits_{t\to\infty}\frac{|Z^G_t(\cE_i)|}{\|\Phi^G_t(v^G(\cE_i))\|}=\mathop{\lim\inf}\limits_{t\to\infty}\frac{|Z_t(E_i)|}{\|\Phi_t(v(E_i))\|}>0.
\]
Thus the classes $v^G(\cE_i)$ span $F^r_G$, and the pair $(\sigma^G_t,\Phi^G_t)$ satisfies the spanning condition with respect to $\Phi^G_t$. By \Cref{lem: geometric preserving for finite G}, $\sigma^G_{t_0}$ is geometric whenever $\sigma_{t_0}$ is geometric.
\end{proof}

\begin{corollary}[\cite{krug2024endomorphismalgebrasequivariantexceptional}]
\label{Exceptional-G-exceptional}
Suppose $\cD$ admits a (strong) full exceptional collection of the form
\[
\cD=\langle E_{1,1},\dots,E_{1,l_1},E_{2,1},\dots,E_{2,l_2},\dots,E_{k,1},\dots,E_{k,l_k}\rangle
\]
such that a finite group $G$ acts transitively on each block $(E_{i,1},\dots,E_{i,l_l})$, i.e. there is a transitive $G$-action on the index set $\{1,\dots,l_i\}$ with $g_{\ast}E_{i,l_i}\cong E_{i,g(l_i)}$. Let $H_i$ be the stabilizer of the first member $E_i:=E_{i,1}$ of the $i$-th block. Assume that there exists $\cE_i=(E_i,(\theta_h))\in\cD_{H_i}$ with 
\[
\mathrm{Res}^1_{H_i}(\cE_i)=E_i=E_{i,1}.
\]
Then there is an induced (strong) full exceptional collection on $\cD_G$, namely
\[
\left\langle\mathrm{Ind}^G_{H_1}(\rho\otimes_{\k}\cE_1)_{\rho\in\mathrm{Irr}(H_1)},\dots,\mathrm{Ind}^G_{H_k}(\rho\otimes_{\k}\cE_k)_{\rho\in\mathrm{Irr}(H_k)}\right\rangle,
\]
where $\mathrm{Irr}(H_i)$ denotes the set of irreducible representation of $H_i\subset G$.
\end{corollary}

\begin{proof}
According to \Cref{SODsfromQuasipath}, there is quasi-convergent path $\sigma_{\bullet}\in\Stab(\cD)$ which recovers the full exceptional collection $\langle E_{1,1}.\dots,E_{k,l_k} \rangle$.
Using the gluing construction of quasi-convergent paths from \Cref{SODsfromQuasipath}, we may choose $\sigma_{\bullet}$ to be $G$-invariant.

 Let $\cE\in\mathrm{Res}^{-1}\cP_{\bullet}$. By \Cref{Inducing_for_finite_G}, the object $E=\mathrm{Res}(\cE)$ is limit semistable with respect to $\sigma_{\bullet}$. Thus, $\cD^{E}\cong E_{i,j}$ for some $1\le i\le k$ and $1\le j\le l_i$. Then   
\begin{equation*}
\begin{aligned}
\cD_G^{\cE}\subset\mathrm{Ind}(\cD^{E_{i,j}})&\cong\mathrm{Ind}(\cD^{E_{i,1}})\\
&\cong   \mathrm{Ind}_{H_i}^{G}\mathrm{Ind}^{H_i}_1\mathrm{Res}^1_{H_i}(\cE_i)\\
&\cong   \mathrm{Ind}^G_{H_i}(\k\langle H_i\rangle\otimes_{\k}\cE_i)\\
&\cong   \bigoplus_{\rho\in\mathrm{Irr}(H_i)}\mathrm{Ind}((\rho\otimes_{\k}\cE_i)^{\oplus m_{i,\rho}}),
\end{aligned}    
\end{equation*}
where $\k\langle H_i\rangle$ is the regular representation of $H_i$, and 
\[
m_{i,\rho}=\frac{\dim_{\k}\rho}{\dim_{\k}\mathrm{End}_{H_i}(\rho)}.
\]
On the other hand, $\mathrm{Res}(\cD_G^{\cE})=\cD^{E}\cong E_{i,j}$. Since $\mathrm{Res}$ is faithful, there exists $\rho\in\mathrm{Irr}(H_i)$ such that $\cD^{\cE}_G\cong\mathrm{Ind}(\rho\otimes\cE_i)$. Consequently, $\mathrm{Res}^{-1}(\sigma_{\bullet})$ is the required quasi-convergent path in $\Stab(\cD^G)$ which recovers $\left\{\mathrm{Ind}^G_{H_i}(\rho\otimes_{\k}\cE_i)_{\rho\in\mathrm{Irr}(H_i)}\right\}$.
\end{proof}

\subsection{Applications}
\label{subsec:application of inducing}
\subsubsection{Projective spaces}
Let $V$ be an $m$-dimensional vector space over $\mathbb C$ and $\mathbb P(V)\cong\bP^{m-1}$ its projectivization. The derived category of $D^b(\bP^{m-1})$ admits a full exceptional collection
\[
\{\cO,\cO(1),\dots,\cO(m-1)\}.
\]
Let $G$ be a finite linearly acting on $\bP(V)$, i.e. there is a group homomorphism
\[
\rho:G\to\GL(V)
\]
such that for $g\in G$ and $[v]\in\bP(V)$ (where $v\ne0\in V$), 
\[
g\cdot[v]=[\rho(g)(v)].
\]
Each $\cO(i)$ then inherits a natural $G$-equivariant structure. Let $V_1,\dots,V_m$ be the irreducible representations of $G$ over $\mathbb C$. By \Cref{Exceptional-G-exceptional}, there is an induced strong semiorthogonal decomposition on $D^b_G(\bP^{m-1})$:
\[
\begin{pmatrix}
\cO\otimes V_1&\cO(1)\otimes V_1&\cdots&\cO(m-1)\otimes V_1\\
\vdots&\vdots&\ddots&\vdots\\
\cO\otimes V_m&\cO(1)\otimes V_m&\cdots&\cO(m-1)\otimes V_m
\end{pmatrix}
\]
satsifying
\begin{enumerate}
\item for $i<j$, 
\[
\Hom^k_G(\cO(j)\otimes V_p,\cO(i)\otimes V_q)=0 \text{ for any } k.
\]
\item for $i=j$,
\[
\Hom^k_G(\cO(i)\otimes V_p,\cO(i)\otimes V_q)=
\begin{cases}
\mathbb C &\text{ for } p=q,k=0;\\
0 &\text{ else.}
\end{cases}
\]
\end{enumerate}

We briefly recall the construction of quasi-convergent paths and fundamental solutions for the non-equivariant \Cref{proposal:main} on $\bP^{m-1}$, following \cite{zuliani2024semiorthogonaldecompositionsprojectivespaces}. Denote by $B$ the open locus of $H^{\ast}(\bP^{m-1};\mathbb C)$ such that for any $\tau\in B$, the summation in $\star_{\tau}$ is convergent. See \cite{manin1999frobenius} for the proof of the non-emptyness of $B$. The quantum differential equation is defined via a meromorphic connection $\tilde{\Delta}$ on the trivial bundle 
\[
H^{\ast}(\bP^{m-1};\mathbb C)\times (B\times\bP^1)\to B\times\bP^1.
\]
In coordinates $(\tau,z)$, $\tilde{\Delta}$ is defined by
\[
\tilde{\triangledown}_{\partial_{\alpha}}=\partial_{\alpha}+\frac{1}{z}(\alpha\star_{\tau}),\quad \alpha\in T_{\tau}B\cong H^{\ast}(\bP^{m-1};\mathbb C),
\]
\[
\tilde{\triangledown}_{z\partial_z}=z\partial_z-\frac{1}{z}(E\star_{\tau})+\mu,
\]
where $E$ is the Euler vector field and $\mu$ is the grading operator on $H^{\ast}(\bP^{m-1};\mathbb C)$. According to \cite{10.1215/00127094-3476593}, for a fixed $\tau\in B$, the canonical fundamental solution $\Phi$ of $\tilde{\triangledown}_{|_{\tau}}$ induces an isomorphism
\[
\begin{aligned}
\Phi\colon H^{\ast}(\bP^{m-1};\mathbb C)&\xrightarrow{\cong}\left\{s\colon\mathbb R_{>0}\to H^{\ast}(\bP^{m-1};\mathbb C)\mid \tilde{\triangledown}_{|_{\tau}}s=0\right\},\\
\end{aligned}
\]
given by 
\[
\Phi(\alpha)(z)\coloneq(2\pi)^{\frac{1-m}{2}}S(\tau,z)z^{-\mu}z^{\rho}\alpha,
\]
where $S(\tau,-)$ is a specific holomorphic function on $\mathbb R_{>0}$ and $\rho=(c_1(\bP^{m-1})\cup)\in\End(H^{\ast}(\bP^{m-1};\mathbb C))$. By \cite[Theorem 5.15]{zuliani2024semiorthogonaldecompositionsprojectivespaces}, there exists a quasi-convergent path $\sigma_t=(Z_t,\cP_t)\in\Stab(\bP^{m-1})$ for $t\in\mathbb R$ whose central charge $Z_t$ is the \emph{quantum cohomology central charge}
\[
Z_{t}(E)\coloneq(2\pi t)^{\frac{m-1}{2}}\int_{\bP^{m-1}}\Phi(\hat{\Gamma}\cdot\mathrm{Ch}(E)).
\]
with $\hat\Gamma$ the Gamma class of $\bP^{m-1}$ and $\mathrm{Ch}$ the modified Chern character defined as
\[
\mathrm{Ch}\coloneq(2\pi i)^{\deg}\ch.
\]
Applying \Cref{thm: inducing solutions} yields a quasi-convergent path $\sigma_t^G=(Z_t^G,\cP_t^G)$ for $t\in\mathbb R_{>0}$ with the $G$-equivariant quantum chohomology central charge $Z_t^G$. In the special case of the projective plane $\mathbb P^2$, combining \cite[Theorem 5.17]{zuliani2024semiorthogonaldecompositionsprojectivespaces} with \Cref{Exceptional-G-exceptional} gives
the induced semiorthogonal decomposition is given by 
\[
\begin{pmatrix}
\cO\otimes V_1&\cO(1)\otimes V_1&\cO(2)\otimes V_1\\
\cO\otimes V_2&\cO(1)\otimes V_2&\cO(2)\otimes V_2\\
\cO\otimes V_3&\cO(1)\otimes V_3&\cO(2)\otimes V_3
\end{pmatrix}
\]
as $r\to0^+$.

\subsubsection{Blow-up surfaces}
Let $X$ be a smooth del Pezzo surface of degree $d$ $(1\le d\le9)$ given by blowing up $\bP^2$ at $r$ distinct points $x_1,\dots,x_r\in\bP^2$ in general position. Write $\pi:X\to\mathbb P^2$ for the blow-up morphism and let $E_i=\pi^{-1}(x_i)$ be the exceptional divisors. According to \cite{DOOrlov_1993} the derived category of coherent sheaves on $X$ has the following full exceptional collection
\[
\begin{pmatrix}
\cO_{E_1}(-1)&&&\\
\vdots&\pi^{\ast}\cO_{\bP^2}&\pi^{\ast}\cO_{\bP^2}(1)&\pi^{\ast}\cO_{\bP^2}(2)\\
\cO_{E_r}(-1)&&&
\end{pmatrix},
\]
where the sheaves $\cO_{E_1}(-1),\dots,\cO_{E_r}(-1)$ are mutually orthogonal. Let $G$ be a finite group acting on $X$ satisfying the conditions in \Cref{Exceptional-G-exceptional}. For example, take $G$ to be a finite subgroup of $\mathrm{PGL}(3)$ that permutes the $r$ points. In this case, the $G$-action on the base $\bP^2$ is inherited from the original $\mathrm{PGL}(3)$-action, and for $g\in G$, we have
\[
g(E_i)=E_j, \quad \text{ if } g(x_i)=x_j.
\]
Therefore, $\pi^{\ast}\cO_{\bP^2}(k)$ inherits a $G$-equivariant structure via the pullback of the $\mathrm{PGL}(3)$-action on $\bP^2$. Meanwhile, the $G$-equivariant structure on $\cO_{E_i}(-1)$ is given by the natural isomorphism
\[
\theta_g\colon\cO_{E_i}(-1)\to g^{\ast}\cO_{E_{g(i)}}(-1).
\]
By \Cref{Exceptional-G-exceptional}, this yields an induced semiorthgonal decomposition of $D^b_G(X)$. According to \cite{Dolgachev2009}, for $4\le r\le8$, the automorphism group $\Aut(X)$ is finite, and such a group $G$ exists for suitable configurations of these $r$ points. 

For simplicity, we restrict to the case where $\pi:X\to Y$ is the blow-up of a single point on a del Pezzo surface $Y$ of degree $d\ge2$ to study the \Cref{proposal:main}. Denote by $E$ the exceptional divisor. The construction of solutions to the non-equivariant \Cref{proposal:main} for blow-up surfaces is due to \cite{karube2024noncommutativemmpblowupsurfaces}. A general class $\zeta(t)\in H^{\ast}(X;\mathbb C)$ can be written as
\[
\zeta(t)=(\zeta_0(t),\pi^{\ast}\zeta_1(t)+c(t)E,\zeta_2(t)),
\]
where $\zeta_i(t)\in H^{2i}(X;\mathbb C)$ and $c(t)\in\mathbb C$. The truncated quantum differential then becomes:
\[
t\frac{d\zeta(t)}{dt}=\frac{1}{z}(0,\pi^{\ast}K_{Y}\zeta_0+(\zeta_0+e^{-\psi E}tc)E,K_Y\zeta_1-c),
\]
where $K_Y$ is the canonical class of $Y$. By \cite{karube2024noncommutativemmpblowupsurfaces}, there exists a quasi-convergent path of stability conditions $\sigma_t=(Z_t,\cP_t)$ for $t\in[t_0,\infty)$, whose central charge takes the form
\[
Z_t(F)=w(\lambda,Z_0)(\mathrm{Ei}(\lambda t)-\mathrm{Ei}(\lambda t_0))\ch_1(F)\cdot E+Z_0(e^{-sE}\ch(F)),
\]
with
\begin{enumerate}
\item $Z_0=Z_{t_0}$ is the initial central charge, chosen from a specific chamber;
\item $\mathrm{Ei}(z)$ is the exponential integral function, defined as
\[
\mathrm{Ei}(z)\coloneq-\int^{\infty}_{-z}\frac{e^{-t}}{t}dt;
\]
\item $\lambda\in\mathbb R\oplus i[-1,1]$ and $s\in\mathbb R$;
\item $w(s,\lambda,Z_0)$ is certain normalization function.
\end{enumerate}
Applying \Cref{thm: inducing solutions} now produces a  quasi-convergent path $\sigma_t^G=(Z_t^G,\cP_t^G)$ in $\Stab(X_G)$ for $t\in[t_0,\infty)$ that fulfills the requirements of \Cref{proposal:main}.

\section{Quasi-convergent paths of $\bT$-stability conditions}
\label{sec: paths in T-stab}
In this section, we first introduce the notion of $\bT$-stability conditions on $\bT$-categories, which generalizes the q-stability conditions on $\mathbb X$-categories in \cite{Ikeda_Qiu_2023}, and establish the deformation property of spaces of $\bT$-stability conditions. We then construct quasi-convergent paths in spaces of $\bT$-stability conditions to fulfill the requirements of \Cref{proposal:main} for equivariant projective spaces from small quantum cohomology.

\subsection{$\bT$-stability conditions on $\bT$-categories}
\label{sec:T-stab}
A $\bT$\emph{-category}, denoted by $\cD_{\bT}$, consists of a triangulated category $\cD$ together with $m$ commuting auto-equivalences
\[
T_i:\cD_{\bT}\to\cD_{\bT}, \quad i=1,\dots,m.
\] 
We write $E[\sum_{i=1}^m l_i T_i]:=T_1^{l_1}\circ\cdots\circ T_m^{l_m}(E)$ for $l_i\in\mathbb Z$ and $E\in\cD_{\bT}$. Set 
\[
R(\bT)=\mathbb Z[T_1^{\pm},\dots,T_m^{\pm}]
\]
and define the $R(\bT)$-action on $K(\cD_{\bT})$ by 
\[
T_{i}^l\cdot[E]:=[E[lT_i]].
\]
Consequently, $K(\cD_{\bT})$ is endowed with the structure of an $R(\bT)$-module. In this context, $\mathrm{Aut}(\cD_{\bT})$ of auto-equivalences is taken to consist of auto-equivalences that commute with $\{T_i\}$. 

We now provide two examples of $\bT$-categories.
\begin{example}[$\bT$-graded algebras]
\label{ex:T-algebra}
A \emph{differential} $\bT$-\emph{graded (mdg) algebra} $A$ is described as an algebra
\[
A:=\bigoplus_{l_i\in\mathbb Z}A^{l+\sum_{i=1}^{m}l_iT_i}
\]   
graded by $\mathbb Z T_1\oplus\mathbb ZT_2\cdots\oplus\mathbb ZT_m$ with differential $d:A^{l+\sum_{i=1}^{m}l_iT_i}\to A^{l+1+\sum_{i=1}^{m}l_iT_i}$ of degree $1$. Let $M=\oplus_{m_i\in\mathbb Z}M^{l+\sum_{i=1}^{m}l_iT_i}$ be a $\bT$-graded modules over $A$ with the differential $d$. The degree shift \[M'=M[k+\sum_{i=1}^mk_iT_i]\] is defined by 
\[
(M')^{l+\sum_{i=1}^{m}l_iT_i}:=M^{l+k+\sum_{i=1}^m(l_i+k_i)T_i}
\]
with differential $d':=(-1)^kd$. Then we obtain $\bT$-categories:
\begin{enumerate}
\item $D(A)$ the derived category of $\bT$-graded modules over $A$;
\item $\mathrm{Perf}\,A$ the \emph{perfect derived category} as the smallest full triangulated subcategory of $D(A)$ consisting of $A$ and closed under taking direct summands.
\end{enumerate}
The \emph{Calabi-Yau}-$\bT$ \emph{completion} of $A$ is defined by
\[
\Pi_{\bT}(A):=T_{A}(\theta)=A\oplus\theta\oplus(\theta\otimes_A\theta)\oplus\cdots
\]
where $\theta:=\Theta_A(T_1+\cdots+T_m-1)$ and $T_A(\theta)$ is the tensor algebra of $\theta$ over $A$. We now assume that $A$ is a homologically smooth $\bT$-graded algebra and $\Theta_{A}$ be its inverse dualizing complex. According to \cite[Theorem 6.3]{Keller:2011nql}, the Calabi-Yau-$\bT$ completion $\Pi_{\bT}(A)$ becomes a \emph{Calabi-Yau-}$\bT$ \emph{algebra} in the sense that $A$ is 
\[
\Theta_A\xrightarrow{\sim}A[T_1+\cdots+T_m]
\]
Furthermore, $D_{fd}(\Pi_{\bT}(A))$ is \emph{Calabi-Yau-}$\bT$ category, i.e. for any objects $X,Y$ in $\cD$ we have a natural isomorphism 
\[
\Hom(X,Y)\xrightarrow{\sim}\mathrm{DHom}(Y,X[T_1+\cdots+T_m]).
\]
\end{example}

\begin{example}[Motivating example: $\bT$-equivariant sheaves]
\label{ex:T-sheaf}
Let $X$ be a smooth projective variety with an action of the torus $\bT=(\mathbb C^{\ast})^m$. By definition, the ring of finite dimensional complex representations of $\bT$ is 
\[
R(\bT)=\mathbb C[T_1^{\pm},\dots,T_m^{\pm}],
\]
where each $T_i$ corresponds to the $i$-th coordinate character of $\bT$. The structural morphism $\pi:X\to\mathrm{Spec}(\mathbb C)$ induces two morphisms 
\begin{equation*}
\begin{aligned}
  \pi_{\ast}:K(X)&\to K(\mathrm{Spec}\,\mathbb C)\cong\mathbb Z\\
  E&\mapsto\sum(-1)^i\mathrm{rk} H^i(X,E),
\end{aligned}    
\end{equation*}
and 
\begin{equation*}
\begin{aligned}
  \pi_{\ast}^G:K^\bT(X)&\to R(\bT)\\
  \cE&\mapsto\sum(-1)^i[H^i(X,\cE)],
\end{aligned}    
\end{equation*}
where $[H^i(X,\cE)]$ denotes the $R(\bT)$-class of the cohomology space $H^i(X,\cE)$. Then the Grothendieck group $K_\num(X)$ and the equivariant Grothendieck group $K_{\num}^\bT(X)$ with a $\mathbb C$-algebra and an $R(\bT)$-algebra structures, respectively, and we have the following commutative diagram
\[
\begin{tikzcd}
\mathcal D_{\bT}^b(X) \arrow[d, "F_{\bT}"'] \arrow[r, "R\pi_{\ast}^{\bT}"] & \mathcal D^b(\Rep(\bT)) \arrow[d, "F"] \\
\mathcal D^b(X) \arrow[r, "R\pi_{\ast}"]                       & \mathcal D^b(\mathbb C),      
\end{tikzcd}
\]
where $F,F_{\bT}$ are the corresponding forgetful functors. For any $\cE\in D^b_{\bT}(X)$, the action of $T_i$ on $\cE$ is defined by
\[
T_i\cdot\cE:=\cE\otimes T_i,
\]
making $\{T_i\}$ $m$ commuting auto-equivalences in $D^b_{\bT}(X)$. Thus $D_{\bT}^b(X)$ equipped with $\{T_i\}$ provides a $\bT$-category. For simplicity, we denote it just by $D^b_{\bT}(X)$.  Finally, let $G$ be a reductive group acting on $X$ and $\bT\subset G$ be a maximal torus with Weyl group $W$. Using the isomorphisms
\[
H^{\ast}_G(X)\cong H^{\ast}_{\bT}(X)^W,\quad K^G(X)_{\mathbb Q}\cong K^{\bT}(X)_{\mathbb Q}^W,
\]
we can regard $D^b_G(X)$ as a $\bT$-category via restriction to $\bT$ together with the $W$-invariance structure. 
\end{example}

\begin{definition}
\label{def:T-stab}
A \emph{pre-$\bT$-stability condition} $(\sigma,\s)$ on a $\bT$-category $\cD_{\bT}$ consists of a (Bridgeland) pre-stability condition $\sigma=(Z,\cP)$ on $\cD_{\bT}$ and a tuple $\s=(s_1,\dots,s_m)\in\mathbb C^m$ satisfying 
\begin{equation}
\label{eq:Gepner}
T_i(\sigma)=s_i\cdot\sigma, \text{ for } i=1,\dots,m.
\end{equation}
\end{definition}

\begin{remark}
When $m=1$, \eqref{eq:Gepner} is the familier Gepner point, which has been studied in \cite{Toda1412,10.1093/imrn/rnv125}. We therefore refer to \eqref{eq:Gepner} a \emph{multi-Gepner} point. 
\end{remark}

For a fixed tuple $\s\subset\mathbb C$, define the \emph{evaluation map}
\[
\overline{\ev}_{\s}:R(\bT)\to\mathbb C, \quad T_i \mapsto e^{\i\pi s_i}.
\]

To spell out the conditions for \eqref{eq:Gepner}, we provide the following equivalent definition of $\bT$-stability condition:

\begin{definition}
A pre-$\bT$-stability condition $(\sigma,\s)$ consists of a pre-stability condition $\sigma=(Z,\cP)$ on $\cD_{\bT}$ and a tuple $\s=(s_1,\dots,s_m)\in\mathbb C^m$ of complex numbers satisfying:
\begin{enumerate}
\item the slicing satisfies $\cP(\phi+\Re(s_i))=\cP(\phi)[T_i]$  for all $\phi\in\mathbb R$;
\item the central charge $Z:K(\cD_{\bT})\to\mathbb C$ is $R(\bT)$-linear and factors through $\overline\ev_{\s}$.
\end{enumerate} 
\end{definition}

For the remainder of the paper, we make the following assumption:
\begin{assumption}
\label{as:free of K}
The Grothendieck group $K(\cD_\bT)$ is free of finite rank over $R(\bT)$, i.e. $K(\cD_{\bT})\cong R(\bT)^n$ for some $n\in\mathbb N$.  
\end{assumption}

We note that \Cref{as:free of K} holds automatically in the context of \Cref{ex:T-algebra}. In the setting of \Cref{ex:T-sheaf}, where $X$ is a smooth projective with an action of $\bT$ and let $\pi:X\to\mathrm{Spec}(\mathbb C)$ be the structural morphism discussed in \Cref{ex:T-sheaf}, the {equivariant Grothendieck-Euler-Poincar\'{e} characteristic of an object} $\cE\in D^b_{\bT}(X)$ is defined as 
\[
\chi^{\bT}(\cE):=\pi_{\ast}^{\bT}(\cE),
\]
and the corresponding {equivariant Grothendieck-Euler-Poincar\'{e} paring} is given by
\[
\chi^{\bT}(\cE,\cF):=\chi^{\bT}(\cE^{\ast}\otimes \cF),
\]
for any $\cE,\cF\in D^b_{\bT}(X)$. The equivariant numerical Grothendieck group $K_{\num}^{\bT}(X)$ is defined as the quotient of $K(X)$ by the radical of the form $\chi^{\bT}(-,-)$. As indicated in \Cref{sec:pre on stab}, the Grothendieck group $K^{\bT}(X)$ is taken to be the numerical Grothendieck group $K_\num^\bT(X)$. 
Recall that the equivariant formality result of \cite{GKM97} implies that $H^{\ast}_\bT(X)$ is a free module over $H^{\ast}_\bT(\mathrm{pt})$ of rank equal to $\dim H^{\ast}(X)$. Furthermore, one can also show that $K^{\bT}(X)$ is a free $R(\bT)$-module and its rank equals $\mathrm{rk}\,K(X)$. Thus \Cref{as:free of K} is satisfied for $D^b_{\bT}(X)$.

Let $\|\cdot\|$ be a fixed norm on $\Gamma\otimes_{\mathbb Z}\mathbb R$. Under the \Cref{as:free of K}, we extend the norm $\|\cdot\|$ to $K(\cD_{\bT})$ by setting
\[
\|\sum_{i=1}^{l}T_1^{k_{i,1}}\cdots T_m^{k_{i,m}}\alpha_i\|:=\sum_{i=1}^l|e^{\i\pi(k_{i,1}s_1+\cdots+k_{i,m}s_m)}|\|\alpha_i\|,
\]
for $\alpha_i\in\Gamma$.

Using the \eqref{support_property}, we can now define the analogue support property for pre-$\bT$-stability conditions:
\begin{definition}
  \label{def:support property for T-stab}
  A pre-$\bT$-stability condition \((\sigma, \s)\) satisfies the $\bT$\emph{-support property} if there exists a constant $C_{\sigma}>0$ such that for every $\sigma$-semistable object $E$ with $[E]=\alpha\in K(\cD_{\bT})$,
\[
\|\alpha\|\le C_{\sigma}\cdot|Z(\alpha)|.
\]
\end{definition}

A pre-$\bT$-stability condition $(\sigma,\s)$ on $\cD_{\bT}$ is called a $\bT$-stability condition if it satisfies the $\bT$-support property. We denote by $\bT\Stab_{\s}\cD_{\bT}$ the set of all $\bT$-stability conditions with respect to a fixed tuple $\s\subset\mathbb C$.  Since the space $\bT\Stab_{\s}\cD_{\bT}$ is a subset of the usual stability space $\Stab(\cD_{\bT})$, the distance $d$ on $\Stab(\cD_{\bT})$ induces a topology on $\bT\Stab_{\s}\cD_{\bT}$. Let $\bC_{\s}$ denote the complex numbers equipped with the $R(\bT)$-module structure via the evaluation map $\ev_{\s}$. The deformation property of $\bT\Stab_{\s}\cD_{\bT}$ is given by the following:

\begin{proposition}
\label{pro:defomation of T-stab}
The forgetful map
\begin{equation*}
\begin{aligned}
\cZ_{\s}:\bT\Stab_{\s}(X)&\to\Hom_{R}(K(\cD_{\bT}),\mathbb C_{\s})\\
((Z,\cP),\s)&\mapsto Z
\end{aligned}
\end{equation*}
is a local homeomorphism of topological spaces. In particular, $\cZ_{\s}$ induces a complex structure on $\bT\Stab_{\s}\cD_{\bT}\subset\Stab(\cD_{\bT})$. 
\end{proposition}

\begin{proof}
We first demonstrate that the $\bT$-support property implies local-finiteness. Suppose $((Z,\cP),\s)\in\bT\Stab_{\s}(\cD_{\bT})$ is not locally-finite. Then for some $\phi$ and $0<\epsilon<\frac{1}{2}$ the $\cP(\phi-\epsilon,\phi+\epsilon)$ is not be of finite-length. Consequently, there exists $E\in\cP(\phi-\epsilon,\phi+\epsilon)$ with an infinite composition series of simple quotients $\{S_i\}$. We claim that $|Z(S_i)|\to0$. Otherwise,
\[
|Z(E)|=\left|\sum Z(S_i)\right|=\left|\sum e^{-\i\pi\phi}Z(S_i)\right|\ge\sum\Re(e^{-\i\pi\phi}Z(S_i))\to\infty,
\]
which is impossible. Since each $S_i$ is $\sigma$-semistable and $\Vert S_i\Vert>0$ for some constant $C>0$, the $\bT$-support property gives
\[
0<C_{\sigma}\le\frac{|Z(S_i)|}{\|S_i\|}\to0,
\]
which is a contradiction. Hence, $((Z,\cP),\s)$ is locally-finite. By the deformation property of $\Stab(\cD_{\bT})$, there exists $0<\epsilon'<\frac{1}{8}$ such that for any $W\in\Hom_{R}(K(\cD_{\bT}),\mathbb C)_{\s}$ satsfying 
\[
|W(E)-Z(E)|<\sin(\epsilon\pi)|Z(E)|
\]
 for every $\sigma$-semistable object $E$, there is a unique slicing $\cQ$ with $(W,\cQ)$ a stability condition on $\cD_{\bT}$ and $d((Z,\cP),(W,\cQ))<\epsilon$. From the construction of $\cQ$ in \cite{MR2373143}, we have $E\in\cP(\phi-\epsilon,\phi+\epsilon)$. Assume that $E[T_i]\notin\cP(\phi+\Re(s_i)-\epsilon,\phi+\Re(s)+\epsilon)$. Then there is a short exact sequence 
\[
0\to A\to E[T_i]\to B\to 0
\]
in $\cP(\phi+\Re(s_i)-\epsilon,\phi+\Re(s_i)+\epsilon)$ with $\phi_W(A)>\phi_W(E[T_i])>\phi_W(B)$. By definition, the sequence
\[
0\to A[-T_i]\to E\to B[-T_i]\to0
\]
belongs to $\cP(\phi-\epsilon,\phi+\epsilon)$ contradicting the fact that $E$ is $W$-semistable. Hence, $((W,\cQ),\s)$ is a pre-$\bT$-stability condition. Now let $F$ be semistable with respect to $(W,\cQ)$ and let $\{A_i\}$ be the Harder-Narasimhan factors of $F$ with respect to $\sigma$. Then
\begin{equation}
\begin{aligned}
\frac{|W(F)|}{\|F\|}&\ge\frac{|Z(F)|}{\|F\|}-\frac{|W(F)-Z(F)|}{\|F\|}\\
&\ge\frac{|Z(F)|}{\|F\|}-\epsilon\\
&\ge{\cos(2\epsilon)\frac{\sum|Z(A_i)|}{\sum\|A_i\|}}-\epsilon\\
&\ge C_{\sigma}^{-1}\cos(2\epsilon)-\epsilon
\end{aligned}
\end{equation}
where for simplicity we assume that $\|W-Z\|\le\epsilon$. Thus, $((W,\cQ),\s)$ satifies the $\bT$-support property and belongs to $\bT\Stab_{\s}(\cD_{\bT})$. This completes the proof.
\end{proof}

Define
\[
\bT\Stab(\cD_{\bT}):=\bigcup_{\s\in\mathbb C^m}\bT\Stab_{\s}(\cD_{\bT}).
\]

If $(\sigma,\s)\in\bT\Stab_{\s}(\cD_{\bT})$ and $(\sigma',\s')\in\bT\Stab_{\s'}(\cD_{\cT})$ with $\s\ne\s'$, then
\begin{equation*}
\begin{aligned}
d(\sigma,\sigma')&\ge\sup_{k\in\mathbb Z}\left\{|\phi^{\pm}_{\sigma}(E[kT_i])-\phi_{\sigma'}^{\pm}(E[kT_i])|\right\}\\
&=\sup_{k\in\mathbb Z}\left\{|\phi^{\pm}_{\sigma}(E)-\phi^{\pm}_{\sigma'}(E)+k(\Re(s_i)-\Re(s_i'))| \right\}\\
&=\infty
\end{aligned}
\end{equation*}
Hence, $\bT\Stab_{\s}(\cD_{\bT})$ and $\bT\Stab_{\s'}(\cD_{\bT})$ lie in different connected components of $\Stab(\cD_{\bT})$ when  $\s\ne\s'$. 

We proceed to consider reduction on $\bT$-stability conditions by passing to certain orbit categories of $\bT$-categories. Let $\cD$ be a triangulated category with an  endofunctor $\Phi:\cD\to\cD$. Recall that the \emph{orbit category} $\cD/\Phi$ has the same objects as $\cD$ and morphism spaces 
\[
\Hom_{\cD/\Phi}(E,F):=\bigoplus_{k\in\mathbb Z}\Hom_{\cD}(E,\Phi^k(F)).
\]

\begin{definition}
Let $\n=\{n_1,\dots,n_m\}\subset\mathbb Z_{>0}$ be a set of positive integers. The \emph{orbit quotient}
\[
\cD_{N}=\cD_{\bT}\sslash[T_1-n_1,\dots,T_m-n_m]
\]
is defined as the triangulated hull of the orbit category $\cD_{\bT}/[T_1-n_1,\dots,T_m-n_m]$. It is $\n$\emph{-reductive} if 
\begin{itemize}
\item the quotient functor $\pi_{\n}:\cD_{\bT}\to\cD_{\n}$ is exact;
\item the Grothendieck group of $\cD_{N}$ is free of finite rank and the induced $R(\bT)$-linear map
\[
[\pi_{\n}]:K(\cD_{\bT})\to K(\cD_{\n})
\]
is a surjection given by $t_{i}\mapsto (-1)^{n_i}$.
\end{itemize}
\end{definition}

Given a $\bT$-stability condition $(\sigma,\n)$  on $\cD_{\bT}$ with $\sigma=(Z,\cP)$, we define $\sigma_{\n}:=(Z_{\n},\cP_{\n})$ on $\cD_{\n}$ by
\begin{itemize}
\item $\cP_{\n}(\phi):=\cP(\phi)$
\item $Z_{\n}:K(\cD_{\n})\to\mathbb C$ satisfying that $Z=Z_{\n}\circ[\pi_{\n}]$.
\end{itemize}
By definition, $\sigma_{\n}$ forms a stability condition on $\cD_{\n}$ with support property. This yields an injection of complex manifolds
\begin{equation*}
\begin{aligned}
\iota_{\n}:\bT\Stab_{\n}(\cD_{\bT})&\to\Stab\cD_{\n}\\
(\sigma,\n)&\to\sigma_{\n}
\end{aligned}
\end{equation*}
Moreover, by \Cref{pro:defomation of T-stab}, one can show that the image of $\iota_{\n}$ is both open and closed in $\Stab\cD_\n$.

\subsection{$G$-NMMP for equivariant projective spaces}
\label{sec:G-NMMP for projective spaces}
In this section, we study the $G$-NMMP for projective spaces equipped with a torus $\bT$-action. Based on the framework of $\bT$-stability conditions introduced in \Cref{sec:T-stab}, we construct quasi-convergent paths satisfying \Cref{proposal:main} in the space of $\bT$-stability conditions with respect to suitable tuple of complex numbers. We begin with a brief description of the equivariant quantum cohomology of the projective spaces, following \cite{Cotti2019EquivariantQD}.

Let $m\ge2$. Consider the diagonal action of $\bT=(\mathbb C^{\ast})^m$ on $\mathbb C^m$. This induces an action of $\bT$ on $\bP^{m-1}$, the projective space parametrizing the one-dimensional subspaces $F\subset\mathbb C^m$. If $(u_1,\dots,u_m)$ is the standard basis of $\mathbb C^m$, denote by $pt_I\in\bP^{m-1}$, with $I=1,\dots,m$ the point corresponding to the coordinate line spanned by $u_I$. The points $pt_I$ are the fixed points of the $\bT$-action. 

Consider the $\bT$-equivariant cohomology algebra $H^{\ast}_{\bT}(\bP^{m-1},\mathbb C)$. Denote
\begin{enumerate}
    \item by $x$ the first equivariant Chern class of the tautological line bundle $\cO(-1)$ on $\bP^{m-1}$ with its standard $\bT$-structure;
    \item by $\y=(y_1,\dots,y_{m-1})$ the equivariant Chern roots of the quotient bundle $\cQ$ (if $F\subset\mathbb C^m$ is the line represented by $p\in\bP^{m-1}$, then the fiber $Q_p$ is  $\mathbb C^m/F$);
    \item $\z=(z_1,\dots,z_m)$ the equivariant parameters corresponding to the factors of the torus $\bT$.
\end{enumerate}
Then
\[
H^{\ast}_{\bT}(\bP^{m-1};\mathbb C)\cong\mathbb C[x,\z]\bigg/\bigg\langle\prod_{i=1}^{m}(x-z_i)\bigg\rangle.
\]
and
\begin{equation}
  \label{eq:k-group}
K^{\ast}_{\bT}(\bP^{m-1};\mathbb C)\cong\mathbb C[x^{\pm 1},\T^{\pm 1}]\bigg/\bigg\langle\prod_{i=1}^{m}(x-T_i)\bigg\rangle,
\end{equation}
where $T_1,\dots,Z_m$ are the equivariant parameters corresponding to the factors of $\bT=(\mathbb C^{\ast})^m$. The ring $H^{\ast}_{\bT}(\bP^{m-1},\mathbb C)$ is a module over $H^{\ast}_{\bT}(\mathrm{pt},\mathbb C)\cong\mathbb C[\z]$. Setting all equivariant parameter $z_i$'s to zero recovers the classical cohomology algebra
\[
H^{\ast}(\bP^{m-1},\mathbb C)\cong\mathbb C[x]/\langle x^m\rangle.
\]
For a given $\beta\in H_2(\bP^{m-1};\mathbb Z)$ and integers $g,n\ge0$, let $\overline{\cM}_{g,n}(\bP^{m-1},\beta)$ be the moduli stack of genus $g$ stable maps to $\bP^{m-1}$ of degree $\beta$ with $n$ marked points. The $\bT$-action on $\bP^{m-1}$ induces a $\bT$-action on $\overline{\cM}_{g,n}(\bP^{m-1},\beta)$. The small quantum product at $\gamma:=e^{t}x\in H^{2}_{\bT}(\bP^{m-1};\mathbb C)$ is given by
\begin{equation*}
 x\star_{\gamma}x_j=
 \begin{cases}
x_{j+1},\quad &j=0,\dots,m-2\\
e^{t}+\sum^{m}_{i=1}s_i(\z)x_{m-i},\quad &j=m-1
 \end{cases}
\end{equation*}
where $s_i(\z)$ are the elementary symmetric polynomials in $\z$. The algebraic structure induced by $\star_{\gamma}$ on $H^{\ast}_{\bT}(\bP^{m-1};\mathbb C)$ is called the
small equivariant quantum cohomology of $\bP^{m-1}$. Furthermore, the equivariant quantum product with the Euler vector field $\cE$ at $\gamma$ in the standard basis $\{x^i\}$ can be written as
\begin{equation*}
E_{\gamma}:=c_1(\bP^{m-1})\star_{\gamma}=m
\begin{pmatrix}
0&0&0&\cdots&0&e^t+(-1)^{m-1}s_{m}(\z)\\
1&0&0&\cdots&0&(-1)^{m-2}s_{m-1}(\z)\\
0&1&0&\cdots&0&(-1)^{m-3}s_{m-2}(\z)\\
&&&\cdots&&\\
0&0&0&\cdots&1&s_1(\z)
\end{pmatrix}.
\end{equation*}
The corresponding quantum differential equation is
\begin{equation}
  \label{eq:pn-qde}
  (q \frac{d}{d q} - x \star_{q}) I(q, \z) = 0
\end{equation}
which matches \eqref{eq: qde} where \(\star_{q} = \star_{(0, \ln q, 0, \cdots, 0)}\) up to a coordinate change \(q = t^{m}\).

\begin{definition}
  Define the \emph{master function} \(\Phi\) by
  \[\Phi(t, q, \z) := e^{\i\pi \sum_{i = 1}^{m} T_{i}}(e^{-\i\pi m} q)^{t} \prod_{i = 1}^{m} \Gamma(T_{i} - t)\]
  where \(\Gamma\) is the Gamma function with Taylor expansion
  \[\Gamma(1 + T) = \exp( - \gamma T + \sum_{n = 2}^{\infty}\frac{(-1)^{n}}{n} \zeta(n)T^{n})\]
  Define the \emph{weight function} \(W\) by
  \[W(t) := \prod_{j = 1}^{m - 1}(y_{j} - t)\]
  with \(y_{1}, \cdots, y_{m - 1}\)  the equivariant Chern roots of the natural quotient bundle \(Q\) on \(\bP^{m-1}\).

  The Jackson integral \(\Phi_{J}\), \(J = 1, \cdots, m\) is the \(H_{\bT}^{*}(\bP^{m - 1}, \bC)\)-valued function on \(\widetilde{\bC^{*}} \times \Omega_{\bT}\) given by
  \[\Phi_{J}(q, \z) := - \sum_{r = 0}^{\infty} \Res_{t = T_{J} + r} \Phi(t, q, \z)W(t),\]
  where \(\widetilde{\bC^{*}}\) is the universal cover of \(\bC^{*}\) and \(\Omega_{\bT}\) is the complement in \(\bC^{m}\) of the hyperplanes
  \[z_{i} - z_{j} = k, \quad i,j = 1, \cdots, m, i \neq j, k \in \bZ\]
\end{definition}

\begin{theorem}(\cite[Theorem 7.4]{Cotti2019EquivariantQD})
  The function \(\Phi_{J}(q, \z)\) with \(J = 1, \cdots, m\) are holomorphic on \(\widetilde{\bC^{*}} \times \Omega_{\bT}\) and are solutions of the quantum differential equation \eqref{eq:pn-qde}.
\end{theorem}

Let
\[\check{T} := \exp(2 \i \pi t), \quad \check{Z}_{J} := \exp(2 \i \pi z_{J}), \quad J=1, \cdots, m.\]
For a Laurent polynomial \(Q(X, \mathbf{Z}) \in \bC[X^{\pm 1}, \mathbf{Z}^{\pm 1}]\), we set
\[\Phi_{Q}(q, \z) := \sum_{J = 1}^{m}Q(\check{Z}_{J}, \check{\mathbf{Z}}) \Phi_{J}(q, \z).\]
In particular, each \(\Phi_{Q}\) solves \ref{eq:pn-qde}. If \(Q(X, Z) = X^{n}\), we denote
\[\Phi^{n} := \Phi_{X^{n}} = \sum_{J = 1}^{m} \check{Z}_{J}^{m} \Phi_{J}.\]

\begin{theorem}(\cite[Theorem 7.11]{Cotti2019EquivariantQD})
  There is a well-defined morphism from \(K_{0}^{\bT}(\bP^{m - 1})_{\bC}\) to the space of solutions of \eqref{eq:pn-qde} given by composing \eqref{eq:k-group} with
  \[Q \to \Phi_{Q}.\]
\end{theorem}

We denote the morphism by \(\theta\). Under \(\theta\), the Beilinson basis \(\langle \cO, \cdots, \cO(m - 1) \rangle\) maps to
\[(\Phi^{0}, \cdots, \Phi^{m-1}).\]

Now we examine the asymptotic behavior of \(\Phi^{n}\). Let \((r, \phi)\) be the coordinate of the universal cover \(\widetilde{\bC^{*}}\) of \(\bC^{*}\). Set
\[q = s^{m}, s = re^{- 2 \i\pi \phi}, \quad\text{ for } r > 0, \phi \in \bR.\]
\begin{lemma}
  \label{lem:15}
  In a sector where \(n \in \bZ\) and \(\phi \in \bR\) satisfy
  \[\frac{n}{m} - 1 < \phi < \frac{n}{m},\]
  we have asymptotic expansion as \(s \to \infty\)
  \[\Phi^{n}(s^{m}, \z) = \frac{(2 \pi)^{\frac{m - 1}{2}}}{\sqrt{m}} e^{\i\pi \sum_{i = 1}^{m} z_{i}}(e^{-\i\pi} \zeta_{m}^{n}s)^{\sum_{i = 1}^{m}z_{i} + \frac{m - 1}{2}} e^{m s \zeta_{m}^{n}} (1 + O(\frac{1}{s})),\]
  where
  \[\zeta_{m} = \exp(\frac{2 \i \pi}{m}),\]
  and
  \[\arg(e^{-\i\pi} \zeta_{m}^{n}s) = 2 \pi \frac{n}{m} - \pi - 2 \pi \phi.\]
\end{lemma}

Taking logarithms gives
\begin{equation}
  \label{eq:asymp}
  \log \Phi^{n}(s^{m}, \z) = m s \zeta^{n}_{m} + O(\log s).
\end{equation}

For any smooth variety $X$, the equivariant derived category $D^b_{\bT}(X)$ is smooth and proper since $\bT$-action is smooth and proper.

The action of $\bT$ on $\mathbb C^m$ induces a natural $\bT$-structure on the structural sheaf $\cO_{\bP^{m-1}}$ and on the tautological line bundle $\cO(-1)$ on $\bP^{m-1}$. Any vector bundle obtained from $\cO_{\bP^{m-1}}$ and $\cO(-1)$ through tensorial operations inherits a $\bT$-structure. Recall that the Beilinson exceptional collection, as a full exceptional collection on derived category of $D^b(\bP^{m - 1})$ is 
\[
\{\cO,\cO(1),\dots,\cO(m-1)\}.
\]
with its natural $\bT$-structure, this collection is a full equivariant exceptional collection in $D^b_{\bT}(\bP^{m-1})$.

\begin{lemma}
  Via the isomorphism \(\theta\), the exceptional basis
  \[\left([\cO(-k - m + 1)], \cdots, [\cO(-k)]\right)\]
  is identified with
  \[Q_{k} := (\Phi^{k + m - 1}, \cdots, \Phi^{k}).\]
\end{lemma}

We know construct a quasi-convergent path in the stability space. Let \(\mathbf{1}^{(w_{1}, \cdots, w_{n})}\) denote the one-dimensional \(\bT\)-representation of weight \((w_{1}, \cdots, w_{m}) \in \bZ^{m}\). Explicitely,
\[(z_{1}, \cdots, z_{m}) \cdot z = z_{1}^{w_{1}} \cdots z_{m}^{w_{m}} z.\]
For \(\s = (s_{1}, \cdots, s_{m}) \in \bC^{m}\) and \(\phi \in \bR\), define
\[\cA_{\s, \phi} = \left\langle \mathbf{1}^{w_{1}, \cdots, w_{m}}[\lfloor {\sum_{i = 1}^{m}w_{i} \Re s_{i} + \phi} \rfloor]_{(w_{1}, \cdots, w_{m}) \in \bZ^{m}} \right\rangle_{\rm ext}.\]
\begin{lemma}
  \(\cA_{\s, \phi}\) is a heart of a bounded t-structure in \(D^{b}(\Rep(\bT))\).
\end{lemma}

\begin{proof}
  Since \(\Rep(\bT)\) is completely decomposable, every object \(\cE \in D^{b}(\Rep(\bT))\) can be written as
  \[\cE = \bigoplus_{i \in I} (\mathbf{1}^{(w_{1, i}, \cdots, w_{m, i})}[k_{i}])^{\oplus n_{i}}\]
  with a finite index set \(I\). Reordering the summations yields a sequence of morphisms which satisfies the requirements in \cite[Lemma 3.2]{Bridgeland2006SpacesOS}. Therefore, $\cA_{\s,\phi}$ is indeed a heart of a bounded t-structure.
\end{proof}

Since
\[K(D^{b}(\Rep(\bT)))_{\bC} = \bC[T_{1}^{\pm 1}, \cdots, T_{m}^{\pm 1}],\]
We consider the central charge
\[
\begin{aligned}
  Z_{\s} = \overline{\ev}_{\s}:  \bC[T_{1}^{\pm 1}, \cdots, T_{m}^{\pm 1}] & \to \bC\\
  T_{i}^{n} &\to e^{\i\pi n s_{i}}.
\end{aligned}
\]
For \(0 < M \in \bR\) and \(\phi \in \bR\), we write \(Z_{\s, M, \phi} = M e^{\i\pi\phi}Z_{\s}\).

\begin{lemma}
  \label{lem:14}
  Let \(\cD = D^{b}(\Rep(\bT))\). Then \(\sigma_{0} = (Z_{0} := Z_{\s, M, \phi}, \cA_{\s, \phi})\) is a \(\s\)-stability condition on \(\cD\).
\end{lemma}

\begin{proof}
For any object $\cE\in\cA_{\s,\phi}$, we have a direct-sum decomposition
\[\cE = \bigoplus_{i \in I} (\mathbf{1}^{(w_{1, i}, \cdots, w_{i, m})}[k_{i}])^{\oplus M_{i}},\]
from which the Harder-Narasimhan filtration follows directly. 
The remaining conditions are clearly satisfied.
\end{proof}

\begin{proposition}
  \label{prop:7}
  Let \(\cE = \{\cE_{0}, \cdots, \cE_{m - 1}\}\) be a strong exceptional collection of \(\cD = D_{\bT}(\bP^{m - 1})\) such that
  \[\cD = \langle \cE_{0} \otimes \Rep(\bT), \cdots, \cE_{m - 1} \otimes \Rep(\bT) \rangle.\]
  Then for any \(M_{i} \in \bR_{> 0}\) and \(\phi \in \bR\), \(i = 0, \dots, n - 1\), satisfying \(\lceil \phi_{i} \rceil < \phi_{i + 1}\) for \(i = 0, \dots, n - 2\), there exists a unique \(\s\)-stability condition \(\sigma = (Z, \cP)\) such that every \(\cE_{i} \otimes L \in \cE_{i} \otimes \Rep(\bT)\) with \(L\) a one-dimensional \(\bT\)-representation is semistable and \[Z(\cE_{i} \otimes L) = M_{i} e^{\i \pi \phi_{i}} Z_{\s}([L]).\]
\end{proposition}

\begin{proof}
  Choose integers \(p_{i} \in \bZ\) uniquely so that \(\phi_{i} + p_{i} \in ({0},{1}]\). Then the hearts \(H_{i} = \cE_{i} \otimes \Rep(\bT)[p_{i}]\) of the subcategories \(\langle \cE_{i} \otimes \Rep(\bT) \rangle\) satisfy
  \[\Hom_{\cD}^{\leq 0}(H_{i}, H_{j}) = 0, \text{ for } i < j.\]
  By \Cref{lem:14}, each \(\sigma_{i} = (Z_{i} = (-1)^{p_{i}}Z_{s, M_{i}, \phi_{i}}, H_{i})\) is a pre-stability condition on \(\langle \cE_{i} \otimes \Rep(T) \rangle\). Inductively applying \cite[Proposition 2.2]{Collins2009GluingSC}, we obtain that a stability condition \(\sigma = (Z, H)\) where \[H = \langle \cE_{0} \otimes \Rep(\bT)[p_{0}], \cdots, \cE_{m} \otimes \Rep(\bT)[p_{m}] \rangle_{\rm ext}\] and \(Z|_{\cE_{i} \otimes \Rep(\bT)} = Z_{i}\). Consequently \(Z(\cE_{i} \otimes L) = M_{i} e^{\i \pi \phi_{i}} Z_{\s}([L])\).

  Again by \cite[Proposition 2.2]{Collins2009GluingSC}, all objects \(\cE_{i} \otimes L \in \cE_{i} \otimes \Rep(\bT)\) are semistable. Uniqueness follows from the construction.
\end{proof}

\begin{sloppypar}
Equation \eqref{eq:asymp} gives the asymptotic behavior of solutions of quantum differential equations corresponding to the sheaves \(\cO(n)\). However, $\Phi^i$ are not ordered according to \(\Im \log \Phi^{i}((re^{- 2 \i\pi \phi})^{m}, \z)\) as \(r \to \infty\). Indeed, mutations are required to obtain the correct strong full exceptional collection. Recall that for an exceptional pair \((E, F)\), the left mutation \(L_{E}F\) and the right mutation \(R_{F}E\) are defined by the following distinguished triangles
\[\Hom^{*}(E, F) \otimes E \to F \to L_{E}F,\]
\[R_{F}E \to E \to \Hom^{*}(E, F)^{\vee} \otimes F.\]
If \(\{E_{0}, \cdots, E_{n}\}\) is a full exceptional collection, then
\[\{E_{0}, \cdots, E_{i}, L_{E_{i + 1}}E_{i + 2}, E_{i + 1}, E_{i + 3}, \cdots, E_{n}\},\]
\[\{E_{0}, \cdots, E_{i},E_{i + 2}, R_{E_{i + 2}}E_{i + 1}, E_{i + 3}, \cdots, E_{n}\}.\]
are also full exceptional collections.
\end{sloppypar}
In general, a mutation of a strong full exceptional collection need not remain strong, but for projective spaces we have the following fact.
\begin{lemma}(\cite[Corollary 2.4, Proposition 3.3]{BP1994})
  \label{lem:16}
  If a full exceptional collection of coherent sheaves of length \(n + 1\) on a smooth projective variety \(X\) of dimension \(n\) is strong, then any mutations of this collection is again strong.
\end{lemma}

We call a parameter \(\theta\) \emph{admissible} if the numbers \(\Im e^{- 2 \pi \i \theta} \zeta^{n}_{m}\) and \(\Re e^{- 2 \pi \i \theta} \zeta^{n}_{m}\), \(0 \leq n \leq m - 1\) are distinct for all \(0 \leq n \leq m - 1\).

\begin{lemma}
  \label{lem:17}
  For an admissible \(\theta\), there exist a full strong exceptional collection \(\{cE_{0}, \cdots,\cE_{m - 1}\}\) of \(D^{b}(\bP^{m - 1})\) and a permutation \(\sigma\) of \(\{0, \cdots, m - 1\}\) such that 
  \begin{enumerate}
    \item \(\log \Phi_{\cE_{i}}(s^{m}, \z) = m s \zeta_{m}^{\sigma(i)} + O(\log s)\);
    \item For \(i = 0, \cdots, m - 2\), we have 
    \[
    \Im e^{-2 \i\pi \theta} \zeta^{\sigma(i)}_{m} < \Im e^{-2 \i\pi \theta} \zeta^{\sigma(i + 1)}_{m}.
    \]
  \end{enumerate}

\end{lemma}

\begin{proof}
  We start with the collection \(\cE = \{\cO, \cdots, \cO(m - 1)\}\). Note that \(\cE\) already satisfies condition (1). If it also satisfies condition (2), we are done. Otherwise, since \(\theta\) is admissible, there exists an index \(0 \leq j \leq m - 2\) with
  \[\Im e^{- 2 \i\pi  \theta} \zeta_{n}^{\sigma(j)} > \Im e^{- 2 \i\pi\theta} \zeta_{m}^{\sigma(j + 1)}.\]
  
  If in addition \(\Re e^{- 2 \i\pi \theta} \zeta_{n}^{\sigma(j)} > \Re e^{- 2 \i\pi \theta} \zeta_{m}^{\sigma(j + 1)}\), replace the pair \((\cE_{j}, \cE_{j + 1})\) by \((\cE_{j + 1}, R_{\cE_{j+ 1}}\cE_{j})\). If the inequality on the real parts is reversed, replace the pair \((\cE_{j}, \cE_{j + 1})\) by \((L_{\cE_{j}}\cE_{j + 1}, \cE_{j})\).

  We claim that condition (1) remains valid after such a  mutatation. Assuming the claim, each operation decreases the inversion number of imaginary parts by onw and therefore after finitely many steps, we obtain a full exceptional collection satsifying both conditions. By \Cref{lem:16}, the resulting collection is strong.

  To prove the claim, consider the case of a right mutation. From the defining triangle of \(R_{\cE_{j + 1}}\cE_{j}\), we have that
  \[\Phi_{R_{\cE_{j + 1}}\cE_{j}} = \Phi_{\cE_{j}} - {\rm hom}(\cE_{j}, \cE_{j + 1}) \Phi_{\cE_{j + 1}}.\]
  Since \(\Re e^{- 2 \i\pi \theta} \zeta_{m}^{\sigma(j)} > \Re e^{- 2 \i\pi \theta} \zeta_{m}^{\sigma(j + 1)}\),
  \[1 + \hom(\cE_{j}, \cE_{j + 1}) e^{m re^{- 2 \i\pi \theta} \zeta^{\sigma(j + 1)}_{m} - m re^{- 2 \i\pi\theta} \zeta_{m}^{\sigma(j)}} = 1 + O(s^{-1})\]
    as \(r \to \infty\). Taking the logarithms, we obtain the required asymptotic behavior. The left mutation case is analogous.
\end{proof}

In constrast to the general base change in \Cref{def:base change}, we use a more specialized construction for the present setting. Here, we define
\[
H_{\bT}^{\Omega_{\bT}}(\bP^{m - 1}) = H_{\bT}^{*}(\bP^{m - 1}) \otimes_{H_{\bT}({\rm pt})} \cO_{\Omega_{\bT}},
\]
where $\cO_{\Omega_{\bT}}$ is the ring of holomorphic functions on $\Omega_{\bT}$. For a $\bT$-equivariant vector bundle $V$ of rank $r$ on $\bP^{m-1}$, its equivariant Chern character is given by 
\[
{\rm ch}_{\bT}(V) \coloneq \sum_{j = 1}^{r} e^{2 \pi \i \xi_{j}}
\]
where \(\xi_{j}\) are the equivariant Chern roots of \(V\). The exponential on the right hand side is, in general, an infinite power series in \(\xi_{j}\). That is precisely why the base change to $\cO_{\Omega_{\bT}}$ is necessary.

\begin{theorem}
  \label{thm:tncmmp}
  For \(\s \in \Omega_{\bT}\), there exists \(R > 0\), such that for every admissible \(\theta\), there exists a quasi-convergent path \(\sigma_{t} = (Z_{t}, \cA_{t})\) with
  \[t = r e^{- 2 \i\pi \theta}, \quad  r \in (R, + \infty)\]
  whose central charge is given by  \[Z_{t}(\alpha) = \ev_{\s} (\int_{\bP^{m - 1}}^{T} \Phi_{t}({\rm Ch}_{\bT}(\alpha))),\] where \(\Phi_{t}\) is a fundamental solution of the equivariant quantum differential equation linearly extended to \({\rm End}(H_{\bT}^{\Omega_{\bT}}(\bP^{m - 1}))\), valid in the sector 
  \[\frac{n}{m} - 1 < \phi < \frac{n}{m}, \quad n\in\mathbb Z, \phi\in\mathbb R.\]
  Moreover, the whole path lies \(\bT \Stab_{\s}(\bP^{m-1})\).
\end{theorem}

\begin{proof}
  Take the fundamental solution \(\Phi_{t}(\alpha) = \Phi_{\alpha}(t)\) and \(s = t^{m}\). Since \(\theta\) is admissible, \Cref{lem:17} provides a strong exceptional collection \(\cE = \{\cE_{0}, \cdots, \cE_{m - 1}\}\) and a permutation \(\sigma \in S_{\{0, \cdots, m - 1\}}\) such that
  \[\Im \log \Phi^{\sigma(i)}(t, \s) < \Im \log \Phi^{\sigma(i + 1)}(t, \s).\]
  For sufficiently large \(r\), we have
  \[\abs{\Im \log \Phi_{\cE_{i}}(t, \s) - \Im \log \Phi_{\cE_{i + 1}}(t, \s)} \geq 2\] for \(0 \leq i \leq m - 2\),
  and consequently
  \[\lceil \Im \log \Phi_{\cE_{i}}(t, \s)\rceil < \Im \log \Phi_{\cE_{i + 1}}(t, \s).\]
  Then by~\Cref{prop:7}, we obtain a path \(\sigma_{\bullet}\) in the stability space for which every \(\cE_{i} \otimes L\)  limit semistable. Since \(\langle \cE_{0} \otimes \Rep(\bT), \cdots, \cE_{n} \otimes \Rep(\bT) \rangle\) is a semiorthogonal decomposition, the first condition of quasi-convergence is satisfied. On the other hand, for two limit semistable objects \(\cE = \cE_{i} \otimes R_{1}\) and \(\cF = \cE_{j} \otimes R_{2}\), we have
  \[l_{t}(\cE) = \log \Phi^{\sigma(i)}(t, \s) + O(\log s) = m s \zeta^{\sigma(i)}_{m} + O(\log s).\]
 Hence,
  \[
  \begin{aligned}
    \lim_{t \to \infty}\frac{l_{t}(\cE/\cF)}{1 + \abs{l_{t}(\cE/\cF)}} &= \lim_{t \to \infty} \frac{m s (\zeta^{\sigma(i)}_{m} - \zeta^{\sigma(j)}_{m})}{1 + \abs{m s (\zeta^{\sigma(i)}_{m} - \zeta^{\sigma(j)}_{n + 1})}} \\
    & = e^{- 2 \i\pi \theta}(\zeta_{m}^{\sigma(i)} - \zeta_{m}^{\sigma(j)}),
  \end{aligned}
  \]
  which exists, fulfilling the second condition of quasi-convergence.
\end{proof}

\begin{sloppypar}
\begin{remark}
  The same reasoning may apply to any Fano variety \(X\) that satisfies the Gamma conjecture II \cite{10.1215/00127094-3476593}, i.e. \(X\) has generically semisimple quantum cohomology and admits a full exceptional collection \(D^{b}(X) = \langle \cE_{1}, \cdots, \cE_{n} \rangle\) such that \(\Phi_{t}(v(\cE_{i}))\) form an asymptotically exponential basis of solutions of the quantum differential equation. We expect that equivariant parameters do not affect the asymptotic behavior as \(t \to \infty\), and that the fundamental solution again produces a quasi-convergent path corresponding to \[D^{b}_{\bT}(X) = \langle \cE_{1} \otimes \Rep(\bT), \cdots, \cE_{n} \otimes \Rep(\bT) \rangle.\] as in the projective-space case.
\end{remark}
\end{sloppypar}

\begin{theorem}
  \label{thm:m=3}
  If \(m = 3\), that is \(\bT = (\bC^{*})^{3}\) acting on \(\bP^{2}\), and \(\theta\) is admissible, then there exists a quasi-convergent path \(\sigma_{t}\), \(t \in [\delta, +\infty)\) such that \(\sigma_{\delta}\) is geometric.
\end{theorem}

\begin{proof}
  We follow the proof of \cite[Theorem 5.11]{zuliani2024semiorthogonaldecompositionsprojectivespaces}. Let \(\cE_{0}, \cE_{1}, \cE_{2}\) be the exceptional collection furnished by \Cref{lem:17}, and set \[u_{j} = \Im (\zeta_{3}^{\sigma(j)} e^{- 2 \pi \i \theta}),\quad j = 0, 1, 2.\] 
  Admissibility guarantees that $u_j$ are distinct. Define 
  \[a_{j}(r) = e^{\i \beta_{j}} \Phi_{\cE_{i}}(r e^{- 2 \pi \i \theta}),\quad \phi_{\cE_{j}}(r) := \frac{1}{\pi} \Im (\ln a_{j}(r)),\] where \(\beta_{0}\) will be determined later.

 Choose \(0 < \epsilon \ll 1\) with \(\frac{\epsilon}{\pi} \abs{\Im (u_{j})} < \frac{1}{4}\) for \(j = 0, \dots, d\). Set
  \[\mu' := \min \left\{\min_{i = 0, 1} \left\{\frac{\epsilon}{\pi} \Im  u_{i + 1} - \Im u_{i}\right\}, \Im u_{0}\right\},\]
  and 
  \[\mu := \min \left\{\min_{i = 0, 1} \left\{\frac{\epsilon}{\pi} \Im  u_{i} - \Im u_{i + 1}\right\}, \Im u_{2}\right\}.\]
  We have \(-\frac{1}{2} < \mu < 0 < \mu' < \frac{1}{2}\). Pick \(\delta' > 0\) with \(\delta' > \epsilon\) such that for every \(r \in (\delta', +\infty)\),
  \[e_{j}(r) = \abs{\phi_{\cE_{j}}(r) - \frac{1}{\pi} (3 r \Im (u_{j}) + \beta_{j})} < \min \left\{ \frac{- \mu}{4}, \frac{\mu'}{4}\right\}.\]

  Now choose \(\beta_{j}\) such that
  \[\frac{1}{\pi} \Im(u_{j}/ \delta') + \beta_{j} = \frac{1}{2} + j.\]
   Fix \(\delta = \delta' - \epsilon>0\). For \(r \in (\delta - \epsilon, \delta + \epsilon)\), we have
  \(\phi_{\cE_{j}}(r) \in (j, j + 1)\)
  and at \(r = \delta\), 
  \[\abs{\phi_{\cE_{j + 1}}(r) - \phi_{\cE_{j}}(r)} < 1,\]
  while for \(r > \delta' + \epsilon\),
  \[\abs{\phi_{\cE_{j + 1}}(r) - \phi_{\cE_{j}}(r)} > 1.\]
  Consequently  \(\lceil \phi_{\cE_{j}}(r) \rceil < \phi_{\cE_{j + 1}}(r)\) for all \(r\in[\delta, \infty)\). \Cref{prop:7} therefore produces a path $\sigma_{\bullet}$, which by the same reasoning as in \Cref{thm:tncmmp} is a quasi-convergent.

 Finally, for a fixed weight \(\omega\), the subcategory \(\langle \cE_{0}\otimes \mathbf{1}^{\omega}, \cE_{1}\otimes \mathbf{1}^{\omega}, \cE_{2}\otimes \mathbf{1}^{\omega}\rangle\) is isomorphic to \(D^{b}(\bP^{2})\). By \cite[Proposition 2.5]{Li2017}, we conclude that the stability condition at \(r = \delta\) is geometric.
\end{proof}

\section{Relation with birational minimal model program}
\label{sec:relat-with-birat}
In this section, we explain how \Cref{conjecture:1} and \Cref{proposal:main} extend certain non-equivariant birational results to the \(G\)-equivariant setting for finite groups, relating to the \(D\)-equivalence conjecture and Dubrovin conjecture, based on \cite{halpernleistner2024noncommutativeminimalmodelprogram}.

\subsection{Minimal models for finite groups}
We begin with a \(G\)-equivariant analogue of \cite[Theorem 3.1]{KO2018}.
\begin{lemma}
  \label{lem:2}
  Let \(G\) be a finite group acting on a smooth projective algebraic variety \(X\) and let \(D^{b}_{G}(X) = \langle \cA, \cB \rangle\) be a \(G\)-semiorthogonal decomposition. Denote by \(B = {\rm Bs} \abs{\omega_{X}}^{G}\) the base locus of the \(G\)-invariant part of \(H^{0}(X, \omega_{X})\). For a closed point \(x \in X\), set \(E_{x} := \oplus_{g \in G} g^{*} \cO_{x}\) with its natural \(G\)-action. Then one of the following holds,
    \begin{enumerate}
      \item For every closed point \(x \in X \backslash B\), \(E_{x} \in \cA\). Moreover, the support of any object of \(\cB\) is contained in \(B\).
      \item For every closed point \(x \in X \backslash B\), \(E_{x} \in \cB\). Moreover, the support of any object of \(\cA\) is contained in \(B\).
    \end{enumerate}
\end{lemma}

\begin{proof}
  Take \(x \in X\backslash B\), then it is clear that \(gx \in X \backslash B\) for all $g\in G$. From the definition of semiorthogonal decomposition, there is a distinguished triangle
  \[F \to E_{x} \to G \xrightarrow{f} F[1]\]
  with \(F \in \cB\) and \(G \in \cA\). Choose a section \(s \in H^{0}(X, \omega_{X})^{G}\) that does not vanish at \(x\). By $G$-invariance, $s$ also does not vanish at each \(gx\). Set \(U = X \backslash Z(s)\). If \(f|_{U} \neq 0\), then \((\otimes s \circ f) |_{U} = \otimes s|_{U} \circ f_{U}\) is also non-zero. However, the \(G\)-equivariant Serre duality gives
  \[\Hom(G, F \otimes \omega_{X}[1]) \cong \Hom(G, F[\dim X - 1])^{\vee} = 0.\]
  This forces \(f|_{U} = 0\) and consequently we obtain a splitting \(E_{x} = F|_{U} \oplus G|_{U}\). Since both \(\cA\) and \(\cB\) are closed under the \(G\)-action, either \(F|_{U} = 0\) or \(G|_{U} = 0\). If \(G|_{U} = 0\), the morphism \(E_{x} \to G|_{U}\) is zero and the triangle yields \(F \cong E_{x} \oplus G[-1]\). By semiorthogonality, \(G = 0\) and thus \(E_{x} \in \cB\). Similarly, if \(F|_{U} = 0\), we obtain \(E_{x} \in \cA\).

  Now suppose \(E_{x} \in \cB\) (resp. \(E_{x} \in \cA\)). For any object \(F \in \cA\) (resp. \(F \in \cB\)), it implies that \(x \not \in {\rm Supp} F\). Therefore, the support of \(F\) is a proper closed subset of \(X\). Now consider the distinguished triangle of structure sheaf with natural \(G\)-structure
  \[F \to \cO_{x} \to G \to F[1].\]
  Since \({\rm Supp} \cO_{X} = {\rm Supp} F \cup {\rm Supp} G\), we must have either \({\rm Supp} F = X\) or \({\rm Supp} G = X\). By the support argument above, it follows that all \(E_{x}\) with \(x \in X \backslash B\) belong to the same component ($\cA$ or $\cB$), and the support of every object of the other component is contained in \(B\). This completes the proof.
\end{proof}

\begin{proposition}
  \label{prop:1}
  Let \(G\) be a finite group and \(\pi: X \to Y\) a \(G\)-equivariant contraction of projective varieties with \(X\) smooth and \(p_{g}^{G}(X)\coloneq\dim H^{0}(X, \omega_{X})^{G} > 0\). Assume that \(X\) admits \(G\)-invariant stability condition and that \Cref{conjecture:1}(A, D) holds for varieties over \(Y\). Then there exists an admissible subcategory \(\cM_{X/Y} \subset D^{b}_{G}(X)\) containing an object whose support is \(X\) with the following property:

  For any other contraction \(X' \to Y\) such that \(X'\) is \(G\)-birationally equivalent to \(X\) over \(Y\), there is an admissible embedding \(\cM_{X/Y} \hookrightarrow D^{b}_{G}(X')\).

  If in addition \Cref{conjecture:1}(B) holds for varieties over \(Y\), then \(\cM_{X/Y}\) carries a canonical \(\Perf(Y)^{\otimes}\)-module structure and the embeddings \(\cM_{X/Y} \hookrightarrow D^{b}(X')\) are \(\Perf(Y)^{\otimes}\)-linear.
\end{proposition}

\begin{proof}
  Let \(D^{b}(X) = \langle \cC_{1}, \cdots, \cC_{n} \rangle\) be the semiorthogonal decomposition assined to $\pi$ by \Cref{conjecture:1}(A) for a generic choice of parameter \(\psi\). Because \(p_{g}^{G}(X) > 0\), it follows from \Cref{lem:2} that exactly one of the \(\cC_{i}\) contains an object whose support is \(X\). Denote this category by \(\cC_{X, \psi}\). Since different generic choice of \(\psi\) gives mutation-equivalent semiorthogonal decompositions, the isomorphism class of \(\cC_{X, \psi}\) is independent of $\psi$. We write it simply as \(\cC_{X}\). If \Cref{conjecture:1}(B) holds, then \(\cC_{X}\) inherits a \(\Perf(Y)^{\otimes}\)-structure.

  Now let \(f: Z \to X\) be a projective \(G\)-birational morphism with \(Z\)-smooth. Associate  to \(f \circ \pi\) the category \(\cC_{Z}\). By \Cref{conjecture:1}(D), we have \(\cC_{Z} \cong \cC_{X}\) as admissible subcategories. Since \(\cC_{Z}\) corresponds to a direct sum of charge lattices of \(\cC_{X}\), which are finite dimensional, there exists a birational morphism \(Z \to X\) such that for any further \(G\)-birational morphism \(Z' \to Z \to X\), we have \(\cC_{Z} \cong \cC_{Z'}\). Take another \(G\)-contraction \(X' \to Y\) birational to \(\pi: X \to Y\). There exists a smooth projective \(Z'\) with \(G\)-birational maps \(G' \to Z\) and \(G' \to X'\) that are compatible with the birational equivalence over \(Y\). It follows that \(\cC_{Z} = \cC_{Z'} \subset \cC_{X'} \subset D^{b}_{G}(X')\) are admissible inclusions.
\end{proof}

\begin{corollary}
  \label{cor:D-equivalence conj}
  Let \(G\) be a finite group and assume that \Cref{conjecture:1}(A, D) holds for varieties over \(\spec \k\), if \(X\) and \(X'\) are birationally equivalent smooth projective varieties and the \(G\)-invariant of linear system \(\abs{K_{X}}\) is base-point free. Then there is a canonical admissible embedding \(D^{b}(X) \hookrightarrow D^{b}(X')\), which is an equivalence if \(G\)-invariant part of \(\abs{K_{X'}}\) is also base-point free.
\end{corollary}

\begin{proof}
  If \(G\)-invariant part of \(K_{X}\) is base-point free, then \(D^{b}_{G}(X)\) admits no \(G\)-semiorthogonal decompositions by \Cref{lem:2}. By \Cref{prop:1}, it remains to show that \(\cM_{X/\spec \k} = D^{b}_{G}(X)\). To see this, consider a \(G\)-birational morphism \(f: Z \to X\) with \(Z\) smooth and projective. Then by \Cref{conjecture:1}(D), \(\cC_{Z}\) densely supported must lie in \(f^{*}(D^{b}_{G}(X))\) and must be equal to \(f^{*}(D^{b}_{G}(X))\).
\end{proof}

In particular, the above corollary implies that \(G\)-birational equivalent Calabi-Yau varieties \(X\) and \(X'\) satisfies \(D^{b}_{G}(X) \cong D^{b}_{G}(X')\).

\subsection{$G$-equivariant Dubrovin's conjectures}

\begin{lemma}
  \label{lem:1}
  Let \(G\) be a finite group and let \(\cC\) a regular proper idempotent complete pre-triangulated dg-\(G\)-category such that \(K(\cC) = K(\Rep(G))\) and $\cC$ admits a \(G\)-invariant stability conditions. Then \(\cC\) is generated by a single exceptional object, i.e. \(\cC = \Rep(G)\).
\end{lemma}

\begin{proof}
  Let \(E\) be an indecomposable object in \(\cC\) and let \(\{V_{i}\}\) be the irreducible representation of \(G\). For each $i$, let \(\cD_{i}\) be the subcategory of \(\cC\) generated by \(E \otimes V_{i}\). Since \(\Hom(E, E) = \bC\) (trivial representation of \(G\)), we have \(\Hom(\cD_{i}, \cD_{j}) = 0\) for \(i \neq j\) . Hence there is an injection
  \[\oplus \cD_{i} \to \cC\]
  Since \(K(\Rep(G)) = K(\cC)\), this injection is an isomorphism and each $\cD_i$ satisfies \(\dim K(\cD_{i}) \otimes \bQ = 1\). Restriction the given $G$-invariant stability condition to \(\cD_{i}\) yields a stability condition on $\cD_i$. By \cite[Lemma 23]{halpernleistner2024noncommutativeminimalmodelprogram}, each \(\cD_{i}\) is then generated by an exceptional object. We may assume from the begining that \(E\) itself is exceptional. Consequently, \(\cC\) is generated by \(E\) and \(\cC = E \otimes \Rep(G) \cong \Rep(G)\).
\end{proof}

\begin{proposition}
  \label{pro:Dubrovin's conj}
  Let \(G\) be a finite group and \(X\) a smooth projective \(G\)-variety for which \Cref{proposal:main} holds for a generic parameter \(z\). Assume moreover that
  
  \begin{enumerate}
    \item \(\ch_G: K^G(X) \otimes \bQ \to H^{*}_{G}(X; \bQ)\) is an isomorphism, and 
    \item there exists a \(\psi\) such that \(E_{\psi}(1) \in \End(H^{*}_{\rm{alg}}(X) \otimes \bC)\) is semisimple with distinct eigenvalues. 
  \end{enumerate}
  Then \(D^{b}_{G}(X)\) admits a \(G\)-full exceptional collection consisting of limit semistable objects.
\end{proposition}

\begin{proof}
  Choose \(z\) so that the eigenvalues \(u_{1}, \cdots, u_{n}\) of \(\frac{-1}{z} E_{\psi}(1)\) have distinct real parts. Then the fundamental solution \(\Phi_{t}\) then satifies \(\norm{\Phi_{t}} \sim e^{\Re(u_{j})t}\) for some \(j\). If \(E\) is an eventually semistable object, we have \(\abs{Z_{t}(E)} \sim e^{\Re(u_{j}) t}\). The spanning condition in~\Cref{proposal:main} guarantees that for each eigenvalue \(u_{j}\), there exists an eventually semistable \(E\) with \(\abs{Z_{t}(E)} \sim e^{u_{j}t}\). This yields a decomposition
  \[D^{b}_{G}(X)=\langle \cC_{1}, \cdots, \cC_{n} \rangle\]
  into \(n\) distinct \(G\)-subcategories. Since $\mathrm{ch}_G$ is an ismorphism, each $\cC_i$ satisfies \(\dim_{R(G)} \cC_{i} = 1\). \Cref{lem:1} now implies that each \(\cC_{i}\) is generated by a single exceptional object.
\end{proof}

\printbibliography

\end{document}

%% file: preamble.tex
\usepackage[margin=1.6in]{geometry}

\usepackage{tikz-cd} %used for drawing commutative diagrams
\usepackage{graphicx} %supports image insertion
\usepackage{adjustbox} %supports adjusting box sizes
\usepackage{enumerate} %supports list environments
\usepackage{pgfplots}
\pgfplotsset{compat=1.17}
\usepackage{amssymb} %provides additional mathematical symbols and fonts
\usepackage{amsmath} %provides rich typesetting capabilities for mathematical expressions in LaTeX
\usepackage{mathtools} %provides more mathematical environments and commands
\usepackage{amsfonts} %provides additional mathematical fonts
\usepackage{amsthm} %supports theorem and proof environments
\usepackage{thmtools}
\usepackage{mathrsfs} %provides additional mathematical script fonts
\usepackage{physics}

\usepackage{esint}
\usepackage{epigraph} 
\usepackage{calligra}
\usepackage{array}
\usepackage{color}

\usepackage{subfiles} %supports separately compiling sub-files
\usepackage{appendix} %supports appendix environment
\usepackage{pgfplots, tikz, stmaryrd}
\usepackage[english]{babel} %supports multiple languages

\usepackage{xcolor}

\usepackage{fouriernc}
\usepackage{fancyhdr}
\usepackage{lastpage}
\usepackage{xstring}

\newenvironment{itenv*}
  {\phantomsection\par\medskip\noindent\itshape}
  {\par\medskip}

\usepackage[hypertexnames=false,colorlinks,linkcolor=red,anchorcolor=green,citecolor=blue]{hyperref} %provides cross-referencing functionality
\hypersetup{colorlinks,linkcolor={blue!50!black},citecolor={blue!50!black},urlcolor={blue!80!black}}
\hypersetup{linktocpage = true}
\usepackage{cleveref}
\DeclareMathOperator{\GL}{GL} %general linear group
\DeclareMathOperator{\PGL}{PGL} %projective general linear group
\renewcommand{\k}{\Bbbk} %abstract field
\renewcommand{\char}{\text{char}} %characteristic
\renewcommand{\Re}{\mathrm{Re}} %real part
\renewcommand{\Im}{\mathrm{Im}} %imaginary part
\renewcommand{\mod}{\mathrm{mod }} %module

\DeclareMathOperator{\Hom}{Hom} %homomorphism
\DeclareMathOperator{\End}{End} %endomorphism
\DeclareMathOperator{\Aut}{Aut} %automorphism
\DeclareMathOperator{\Rep}{Rep} %Representation
\DeclareMathOperator{\cE}{\mathcal{E}} %sheaf E
\DeclareMathOperator{\cF}{\mathcal{F}} %sheaf F
\DeclareMathOperator{\cG}{\mathcal{G}} %sheaf G
\DeclareMathOperator{\cO}{\mathcal{O}} %structure sheaf
\DeclareMathOperator{\sheafhom}{\mathscr{H}\text{\kern -3pt {\calligra\large om}}\,} %sheaf hom

\DeclareMathOperator{\cA}{\mathcal{A}} %smooth forms
\DeclareMathOperator{\T}{\mathbb{T}} %torus
\DeclareMathOperator{\ch}{ch} %chern character
\DeclareMathOperator{\cD}{\mathcal{D}} %currents
\DeclareMathOperator{\spec}{Spec} %prime spectrum
\usetikzlibrary{decorations.markings}
\tikzset{->-/.style={decoration={ markings, mark=at position #1 with
{\arrow{>}}},postaction={decorate}}}
\tikzset{-<-/.style={decoration={ markings, mark=at position #1 with
{\arrow{<}}},postaction={decorate}}}

\def\ev{\mathrm{ev}}

\def\hom{\mathrm{Hom}}
\def\spec{\mathrm{Spec}}
\def\exp{\mathrm{exp}}
\def\Stab{\mathrm{Stab}}

\def\log{\mathrm{log}\,}
\def\num{\mathrm{num}}

\def\Perf{\mathrm{Perf}}

\def\End{\mathrm{End}}

\def\gl2{\widetilde{\mathrm{GL}}^+_{2}(\mathbb R)}

\def\cA{\mathcal A}
\def\cB{\mathcal B}
\def\cC{\mathcal C}
\def\cD{\mathcal D}
\def\cE{\mathcal E}
\def\cF{\mathcal F}
\def\cG{\mathcal G}

\def\cL{\mathcal L}
\def\cM{\mathcal M}

\def\cO{\mathcal O}
\def\cP{\mathcal P}
\def\cQ{\mathcal Q}

\def\cT{\mathcal T}

\def\cZ{\mathcal Z}
\def\bZ{\mathbb Z}
\def\bP{\mathbb P}
\def\bQ{\mathbb Q}
\def\bR{\mathbb R}

\def\bC{\mathbb C}
\def\bT{\mathbb T}
\def\k{\mathbf{k}}

\def\y{\mathbf{y}}
\def\z{\mathbf{z}}

\def\i{\mathbf{i}}

\def\s{\mathbf{s}}
\def\n{\mathbf{n}}

\def\T{\mathbf{T}}

\theoremstyle{plain}
\newtheorem{theorem}{Theorem}[section]
\newtheorem{corollary}[theorem]{Corollary}
\newtheorem{proposition}[theorem]{Proposition}
\theoremstyle{definition}
\newtheorem{definition}[theorem]{Definition}

\newtheorem{example}[theorem]{Example}
\newtheorem{remark}[theorem]{Remark}

\newtheorem{lemma}[theorem]{Lemma}
\newtheorem{assumption}[theorem]{Assumption}
\newtheorem{conjecture}[theorem]{Conjecture}

\newtheorem{proposal}[theorem]{Proposal}
\usepackage{url}

\usetikzlibrary{matrix,positioning,decorations.markings,arrows,decorations.pathmorphing, 
backgrounds,fit,positioning,shapes.symbols,chains,shadings,fadings,calc}
\tikzset{->-/.style={decoration={ markings, mark=at position #1 with
{\arrow{>}}},postaction={decorate}}}
\tikzset{-<-/.style={decoration={ markings, mark=at position #1 with
{\arrow{<}}},postaction={decorate}}}

\usepackage[backend=biber,maxbibnames=99,maxalphanames=99,style=ext-alphabetic,date=year,url=false,isbn=false,doi=false]{biblatex}
\DeclareFieldFormat[inbook, article, misc, book, incollection]{citetitle}{#1}
\DeclareFieldFormat[inbook, article, misc, book, incollection]{title}{#1}
\DeclareFieldFormat{eprint:arxiv}{%
  arXiv\addcolon\space
  \ifhyperref
    {\href{https://arxiv.org/abs/#1}{\nolinkurl{#1}}}
    {\nolinkurl{#1}}%
}
\DeclareUnicodeCharacter{0411}{\CYRB}

%% file: sample.bib
@article {MR2373143,
    AUTHOR = {Bridgeland, Tom},
     TITLE = {Stability conditions on triangulated categories},
   JOURNAL = {Ann. of Math. (2)},
  FJOURNAL = {Annals of Mathematics. Second Series},
    VOLUME = {166},
      YEAR = {2007},
    NUMBER = {2},
     PAGES = {317--345},
      ISSN = {0003-486X,1939-8980},
   MRCLASS = {14F05 (18E30)},
  MRNUMBER = {2373143},
MRREVIEWER = {Leovigildo\ M.\ Alonso Tarrio},
       DOI = {10.4007/annals.2007.166.317},
       URL = {https://tlink.lib.tsinghua.edu.cn:443/https/443/org/doi/yitlink/10.4007/annals.2007.166.317},
}

@misc{kontsevich2008stabilitystructuresmotivicdonaldsonthomas,
      TITLE={Stability structures, motivic Donaldson-Thomas invariants and cluster transformations}, 
      AUTHOR={Maxim Kontsevich and Yan Soibelman},
      YEAR={2008},
      EPRINT={0811.2435},
      archivePrefix={arXiv},
      primaryClass={math.AG},
      URL={https://arxiv.org/abs/0811.2435}, 
}

@incollection {Bridgeland2006SpacesOS,
    AUTHOR = {Bridgeland, Tom},
     TITLE = {Spaces of stability conditions},
 BOOKTITLE = {Algebraic geometry---{S}eattle 2005. {P}art 1},
    SERIES = {Proc. Sympos. Pure Math.},
    VOLUME = {80, Part 1},
     PAGES = {1--21},
 PUBLISHER = {Amer. Math. Soc., Providence, RI},
      YEAR = {2009},
      ISBN = {978-0-8218-4702-2},
   MRCLASS = {14F05 (14J33 18E30 53D37 81T40 81T45)},
  MRNUMBER = {2483930},
MRREVIEWER = {Yunfeng\ Jiang},
       DOI = {10.1090/pspum/080.1/2483930},
       URL = {https://doi.org/10.1090/pspum/080.1/2483930},
}

@inproceedings {Bondal2002DerivedCO,
    AUTHOR = {Bondal, Alexei and Orlov, Dmitri},
     TITLE = {Derived categories of coherent sheaves},
 BOOKTITLE = {Proceedings of the {I}nternational {C}ongress of
              {M}athematicians, {V}ol. {II} ({B}eijing, 2002)},
     PAGES = {47--56},
 PUBLISHER = {Higher Ed. Press, Beijing},
      YEAR = {2002},
      ISBN = {7-04-008690-5},
   MRCLASS = {18E30 (14A22 14F05)},
  MRNUMBER = {1957019},
MRREVIEWER = {Bal\'azs\ Szendr\H oi},
}

@misc{Bondal1995SemiorthogonalDF,
  TITLE={Semiorthogonal decompositions for algebraic varieties.},
  Author={Alexei Bondal and Dmitri Orlov},
  EPRINT={alg-geom/9506012},
  archivePrefix={arXiv},
  primaryClass={math.AG},
  YEAR={1995},
  URL={https://arxiv.org/abs/alg-geom/9506012}
}

@article{Collins2009GluingSC,
    AUTHOR = {Collins, John and Polishchuk, Alexander},
     TITLE = {Gluing stability conditions},
   JOURNAL = {Adv. Theor. Math. Phys.},
  FJOURNAL = {Advances in Theoretical and Mathematical Physics},
    VOLUME = {14},
      YEAR = {2010},
    NUMBER = {2},
     PAGES = {563--607},
      ISSN = {1095-0761,1095-0753},
   MRCLASS = {14F05},
  MRNUMBER = {2721656},
MRREVIEWER = {Daniele\ Faenzi},
       DOI = {10.4310/atmp.2010.v14.n2.a6},
       URL = {https://doi.org/10.4310/atmp.2010.v14.n2.a6},
}

@misc{halpernleistner2024noncommutativeminimalmodelprogram,
      title={The noncommutative minimal model program}, 
      author={Daniel Halpern-Leistner},
      year={2024},
      eprint={2301.13168},
      archivePrefix={arXiv},
      primaryClass={math.AG},
      url={https://arxiv.org/abs/2301.13168}, 
}

@misc{halpernleistner2024stabilityconditionssemiorthogonaldecompositions,
      title={Stability conditions and semiorthogonal decompositions I: quasi-convergence}, 
      author={Daniel Halpern-Leistner and Jeffrey Jiang and Antonios-Alexandros Robotis},
      year={2024},
      eprint={2401.00600},
      archivePrefix={arXiv},
      primaryClass={math.AG},
      url={https://arxiv.org/abs/2401.00600}, 
}

@article {dell2024fusionequivariantstabilityconditionsmorita,
    AUTHOR = {Dell, Hannah and Heng, Edmund and Licata, Anthony M.},
     TITLE = {Fusion-equivariant stability conditions and {M}orita duality},
   JOURNAL = {Math. Z.},
  FJOURNAL = {Mathematische Zeitschrift},
    VOLUME = {311},
      YEAR = {2025},
    NUMBER = {3},
     PAGES = {Paper No. 48, 31},
      ISSN = {0025-5874,1432-1823},
   MRCLASS = {18M20 (14F08 18G80 32Q26)},
  MRNUMBER = {4953680},
       DOI = {10.1007/s00209-025-03838-z},
       URL = {https://doi.org/10.1007/s00209-025-03838-z},
}

@article{Beckmann2020OnED,
    AUTHOR = {Beckmann, Thorsten and Oberdieck, Georg},
     TITLE = {On equivariant derived categories},
   JOURNAL = {Eur. J. Math.},
  FJOURNAL = {European Journal of Mathematics},
    VOLUME = {9},
      YEAR = {2023},
    NUMBER = {2},
     PAGES = {Paper No. 36, 39},
      ISSN = {2199-675X,2199-6768},
   MRCLASS = {14F08 (13D03 14H60 18G80)},
  MRNUMBER = {4589277},
MRREVIEWER = {Stefan\ Schr\"oer},
       DOI = {10.1007/s40879-023-00635-y},
       URL = {https://doi.org/10.1007/s40879-023-00635-y},
}

@misc{elagin2015equivarianttriangulatedcategories,
      title={On equivariant triangulated categories}, 
      author={Alexey Elagin},
      year={2015},
      eprint={1403.7027},
      archivePrefix={arXiv},
      primaryClass={math.AG},
      url={https://arxiv.org/abs/1403.7027}, 
}

@article {krug2024endomorphismalgebrasequivariantexceptional,
    AUTHOR = {Krug, Andreas and Nikolov, Erik},
     TITLE = {Endomorphism algebras of equivariant exceptional collections},
   JOURNAL = {Algebr. Represent. Theory},
  FJOURNAL = {Algebras and Representation Theory},
    VOLUME = {28},
      YEAR = {2025},
    NUMBER = {1},
     PAGES = {193--210},
      ISSN = {1386-923X,1572-9079},
   MRCLASS = {18G80 (14F08 14L30 16E35 16G10 16G20)},
  MRNUMBER = {4875693},
MRREVIEWER = {Xiao-Wu\ Chen},
       DOI = {10.1007/s10468-025-10313-0},
       URL = {https://doi.org/10.1007/s10468-025-10313-0},
}

@article{MacriSukhenduPaolo07,
    AUTHOR = {Macr\`i, Emanuele and Mehrotra, Sukhendu and Stellari, Paolo},
     TITLE = {Inducing stability conditions},
   JOURNAL = {J. Algebraic Geom.},
  FJOURNAL = {Journal of Algebraic Geometry},
    VOLUME = {18},
      YEAR = {2009},
    NUMBER = {4},
     PAGES = {605--649},
      ISSN = {1056-3911,1534-7486},
   MRCLASS = {14F05 (14L30 18E30)},
  MRNUMBER = {2524593},
MRREVIEWER = {Amnon\ Neeman},
       DOI = {10.1090/S1056-3911-09-00524-4},
       URL = {https://doi.org/10.1090/S1056-3911-09-00524-4},
}

@book {Anderson_Fulton_2023,
    AUTHOR = {Anderson, David and Fulton, William},
     TITLE = {Equivariant cohomology in algebraic geometry},
    SERIES = {Cambridge Studies in Advanced Mathematics},
    VOLUME = {210},
 PUBLISHER = {Cambridge University Press, Cambridge},
      YEAR = {2024},
     PAGES = {xv+446},
      ISBN = {978-1-00-934998-7},
   MRCLASS = {14L30 (05E14 14-02 14F43 14Mxx 20G05 55N91)},
  MRNUMBER = {4655919},
MRREVIEWER = {Michael\ Orin\ Joyce},
}

@article{Givental1996EquivariantG,
    AUTHOR = {Givental, Alexander B.},
     TITLE = {Equivariant {G}romov-{W}itten invariants},
   JOURNAL = {Internat. Math. Res. Notices},
  FJOURNAL = {International Mathematics Research Notices},
      YEAR = {1996},
    NUMBER = {13},
     PAGES = {613--663},
      ISSN = {1073-7928,1687-0247},
   MRCLASS = {14D07 (14D05 14J32 14N10 32G20)},
  MRNUMBER = {1408320},
MRREVIEWER = {Claire\ Voisin},
       DOI = {10.1155/S1073792896000414},
       URL = {https://doi.org/10.1155/S1073792896000414},
}

@book {Cox1999MirrorSA,
    AUTHOR = {Cox, David A. and Katz, Sheldon},
     TITLE = {Mirror symmetry and algebraic geometry},
    SERIES = {Mathematical Surveys and Monographs},
    VOLUME = {68},
 PUBLISHER = {American Mathematical Society, Providence, RI},
      YEAR = {1999},
     PAGES = {xxii+469},
      ISBN = {0-8218-1059-6},
   MRCLASS = {14J32 (14-02 14M25 14N10 14N35 32G81 32J81 32Q25)},
  MRNUMBER = {1677117},
MRREVIEWER = {Andreas\ Gathmann},
       DOI = {10.1090/surv/068},
       URL = {https://doi.org/10.1090/surv/068},
}

@incollection {1994hep.th....5035K,
    AUTHOR = {Kontsevich, Maxim},
     TITLE = {Enumeration of rational curves via torus actions},
 BOOKTITLE = {The moduli space of curves ({T}exel {I}sland, 1994)},
    SERIES = {Progr. Math.},
    VOLUME = {129},
     PAGES = {335--368},
 PUBLISHER = {Birkh\"auser Boston, Boston, MA},
      YEAR = {1995},
      ISBN = {0-8176-3784-2},
   MRCLASS = {14N10 (14D22 14L30)},
  MRNUMBER = {1363062},
MRREVIEWER = {Anatoly\ Libgober},
       DOI = {10.1007/978-1-4612-4264-2\_12},
       URL = {https://doi.org/10.1007/978-1-4612-4264-2_12},
}

@article{Behrend1995StacksOS,
    AUTHOR = {Behrend, K. and Manin, Yu.},
     TITLE = {Stacks of stable maps and {G}romov-{W}itten invariants},
   JOURNAL = {Duke Math. J.},
  FJOURNAL = {Duke Mathematical Journal},
    VOLUME = {85},
      YEAR = {1996},
    NUMBER = {1},
     PAGES = {1--60},
      ISSN = {0012-7094,1547-7398},
   MRCLASS = {14D20 (14C25 14D22)},
  MRNUMBER = {1412436},
MRREVIEWER = {Barbara\ Fantechi},
       DOI = {10.1215/S0012-7094-96-08501-4},
       URL = {https://doi.org/10.1215/S0012-7094-96-08501-4},
}

@article{liu2017equivariant,
    AUTHOR = {Liu, Chiu-Chu Melissa and Sheshmani, Artan},
     TITLE = {Equivariant {G}romov-{W}itten invariants of algebraic {GKM}
              manifolds},
   JOURNAL = {SIGMA Symmetry Integrability Geom. Methods Appl.},
  FJOURNAL = {SIGMA. Symmetry, Integrability and Geometry. Methods and
              Applications},
    VOLUME = {13},
      YEAR = {2017},
     PAGES = {Paper No. 048, 21},
      ISSN = {1815-0659},
   MRCLASS = {14N35 (14D20 14H10)},
  MRNUMBER = {3667222},
MRREVIEWER = {Jie\ Zhou},
       DOI = {10.3842/SIGMA.2017.048},
       URL = {https://doi.org/10.3842/SIGMA.2017.048},
}

@misc{zuliani2024semiorthogonaldecompositionsprojectivespaces,
      title={Semiorthogonal decompositions of projective spaces from small quantum cohomology}, 
      author={Vanja Zuliani},
      year={2024},
      eprint={2406.17616},
      archivePrefix={arXiv},
      primaryClass={math.AG},
      url={https://arxiv.org/abs/2406.17616}, 
}

@incollection {Cotti2019EquivariantQD,
    AUTHOR = {Cotti, Giordano and Varchenko, Alexander},
     TITLE = {Equivariant quantum differential equation and q{KZ} equations
              for a projective space: {S}tokes bases as exceptional
              collections, {S}tokes matrices as {G}ram matrices, and {B}-theorem},
 BOOKTITLE = {Integrability, quantization, and geometry. {I}. {I}ntegrable
              systems},
    SERIES = {Proc. Sympos. Pure Math.},
    VOLUME = {103.1},
     PAGES = {101--170},
 PUBLISHER = {Amer. Math. Soc., Providence, RI},
      YEAR = {2021},
      ISBN = {978-1-4704-5591-0},
   MRCLASS = {14N35 (17B80 34M35 34M40)},
  MRNUMBER = {4285678},
MRREVIEWER = {Ahmed\ Lesfari},
       DOI = {10.1090/pspum/103.1/01833},
       URL = {https://doi.org/10.1090/pspum/103.1/01833},
}

@article{10.1215/00127094-3476593,
    AUTHOR = {Galkin, Sergey and Golyshev, Vasily and Iritani, Hiroshi},
     TITLE = {Gamma classes and quantum cohomology of {F}ano manifolds:
              Gamma conjectures},
   JOURNAL = {Duke Math. J.},
  FJOURNAL = {Duke Mathematical Journal},
    VOLUME = {165},
      YEAR = {2016},
    NUMBER = {11},
     PAGES = {2005--2077},
      ISSN = {0012-7094,1547-7398},
   MRCLASS = {53D37 (11G42 14J33 14J45 14M15 14N35)},
  MRNUMBER = {3536989},
MRREVIEWER = {Jian\ Xun\ Hu},
       DOI = {10.1215/00127094-3476593},
       URL = {https://doi.org/10.1215/00127094-3476593},
}

@article{IRITANI20091016,
    AUTHOR = {Iritani, Hiroshi},
     TITLE = {An integral structure in quantum cohomology and mirror
              symmetry for toric orbifolds},
   JOURNAL = {Adv. Math.},
  FJOURNAL = {Advances in Mathematics},
    VOLUME = {222},
      YEAR = {2009},
    NUMBER = {3},
     PAGES = {1016--1079},
      ISSN = {0001-8708,1090-2082},
   MRCLASS = {53D37 (14J33 14N35 53D45)},
  MRNUMBER = {2553377},
MRREVIEWER = {Hsian-Hua\ Tseng},
       DOI = {10.1016/j.aim.2009.05.016},
       URL = {https://doi.org/10.1016/j.aim.2009.05.016},
}

@article{Polishchuk06,
    AUTHOR = {Polishchuk, Alexander},
     TITLE = {Constant families of {$t$}-structures on derived categories of
              coherent sheaves},
   JOURNAL = {Mosc. Math. J.},
  FJOURNAL = {Moscow Mathematical Journal},
    VOLUME = {7},
      YEAR = {2007},
    NUMBER = {1},
     PAGES = {109--134, 167},
      ISSN = {1609-3321,1609-4514},
   MRCLASS = {14F05 (18E30)},
  MRNUMBER = {2324559},
MRREVIEWER = {Andrei\ D.\ Halanay},
       DOI = {10.17323/1609-4514-2007-7-1-109-134},
       URL = {https://doi.org/10.17323/1609-4514-2007-7-1-109-134},
}

@article{Keller:2011nql,
    AUTHOR = {Keller, Bernhard},
     TITLE = {Deformed {C}alabi-{Y}au completions},
      NOTE = {With an appendix by Michel Van den Bergh},
   JOURNAL = {J. Reine Angew. Math.},
  FJOURNAL = {Journal f\"ur die Reine und Angewandte Mathematik. [Crelle's
              Journal]},
    VOLUME = {654},
      YEAR = {2011},
     PAGES = {125--180},
      ISSN = {0075-4102,1435-5345},
   MRCLASS = {18E30 (13F60 16E35 16E45 18E35 18G10)},
  MRNUMBER = {2795754},
MRREVIEWER = {Gregoire\ Dupont},
       DOI = {10.1515/CRELLE.2011.031},
       URL = {https://doi.org/10.1515/CRELLE.2011.031},
}

@article{Ikeda_Qiu_2023,
    AUTHOR = {Ikeda, Akishi and Qiu, Yu},
     TITLE = {{$q$}-stability conditions on {C}alabi-{Y}au-{$\Bbb X$}
              categories},
   JOURNAL = {Compos. Math.},
  FJOURNAL = {Compositio Mathematica},
    VOLUME = {159},
      YEAR = {2023},
    NUMBER = {7},
     PAGES = {1347--1386},
      ISSN = {0010-437X,1570-5846},
   MRCLASS = {18G80 (14F08 32Q26)},
  MRNUMBER = {4599210},
       DOI = {10.1112/s0010437x23007194},
       URL = {https://doi.org/10.1112/s0010437x23007194}
}

@article{GKM97,
    AUTHOR = {Goresky, Mark and Kottwitz, Robert and MacPherson, Robert},
     TITLE = {Equivariant cohomology, {K}oszul duality, and the localization
              theorem},
   JOURNAL = {Invent. Math.},
  FJOURNAL = {Inventiones Mathematicae},
    VOLUME = {131},
      YEAR = {1998},
    NUMBER = {1},
     PAGES = {25--83},
      ISSN = {0020-9910,1432-1297},
   MRCLASS = {55N91 (14F25 14F32 16E99 18G10 55N33)},
  MRNUMBER = {1489894},
MRREVIEWER = {Roy\ Joshua},
       DOI = {10.1007/s002220050197},
       URL = {https://doi.org/10.1007/s002220050197},
}

@article{Bridgelandk3surface,
    AUTHOR = {Bridgeland, Tom},
     TITLE = {Stability conditions on {$K3$} surfaces},
   JOURNAL = {Duke Math. J.},
  FJOURNAL = {Duke Mathematical Journal},
    VOLUME = {141},
      YEAR = {2008},
    NUMBER = {2},
     PAGES = {241--291},
      ISSN = {0012-7094,1547-7398},
   MRCLASS = {14F05 (14J28 18E30)},
  MRNUMBER = {2376815},
MRREVIEWER = {Andrei\ D.\ Halanay},
       DOI = {10.1215/S0012-7094-08-14122-5},
       URL = {https://doi.org/10.1215/S0012-7094-08-14122-5},

}

@book{manin1999frobenius,
    AUTHOR = {Manin, Yuri I.},
     TITLE = {Frobenius manifolds, quantum cohomology, and moduli spaces},
    SERIES = {American Mathematical Society Colloquium Publications},
    VOLUME = {47},
 PUBLISHER = {American Mathematical Society, Providence, RI},
      YEAR = {1999},
     PAGES = {xiv+303},
      ISBN = {0-8218-1917-8},
   MRCLASS = {53D45 (14H10 14N35 18D50 32G34)},
  MRNUMBER = {1702284},
MRREVIEWER = {Alexandre\ I.\ Kabanov},
       DOI = {10.1090/coll/047},
       URL = {https://doi.org/10.1090/coll/047},
}

@article{DOOrlov_1993,
    AUTHOR = {Orlov, Dmitri O.},
     TITLE = {Projective bundles, monoidal transformations, and derived
              categories of coherent sheaves},
   JOURNAL = {Izv. Ross. Akad. Nauk Ser. Mat.},
  FJOURNAL = {Izvestiya Rossiiskoi Akademii Nauk. Seriya Matematicheskaya},
    VOLUME = {56},
      YEAR = {1992},
    NUMBER = {4},
     PAGES = {852--862},
      ISSN = {1607-0046,2587-5906},
   MRCLASS = {14F05 (18E30 18F20)},
  MRNUMBER = {1208153},
MRREVIEWER = {Krzysztof\ Jaczewski},
       DOI = {10.1070/IM1993v041n01ABEH002182},
       URL = {https://doi.org/10.1070/IM1993v041n01ABEH002182},
}

@incollection {Dolgachev2009,
    AUTHOR = {Dolgachev, Igor V.},
     TITLE = {Finite subgroups of the plane {C}remona group},
 BOOKTITLE = {Algebraic geometry in {E}ast {A}sia---{S}eoul 2008},
    SERIES = {Adv. Stud. Pure Math.},
    VOLUME = {60},
     PAGES = {1--49},
 PUBLISHER = {Math. Soc. Japan, Tokyo},
      YEAR = {2010},
      ISBN = {978-4-931469-63-1},
   MRCLASS = {14E07},
  MRNUMBER = {2732091},
MRREVIEWER = {Concettina\ Galati},
       DOI = {10.2969/aspm/06010001},
       URL = {https://doi.org/10.2969/aspm/06010001},
}

@misc{karube2024noncommutativemmpblowupsurfaces,
      title={The noncommutative MMP for blowup surfaces}, 
      author={Tomohiro Karube},
      year={2024},
      eprint={2410.18446},
      archivePrefix={arXiv},
      primaryClass={math.AG},
      url={https://arxiv.org/abs/2410.18446},
}

@book{wasow2018asymptotic,
    AUTHOR = {Wasow, Wolfgang},
     TITLE = {Asymptotic expansions for ordinary differential equations},
      NOTE = {Reprint of the 1976 edition},
 PUBLISHER = {Dover Publications, Inc., New York},
      YEAR = {1987},
     PAGES = {x+374},
      ISBN = {0-486-65456-7},
   MRCLASS = {34-02 (34E05)},
  MRNUMBER = {919406},
}

@article{BP1994,
    AUTHOR = {Bondal, A. I. and Polishchuk, A. E.},
     TITLE = {Homological properties of associative algebras: the method of
              helices},
   JOURNAL = {Izv. Ross. Akad. Nauk Ser. Mat.},
  FJOURNAL = {Izvestiya Rossiiskoi Akademii Nauk. Seriya Matematicheskaya},
    VOLUME = {57},
      YEAR = {1993},
    NUMBER = {2},
     PAGES = {3--50},
      ISSN = {1607-0046,2587-5906},
   MRCLASS = {16E99 (14F05 14J45 18F20 18G50)},
  MRNUMBER = {1230966},
       DOI = {10.1070/IM1994v042n02ABEH001536},
       URL = {https://doi.org/10.1070/IM1994v042n02ABEH001536},
}

@misc{KO2018,
      title={Nonexistence of semiorthogonal decompositions and sections of the canonical bundle},
      author={Kotaro Kawatani and Shinnosuke Okawa},
      year={2018},
      eprint={1508.00682},
      archivePrefix={arXiv},
      primaryClass={math.AG},
      url={https://arxiv.org/abs/1508.00682},
}

@ARTICLE{2023arXiv231002917Q,
    AUTHOR = {Qiu, Yu and Zhang, Xiaoting},
     TITLE = {Fusion-stable structures on triangulated categories},
   JOURNAL = {Selecta Math. (N.S.)},
  FJOURNAL = {Selecta Mathematica. New Series},
    VOLUME = {31},
      YEAR = {2025},
    NUMBER = {3},
     PAGES = {Paper No. 50, 29},
      ISSN = {1022-1824,1420-9020},
   MRCLASS = {18M20 (16G20 18G80 20F36)},
  MRNUMBER = {4912527},
       DOI = {10.1007/s00029-025-01037-6},
       URL = {https://doi.org/10.1007/s00029-025-01037-6},
}

@article{Li2017,
    AUTHOR = {Li, Chunyi},
     TITLE = {The space of stability conditions on the projective plane},
   JOURNAL = {Selecta Math. (N.S.)},
  FJOURNAL = {Selecta Mathematica. New Series},
    VOLUME = {23},
      YEAR = {2017},
    NUMBER = {4},
     PAGES = {2927--2945},
      ISSN = {1022-1824,1420-9020},
   MRCLASS = {14F05 (18E30)},
  MRNUMBER = {3703470},
MRREVIEWER = {Martijn\ Kool},
       DOI = {10.1007/s00029-017-0352-4},
       URL = {https://doi.org/10.1007/s00029-017-0352-4},
}

@article{ChenRuancohomology,
    AUTHOR = {Chen, Weimin and Ruan, Yongbin},
     TITLE = {A new cohomology theory of orbifold},
   JOURNAL = {Comm. Math. Phys.},
  FJOURNAL = {Communications in Mathematical Physics},
    VOLUME = {248},
      YEAR = {2004},
    NUMBER = {1},
     PAGES = {1--31},
      ISSN = {0010-3616,1432-0916},
   MRCLASS = {57R19 (53D45)},
  MRNUMBER = {2104605},
MRREVIEWER = {Paolo\ Lisca},
       DOI = {10.1007/s00220-004-1089-4},
       URL = {https://doi.org/10.1007/s00220-004-1089-4},
}

@article{chen2002orbifold,
    AUTHOR = {Chen, Weimin and Ruan, Yongbin},
     TITLE = {Orbifold {G}romov-{W}itten theory},
 BOOKTITLE = {Orbifolds in mathematics and physics ({M}adison, {WI}, 2001)},
    SERIES = {Contemp. Math.},
    VOLUME = {310},
     PAGES = {25--85},
 PUBLISHER = {Amer. Math. Soc., Providence, RI},
      YEAR = {2002},
      ISBN = {0-8218-2990-4},
   MRCLASS = {53D45 (14N35)},
  MRNUMBER = {1950941},
MRREVIEWER = {Ignasi\ Mundet-Riera},
       DOI = {10.1090/conm/310/05398},
       URL = {https://doi.org/10.1090/conm/310/05398},
}

@article{Blanc2006LinearisationOF,
    AUTHOR = {Blanc, J\'er\'emy},
     TITLE = {Linearisation of finite abelian subgroups of the {C}remona
              group of the plane},
   JOURNAL = {Groups Geom. Dyn.},
  FJOURNAL = {Groups, Geometry, and Dynamics},
    VOLUME = {3},
      YEAR = {2009},
    NUMBER = {2},
     PAGES = {215--266},
      ISSN = {1661-7207,1661-7215},
   MRCLASS = {14E07 (14E05 14L30)},
  MRNUMBER = {2486798},
MRREVIEWER = {Julie\ D\'eserti},
       DOI = {10.4171/GGD/55},
       URL = {https://doi.org/10.4171/GGD/55},
}

@article{Tsygankov_2011,
    AUTHOR = {Tsygankov, Vladimir I.},
     TITLE = {Equations of {$G$}-minimal conic bundles},
   JOURNAL = {Mat. Sb.},
  FJOURNAL = {Matematicheski\u i\ Sbornik},
    VOLUME = {202},
      YEAR = {2011},
    NUMBER = {11},
     PAGES = {103--160},
      ISSN = {0368-8666,2305-2783},
   MRCLASS = {14E25 (14E07)},
  MRNUMBER = {2907201},
MRREVIEWER = {Alexandr\ V.\ Pukhlikov},
       DOI = {10.1070/SM2011v202n11ABEH004204},
       URL = {https://doi.org/10.1070/SM2011v202n11ABEH004204},
}

@article{Prokhorov_2021,
    AUTHOR = {Prokhorov, Yuri G.},
     TITLE = {Equivariant minimal model program},
   JOURNAL = {Uspekhi Mat. Nauk},
  FJOURNAL = {Uspekhi Matematicheskikh Nauk},
    VOLUME = {76},
      YEAR = {2021},
    NUMBER = {3(459)},
     PAGES = {93--182},
      ISSN = {0042-1316,2305-2872},
   MRCLASS = {14E30},
  MRNUMBER = {4265398},
MRREVIEWER = {Igor\ V.\ Nikolaev},
       DOI = {10.4213/rm9990},
       URL = {https://doi.org/10.4213/rm9990},

}

@book {Kollár_Mori_1998,
    AUTHOR = {Koll\'ar, J\'anos and Mori, Shigefumi},
     TITLE = {Birational geometry of algebraic varieties},
    SERIES = {Cambridge Tracts in Mathematics},
    VOLUME = {134},
 PUBLISHER = {Cambridge University Press, Cambridge},
      YEAR = {1998},
     PAGES = {viii+254},
      ISBN = {0-521-63277-3},
   MRCLASS = {14E30},
  MRNUMBER = {1658959},
MRREVIEWER = {Mark\ Gross},
       DOI = {10.1017/CBO9780511662560},
       URL = {https://doi.org/10.1017/CBO9780511662560},
}

@misc{ESS2025,
      title={Atomic decompositions for derived categories of G-surfaces},
      author={Alexey Elagin and Julia Schneider and Evgeny Shinder},
      year={2025},
      eprint={2512.05064},
      archivePrefix={arXiv},
      primaryClass={math.AG},
      url={https://arxiv.org/abs/2512.05064},
}

@article{Toda1412,
    AUTHOR = {Toda, Yukinobu},
     TITLE = {Gepner type stability conditions on graded matrix
              factorizations},
   JOURNAL = {Algebr. Geom.},
  FJOURNAL = {Algebraic Geometry},
    VOLUME = {1},
      YEAR = {2014},
    NUMBER = {5},
     PAGES = {613--665},
      ISSN = {2313-1691,2214-2584},
   MRCLASS = {14F05 (18E30)},
  MRNUMBER = {3296807},
MRREVIEWER = {Pawel\ Sosna},
       DOI = {10.14231/AG-2014-026},
       URL = {https://doi.org/10.14231/AG-2014-026},
}

@article{10.1093/imrn/rnv125,
    AUTHOR = {Toda, Yukinobu},
     TITLE = {Gepner type stability condition via {O}rlov/{K}uznetsov
              equivalence},
   JOURNAL = {Int. Math. Res. Not. IMRN},
  FJOURNAL = {International Mathematics Research Notices. IMRN},
      YEAR = {2016},
    NUMBER = {1},
     PAGES = {24--82},
      ISSN = {1073-7928,1687-0247},
   MRCLASS = {14F05},
  MRNUMBER = {3514058},
MRREVIEWER = {Mihnea\ Popa},
       DOI = {10.1093/imrn/rnv125},
       URL = {https://doi.org/10.1093/imrn/rnv125},
}

@misc{Li2026,
      title={A Remark on Stability Conditions on Smooth Projective Varieties},
      author={Chunyi Li},
      year={2026},
      eprint={2601.22994},
      archivePrefix={arXiv},
      primaryClass={math.AG},
      url={https://arxiv.org/abs/2601.22994},
}

@misc{KRZ2026,
      title={Toward the noncommutative minimal model program for Fano varieties},
      author={Tomohiro Karube and Antonios-Alexandros Robotis and Vanja Zuliani},
      year={2026},
      eprint={2601.20739},
      archivePrefix={arXiv},
      primaryClass={math.AG},
      url={https://arxiv.org/abs/2601.20739},
}
